\documentclass[a4paper,12pt,english]{amsart}
\usepackage{amsmath,latexsym,amssymb,amsfonts}
\usepackage{amscd,amssymb,epsfig}
\usepackage[arrow, matrix, curve]{xy}
\usepackage{hyperref}
\usepackage{mathdots}
\usepackage{enumerate}
\usepackage{tikz}
\usetikzlibrary{decorations.markings}
\tikzset{->-/.style={decoration={markings,mark=at position #1 with {\arrow{>}}},postaction={decorate}}}

\numberwithin{equation}{section}

\newtheorem{theorem}{Theorem}[section]
\newtheorem{lemma}[theorem]{Lemma}
\newtheorem{corollary}[theorem]{Corollary}
\newtheorem{proposition}[theorem]{Proposition}
\newtheorem{algorithm}[theorem]{Algorithm}

\theoremstyle{definition}
\newtheorem{definition}[theorem]{Definition}

\newtheorem{example}[theorem]{Example}
\newtheorem{remark}[theorem]{Remark}

\renewcommand\labelenumi{(\roman{enumi})}
\renewcommand\theenumi\labelenumi

\renewcommand{\epsilon}{\varepsilon}

\newcommand{\A}{\mathcal{A}}
\newcommand{\B}{\mathcal{B}}
\newcommand{\C}{\mathbb{C}}
\newcommand{\D}{\mathcal{D}}
\newcommand{\E}{\mathbb{E}}
\newcommand{\F}{\mathcal{F}}
\renewcommand{\H}{\mathcal{H}}
\newcommand{\cL}{\mathcal{L}}
\newcommand{\M}{\mathcal{M}}
\newcommand{\N}{\mathbb{N}}

\newcommand{\R}{\mathbb{R}}

\newcommand{\bS}{\mathbb{S}}
\newcommand{\V}{\mathcal{V}}
\newcommand{\W}{\mathcal{W}}
\newcommand{\X}{\mathbb{X}}

\newcommand{\FF}{\mathbb{F}}

\newcommand{\id}{\operatorname{id}}

\newcommand{\ev}{\operatorname{ev}}
\newcommand{\Ev}{\operatorname{Ev}}
\newcommand{\tr}{\operatorname{tr}}
\newcommand{\ran}{\operatorname{ran}}
\newcommand{\vN}{\operatorname{vN}}

\newcommand{\Tr}{\operatorname{Tr}}
\newcommand{\im}{\operatorname{im}}

\newcommand{\dom}{\operatorname{dom}}

\newcommand{\rank}{\operatorname{rank}}

\renewcommand{\1}{\mathbf{1}}
\newcommand{\0}{\mathbf{0}}

\newcommand{\ootimes}{\mathbin{\overline{\otimes}}}

\newcommand{\full}{\mathrm{full}}

\renewcommand{\star}{\bigstar}

\makeatletter
\def\moverlay{\mathpalette\mov@rlay}
\def\mov@rlay#1#2{\leavevmode\vtop{%
\baselineskip\z@skip \lineskiplimit-\maxdimen
\ialign{\hfil$#1##$\hfil\cr#2\crcr}}}
\makeatother

\def\plangle{\moverlay{(\cr<}}
\def\prangle{\moverlay{)\cr>}}

\makeindex

\title[The free field]{The free field: zero divisors, Atiyah property and realizations via unbounded operators}

\author[T. Mai]{Tobias Mai}
\address{Saarland University, Faculty of Mathematics, D-66123 Saarbr\"ucken, Germany}
\email{mai@math.uni-sb.de}

\author[R. Speicher]{Roland Speicher}
\address{Saarland University, Faculty of Mathematics, D-66123 Saarbr\"ucken, Germany}
\email{speicher@math.uni-sb.de}

\author[S. Yin]{Sheng Yin}
\address{Saarland University, Faculty of Mathematics, D-66123 Saarbr\"ucken, Germany}
\email{yin@math.uni-sb.de}

\date{\today}

\thanks{We thank Ken Dykema for discussions about the relation between the Atiyah property and the rational closure, as well as providing us with the example in \ref{ex:Dykema_Pascoe}.\\
This work has been supported by the ERC Advanced Grant NCDFP 339760 and by the SFB-TRR 195, Project I.12.}

\keywords{free field, non-commutative rational functions, 
Atiyah property, free probability, free entropy dimension, zero divisor}

\subjclass[2000]{46L54 (12E15)}

\begin{document}

\begin{abstract}

We consider noncommutative rational functions as well as matrices in polynomials in noncommuting variables in two settings: in an algebraic context the variables are formal variables, and their rational functions generate the ''free field''; in an analytic context the variables are given by operators from a finite von Neumann algebra and the question of rational functions is treated within the affiliated unbounded operators.  Our main result shows that for a ''good'' class of operators -- namely those for which the free entropy dimension is maximal -- the analytic and the algebraic theory are isomorphic. This means in particular that any non-trivial rational function can be evaluated as an unbounded operator for any  such good tuple and that those operators don't have zero divisors. On the matrix side, this means that matrices of polynomials which are invertible in the free field are also invertible as matrices over unbounded operators when we plug in our good operator tuples. We also address the question how this is related to the strong Atiyah property. The above yields a quite complete picture for the question of zero divisors (or atoms in the corresponding distributions) for operator tuples with maximal free entropy dimension. We give also some partial results for the question of existence and regularity of a density of the distribution.

\end{abstract}

\maketitle

\tableofcontents

\section{Introduction}

In the last few years there has been quite some progress on understanding qualitative and quantitative properties of 
\begin{itemize}
\item
the asymptotic eigenvalue distribution of polynomials in tuples of random matrices, for big classes of random matrices and
\item
the distribution of polynomials in tuples of operators on infinite-dimensional Hilbert spaces, for big classes of operators
\end{itemize}

Those two arenas for looking on polynomials of several, in general non-commuting, variables, are closely related; namely, free probability theory has taught us that the limit of random matrix ensembles is, in many situations, given by operators in interesting operator algebras. Hence random matrices can tell us something about interesting operators (and their related $C^*$- and von Neumann algebras) and operator theory provides tools for dealing with asymptotic properties of random matrices.

In particular, in the context of free probability theory
one has isolated a precise notion for the big classes of random matrices and operators, alluded to above, in terms of the concept of free entropy dimension.

If we have operators $(X_1,\dots,X_n)$ on an infinite dimensional Hilbert space (living in a von Neumann algebra equipped with a trace, to be precise) then saying that they have maximal free entropy dimension, $\delta(X_1,\dots,X_n)=n$, means (these slogans will be made more precise later)
\begin{itemize}
\item
in the random matrix world:
that there are many matrix tuples which converge in distribution to this tuple $(X_1,\dots,X_n)$
\item
in the operator world: that we have a nice ''free calculus'' theory of non-commutative derivatives for the non-commutative polynomials in those operators.
\end{itemize}

Many random matrices and their limit operators fall into this category. The most basic example is given by independent Gaussian random matrices and their limit, given by free semicircular operators. One should note, however, that freeness between the limit operators is not necessary for having maximal free entropy dimension.

What has been shown before in \cite{MSW17,CS16} for such operator tuples is that non-commutative polynomials in them have no zero divisors in the generated von Neumann algebra, which means that the distribution of such polynomials possesses no atoms.

In \cite{HMS18} we also started to extend the frame of investigations from non-commutative polynomials $\C \langle x_1,\dots,x_n\rangle$ to the much bigger class of non-commutative rational functions. The latter -- which is a skew field, usually called the ''free field'' and denoted by $\C \plangle x_1,\dots,x_n\prangle$ -- is given by all meaningful rational expressions in the non-commuting formal variables $x_1,\dots,x_n$, where two such expressions are being identified if they can be transformed into each other by algebraic manipulations. The existence of such an object is quite non-trivial, and was established by Amitsur \cite{Ami66} and extensively studied by Cohn \cite{Coh95,Coh06}.
We want now to apply such a non-commutative rational function $r(x_1,\dots,x_n)$ to our tuples of operators $(X_1,\dots,X_n)$. Apriori, one has to face the problem that for a given tuple of operators there are always polynomials in those operators which have zero in their spectrum, and hence are not invertible as bounded operators. Hence, in investigations in \cite{HMS18,Yin18} we restricted to tuples which are in the domain of our rational function. However, there we tried to stay within bounded operators. If we allow, however, also unbounded operators, then we gain a lot of freedom and as we will show we can actually evaluate any non-commutative rational function on any tuple of operators with maximal free entropy dimension. This relies crucially on the fact that we consider only operators in finite von Neumann algebras, i.e., those equipped with a faithful trace, and there one has a very nice theory of affiliated unbounded operators. Whereas the set of all unbounded operators on 
an infinite dimensional Hilbert space has many pathological properties, in the finite setting the affiliated unbounded operators form actually a $*$-algebra and it is quite easy to control their invertibility. As in the finite-dimensional setting the only obstruction to invertibility is the existence of a kernel; in an algebraic formulation, such an operator is invertible (as an unbounded operator) if it has no zero divisor. Hence the question whether we can evaluate rational functions in tuples of operators goes hand in hand with the question whether we can exclude zero divisors for such evaluations.

The idea in \cite{MSW17} for proving the absence of zero divisors for polynomials was to reduce the degree of the polynomial by taking derivatives. This approach does not seem to work in the rational case (as the derivative of a rational function does not necessarily decrease the complexity of the considered function). However, rational functions are via the linearization idea related to matrices over polynomials and we will achieve the proof of the absence of zero divisors for elements in $\C \plangle x_1,\dots,x_n\prangle$
by showing the absence of zero divisors for special matrices over $\C \langle x_1,\dots,x_n\rangle$. 
It will turn out that we can precisely characterize this class of matrices as those which are invertible over the free field. 

\begin{theorem}\label{the:main}
Consider operators $X_1,\dots,X_n$ in a finite von Neumann algebra $\M$, with maximal free entropy dimension, i.e., $\delta^*(X_1,\dots,X_n)=n$. Then we have the following. 
\begin{itemize}
\item
For any matrix $P\in M_N(\C \langle x_1,\dots,x_n\rangle)$ over non-commutative polynomials which is full (i.e., cannot be written as a product of strictly rectangular matrices over the polynomials; which is the same as saying that the matrix is invertible in the matrices over the free field) we have that $P(X_1,\dots,X_n)$, the evaluation of this matrix in our operators, is invertible as a matrix over the unbounded operators affiliated to $\M$.
\item
For any non-commutative rational function $0\not= r\in \C \plangle x_1,\dots,x_n\prangle$ in the free field we have that its evaluation $r(X_1,\dots,X_n)$ in the operators $X_1,\dots,X_n$ is well-defined as an unbounded operator affiliated to $\M$, is different from zero and has no zero divisor; hence it can also be inverted as an unbounded operator.
\end{itemize}

\end{theorem}

This gives us a complete understanding for (the absence of) atoms for matrices in polynomials and for rational functions in tuples of operators with maximal free entropy dimension. One expects in this generality also the absence of a singular part, and hence the existence of a density, for the distribution. We are able to show this for linear matrices, but the same question for rational functions has to remain open, as the linearization does not seem to give a direct transfer of results on such questions between rational functions and matrices.

Note that we can rephrase our result about zero divisors for matrices $P\in M_N(\C \langle x_1,\dots,x_n\rangle)$
also in the form that
the point spectrum of the analytic operator $P(X_1,\dots,X_n)$ is the same as the point spectrum of the abstract algebraic element $P(x_1,\dots,x_n)$. 

Our methods are a blend of algebraic and analytic methods; the algebraic part depends quite substantially on the fundamental work of Cohn on the free field (see, e.g., \cite{Coh95,Coh06,CR94,CR99}), and also the more recent work of Garg, Gurvits, Oliveira, and Wigderson \cite{GGOW16} on this subject, which highlights the characterization of full matrices in terms of shrunk subspaces; whereas the analytic part uses heavily ideas and recent progress \cite{Dab10,KV12,SS15,CS16,MSW17} from "free analysis" and its non-commutative derivative calculus. 

Note that our results give in particular that the free field can be realized in a canonical way in the algebra of unbounded operators as the field generated by any tuple of self-adjoint operators with maximal free entropy dimensions. This is in contrast to the usual approaches to non-commutative rational functions in  free analysis where the free field is in some sense generated by tuples of matrices of all sizes. This fits with the idea that operators with maximal free entropy dimension are limits of sufficiently many matrices and should thus represent typical properties of matrices of all sizes. 

The only known realization of the free field in terms of unbounded operators is the result of Linnell \cite{Lin93}, who showed, with very different methods, that the operators in the regular representation of the free group have this property. His result is not a direct corollary of ours, as the generators of the free group are not selfadjoint but unitary operators. However, it it feasible that our approach can also be adapted to dealing with unitary operators. 

Our investigations are of course related to the zero divisor conjecture, the Atiyah conjecture, or $l^2$-Betti numbers; for work in this context see, for example, \cite{GLSZ00,DLMSY03,PTh11}. Whereas those are about properties of elements in the group algebra (where sometimes, in particular in the work of Linnell, this is embedded in the unbouded operators affiliated to the corresponding group von Neumann algebra), we look here at a situation where our tuple of operators is not necessarily coming from a group, but we ask the same type of questions as for the group case. In particular, the ''strong Atiyah property'' was defined in this general setting, and proved for cases where the operators are free, by Shlyakhtenko and Skoufranis in \cite{SS15}. In this general setting, a main object of interest is to identify operators which behave like the generators of a free group; of course, not just on the bare algebraic level, in the sense that the operators have no algebraic relations, but in a more refined way. In particular, one of the main questions in free probability theory is which operators generate as a von Neumann algebra the free group factor $L(\FF_n)$, i.e., the von Neumann algebra generated by the free group on $n$ generators. There is some hope that the free entropy dimension might be some invariant for such questions; in particular, whether $\delta^*(X_1,\dots,X_n)=n$ means that the von Neumann algebra generated by $X_1,\dots,X_n$ is isomorphic to $L(\FF_n)$. At the moment such questions are out of reach. What we provide here is that we relax substantially our goal; instead of taking an analytic closure we push the algebraic closure to its limit by also allowing inverses and replace the von Neumann algebra by the rational (or division) closure generated by our operators in the algebra of affiliated unbounded operators; for this setting we have then a positive answer to the analogue of the above question: for all operators with maximal free entropy dimension their generated rational closure is isomorphic to the free field. Whether the free entropy dimension is also an invariant for the rational closure when the free entropy dimension is not maximal is an interesting question for further investigations.

The paper is organized as follows.
In Section 2, we recall the basic concepts and results around the inner rank and full matrices over polynomials in noncommuting variables. In particular, we provide: in Proposition \ref{prop:shrunk_subspace} the characterization of full linear matrices in terms of shrunk subspaces; and its consequence, Corollary \ref{cor-of-shrunk}, which will play a main role in our later analysis in Section 4.
For the convenience of the reader, we provide in the appendix proofs of all the relevant results of Section 2; this is based on the work of Cohn, but streamlined to our more special setting. In Section 3, we provide the ''free analysis'' tools which are used in the sequel; in particular, we prove matricial extensions of results of Voiculescu and Dabrowski on noncommutative derivatives. The main result is Theorem \ref{thm:Fisher-growth}, which provides the crucial reduction argument in our later analysis. Section 4 is addressing the question when linear matrices of operators with maximal free entropy dimension are invertible, and yields, in the form of Theorem \ref{thm:main-1}, the first part of Theorem \ref{the:main} for linear matrices (the general case of matrices in polynomials will follow later from Theorem \ref{thm:Atiyah-1}). In Section 5 we switch from matrices over polynomials to rational functions. Section 5 provides the definition and basic facts about noncommutative rational functions and the rational closure of an algebra; in particular, the linearization idea is presented, which makes the connection between noncommutative rational functions and matrices over noncommutative polynomials. In Section 6, we translate then our main result about the invertibility of full matrices to a similar statement about the invertibility of noncommutative rational functions, thus giving the second part of Theorem \ref{the:main} in Corollary \ref{cor:full_entropy_dimension_implies_Atiyah}. It is also shown, in Theorems \ref{thm:Atiyah-1} and \ref{thm:Atiyah-2}, how this relates to the strong Atiyah property. In Sections 7 and 8
we give some preliminary results on the absence of a singular part and on regularity properties of the distribution of linear matrices. This is based on ideas from \cite{CS16, AjEK18,AEK18}.

\section{Inner rank of matrices}

In this section, we introduce the inner rank for matrices over noncommutative algebras as an analogous notion of the rank for matrices over numbers or commutative algebras. First, we consider the general case, and let $\A$ be a unital (not necessarily commutative) algebra.

\begin{definition}\label{def:inner-rank_full}
For any non-zero $A\in M_{m.n}(\A)$, the \emph{inner rank} of $A$ is defined as the least positive integer $r$ such that there are matrices $P\in M_{m,r}(\A)$, $Q\in M_{r,n}(\A)$ satisfying $A=PQ$. We denote this number by $\rho(A)$, and any such factorization with $r=\rho(A)$ is called a \emph{rank factorization}. In particular, if $\rho(A)=\min\{m,n\}$, namely, if there is no such factorization with $r<\min\{m,n\}$, then $A$ is called \emph{full}. Additionally, if $A$ is the zero matrix, we define $\rho(A)=0$.
\end{definition}

As indicated by its name, this notion is trying to capture the properties of the usual rank of matrices in linear algebra; it's not difficult to check that it becomes the usual rank of matrices if $\A=\C$. Moreover, similar to the fact that a matrix of rank $r$ in $M_{m,n}(\C)$ always has a non-singular $r\times r$ block, we have the following theorem.

\begin{theorem}
\label{thm:full minor}(See \cite[Theorem 5.4.9]{Coh06}) Suppose that the set of all square full matrices over $\A$ is closed under products and diagonal sums (see Definition \ref{def:diagonal sum}). Then for any $A\in M_{m,n}(\A)$, there exists a square block of $A$ which is a full matrix over $\A$ of dimension $\rho(A)$. Moreover, $\rho(A)$ is the maximal dimension for such blocks.
\end{theorem}

See Appendix \ref{sec:inner rank} for a detailed proof based on Cohn's book \cite{Coh06}. There is another important property of inner rank that we need to highlight here; for that purpose, we need the notion of stably finite algebras.

\begin{definition}
$\A$ is called \emph{stably finite} (or \emph{weakly finite}) if for any $n\in\mathbb{N}$, and all $A,B\in M_{n}(\A)$ the equation $AB=\1_{n}$ implies that also $BA=\1_{n}$ holds.
\end{definition}

\begin{proposition}
\label{prop:invertible minor}(See \cite[Proposition 5.4.6]{Coh06}) Suppose that $\A$ is stably finite. Let $A\in M_{m+n}(\A)$ be of the form
\[A=\begin{pmatrix}B & C\\D & E\end{pmatrix},\]
where $B\in M_{m}(\A)$, $C\in M_{m,n}(\A)$, $D\in M_{n,m}(\A)$ and $E\in M_{n}(\A)$. If $B$ is invertible, then $\rho(A)\geqslant m$, with equality if and only if $E=DB^{-1}C$.
\end{proposition}

In the remaining part of this section, we set $\A=\C\left\langle x_{1},\dots,x_{d}\right\rangle $, the algebra of noncommutative polynomials in (formal) non-commuting variables $x_{1},\dots,x_{d}$. Then the requirements in Theorem \ref{thm:full minor} can be verified (see Appendix \ref{sec:full minor}), so as a corollary we have the following proposition which is Lemma 4 of Section 4 in \cite{CR94}; a proof can be found at the end of Appendix \ref{sec:full minor}. This proposition is needed for the induction step in the proof for Theorem \ref{thm:main-1}.

\begin{proposition}
\label{prop:full minor}Let $A\in M_{n}(\C\left\langle x_{1},\dots,x_{d}\right\rangle )$ be given in the form $A=(a,A')$, where $a$ is the first column of $A$ and $A'$ is the remaining block. Assume that $A$ is full, then there is a full $(n-1)\times(n-1)$ block in $A'$.
\end{proposition}

Now, consider a matrix of form
\[A=\begin{pmatrix}P & \0\\Q & R\end{pmatrix}\in M_{m,n}(\C\left\langle x_{1},\dots,x_{d}\right\rangle),\]
which has a zero block of size $r\times s$ and blocks $P$, $Q$, $R$ of size $r\times (n-s)$, $(m-r)\times (n-s)$, $(m-r)\times s$, respectively. Then we have the factorization
\[A=\begin{pmatrix}P & \0\\Q & R\end{pmatrix}=\begin{pmatrix}P & \0\\\0 & \1_{m-r}\end{pmatrix}\begin{pmatrix}\1_{n-s} & \0\\Q & R\end{pmatrix}.\]
So $A$ has been expressed as a product of an $m\times(m+n-r-s)$ matrix and an $(m+n-r-s)\times n$ matrix; this allows us to conclude that $\rho(A)\leqslant m+n-r-s$. Therefore, if the size of the zero block of $A$ satisfies $r+s>\max\{m,n\}$, then we have $\rho(A)<\min\{m,n\}$, which means that $A$ is not full. Such matrices are called hollow matrices.

\begin{definition}
\label{def:hollow}A matrix in $M_{m,n}(\C\left\langle x_{1},\dots,x_{d}\right\rangle )$ is called \emph{hollow} if it has an $r\times s$ block of zeros with $r+s>\max\{m,n\}$.
\end{definition}

In general, a non-full $A\in M_{m,n}(\C\left\langle x_{1},\dots,x_{d}\right\rangle)$ may not have any zero blocks. However, we will be mostly interested in special matrices for which we can say more.

\begin{definition}
A matrix $A\in M_{n}(\C\left\langle x_{1},\dots,x_{d}\right\rangle )$ is called \emph{linear} if it can be writen in the form $A=A_{0}+A_{1}x_{1}+\cdots+A_{d}x_{d}$, where $A_{0},A_{1}\dots,A_{d}$ are $n\times n$ matrices over $\C$. Note that we allow also a constant term in a general linear matrix. And we call the non-constant part $A-A_{0}=A_{1}x_{1}+\cdots+A_{d}x_{d}$ the \emph{homogeneous part of $A$}.
\end{definition}

For linear matrices we have the following theorem for the relation between non-full and hollow;  for a proof, see  Appendix \ref{sec:linear matrices}.

\begin{theorem}
\label{thm:hollow and full}(See \cite[Corollary 6.3.6]{Coh95}) Let $A\in M_{n}(\C\left\langle x_{1},\dots,x_{d}\right\rangle )$ be linear. If $A$ is not full, then there exist invertible matrices $U,V\in M_n(\C)$ such that $UAV$ is hollow.
\end{theorem}

\begin{definition}
Let a linear $A=A_{0}+A_{1}x_{1}+\cdots+A_{d}x_{d}\in M_{n}(\C\left\langle x_{1},\dots,x_{d}\right\rangle )$ be given. If there are subspaces $V,W$ of $\C^{n}$ with $\dim W<\dim V$ such that $A_{i}V\subseteq W$ for all $i=0,1,\dots,d$, then $V$ is called a \emph{shrunk subspace}. In this case we also say that $A$ has a shrunk subspace $V$.
\end{definition}

These shrunk subspaces can be used to describe the fullness of linear matrices. It seems that this notion and the following proposition appeared for the first time in \cite{GGOW16}.

\begin{proposition}\label{prop:shrunk_subspace}
A linear matrix $A\in M_{n}(\C\left\langle x_{1},\dots,x_{d}\right\rangle )$ is full if and only if it has no shrunk subspace.
\end{proposition}

\begin{proof}
It is clear that if $A$ has such a shrunk subspace $V$, then there are unitary matrices $P$ and $Q$ such that each $PA_{i}Q$ has a $\dim V\times(n-\dim W)$ block of zeros. So it follows from $\dim W<\dim V$ that $PAQ$ is a hollow matrix. Thus $PAQ$ is not full, which also implies $A$ is not full.

For the converse, if $A$ is not full, by Theorem \ref{thm:hollow and full}, there exist some invertible matrices $P$ and $Q$ over $\C$ such that $PAQ$ has a $r\times s$ block of zeros with $r+s>n$. So it is not difficult to see that $PAQ$ has a shrunk subspace $V$ and thus $A$ has a shrunk subspace $QV$, as asserted.
\end{proof}

From this we get the following corollary, which will be a main ingredient in the proof for Theorem \ref{thm:main-1}.

\begin{corollary}\label{cor-of-shrunk}
\label{cor:fullness_contraction}Let $A\in M_{n}(\C\left\langle x_{1},\dots,x_{d}\right\rangle )$ be a linear full matrix and $P,Q\in M_{n}(\C)$ be given such that $PAQ=0$. Then we have
\[\rank(P)+\rank(Q)\leqslant n.\]
\end{corollary}

\begin{proof}
Let $A=A_{0}+A_{1}x_{1}+\cdots+A_{d}x_{d}$. We define $V=\im Q$ and $W=\ker P$, then we have
\[A_{i}V\subseteq W,\ i=0,1,\dots,d\]
by $PAQ=0$. As $A$ is full, there is no shrunk subspace, and hence $\dim W\geqslant\dim V$, that is, $n-\rank(P)\geqslant\rank(Q)$.
\end{proof}

We finish by mentioning another interesting criterion for the fullness of linear matrices that was given in \cite{GGOW16}.

\begin{proposition}\label{prop:rank-decreasing}
Consider a linear matrix $A = A_{1}x_{1}+\cdots+A_{d}x_{d}$ in $M_{n}(\C\left\langle x_{1},\dots,x_{d}\right\rangle )$ with zero constant part. Then $A$ is full if and only if the associated \emph{quantum operator}
$$\cL:\ M_n(\C) \to M_n(\C), \qquad b \mapsto \sum^d_{i=1} A_i b A_i^\ast$$
is \emph{nowhere rank-decreasing}, i.e., there is no positive semidefinite matrix $b \in M_n(\C)$ for which $\rank(\cL(b)) < \rank(b)$ holds.
\end{proposition}

This connects fullness very nicely with concepts that are used, for instance, in \cite{AjEK18,AEK18}; we will say more about this in Section \ref{sec:regularity_linear_matrices} and Section \ref{sec:Hoelder_continuity}.

\section{Matricial differential calculus}

This section is devoted to ``free analysis'' that provides the analytic tools used in the sequel. Our main goal is Theorem \ref{thm:Fisher-growth}, by which we generalize the crucial ``reduction argument'' of \cite{MSW17} that was formulated in Proposition 3.9 therein to the case of square matrices of noncommutative polynomials.
The proof of Theorem \ref{thm:Fisher-growth} that will be given below, however, does not rely on the corresponding result in \cite{MSW17}; we rather have to repeat the arguments of \cite{MSW17}, which were built in particular on the work of Voiculescu \cite{Voi98} and Dabrowski \cite{Dab10} about the $L^2$-theory of operators induced by noncommutative derivatives, in our matricial setup.
Conceptually, the proof given below will follow the lines of \cite{MSW17}. The main difference compared to this preceding work is that here some matricial extension of the aforementioned results due to Voiculescu and Dabrowski are needed. We point out that especially the matricial extension of the amazing norm estimates that were obtained in \cite{Dab10} requires some care; while they are proven in our context almost in the same way as the corresponding scalar-valued results, they are not a direct consequence thereof.
We highlight that the matricial extension of the $L^2$-theory for free differential operators that is presented here fits into the much more general frame developed by Shlyakhtenko in \cite{S00}; thus, some of our results could alternatively be derived from \cite{S00}, but for the sake of a self-contained exposition, we prefer to give direct proofs that are adapted to our situation. Furthermore, we remark that in \cite{CDS14}, the $L^2$-theory for free differential operators was extended even to the setting of planar algebras.

\subsection{Matricial extension of derivations}

We begin by an introductory discussion around derivations and their matricial extensions in some general algebraic framework.

Let $\A$ be a unital complex algebra and let $\M$ be any $\A$-bimodule. Denote by $\cdot$ the left respectively right action of $\A$ on $\M$. Clearly, for any $N\in\N$, we have that $M_N(\M)$ is a $M_N(\A)$-bimodule with respect to the left respectively right action defined by
$$a^1 \cdot m \cdot a^2 = \bigg(\sum^N_{p,q=1} a^1_{k,p} \cdot m_{p,q} \cdot a^2_{q,l}\bigg)_{k,l=1}^N$$
for any given $a^1=(a^1_{kl})_{k,l=1}^N$, $a^2=(a^2_{kl})_{k,l=1}^N$ in $M_N(\A)$ and $m=(m_{kl})_{k,l=1}^N$ in $M_N(\M)$. Note that $M_N(\A)$ forms canonically a complex unital algebra by the usual matrix multiplication but performed with respect to the multiplication given on $\A$.

Now, let us consider an $\M$-valued derivation $\partial$ on $\A$, i.e., a linear mapping $\partial: \A \to \M$ that satisfies the \emph{Leibniz rule}
$$\partial(a_1 a_2) = \partial(a_1) \cdot a_2 + a_1 \cdot \partial(a_2) \qquad\text{for all $a_1,a_2\in \A$}.$$
We may introduce then its amplification $\partial^{(N)}$ to $M_N(\A)$, which is given by
$$\partial^{(N)}:\ M_N(\A) \to M_N(\M),\quad (a_{kl})_{k,l=1}^N \mapsto \big(\partial(a_{kl})\big)_{k,l=1}^N.$$
For later use, we agree here on the following notation: whenever $\V$ and $\W$ are vector spaces and $\phi: \V\to\W$ is any linear map between them, then we may introduce for $N\in\N$ the \emph{(matricial) amplification $\phi^{(N)}$} of $\phi$ by
$$\phi^{(N)}:\ M_N(\V) \to M_N(\W),\quad (v_{kl})_{k,l=1}^N \mapsto \big(\phi(v_{kl})\big)_{k,l=1}^N.$$

\begin{lemma}\label{lem:derivation_amplification}
The matricial amplification $\partial^{(N)}: M_N(\A) \to M_N(\M)$ of any $\M$-valued derivation $\partial: \A \to \M$ on $\A$ is an $M_N(\M)$-valued derivation on $M_N(\A)$.
\end{lemma}

\begin{proof}
Let  $a^1=(a^1_{kl})_{k,l=1}^N$ and $a^2=(a^2_{kl})_{k,l=1}^N$ in $M_N(\A)$ be given. Then
\begin{align*}
\partial^{(N)}\big(a^1 a^2\big) &= \partial^{(N)} \bigg(\sum^N_{p=1} a^1_{kp} a^2_{pl}\bigg)_{k,l=1}^N\\
                                &= \bigg(\sum^N_{p=1} \partial\big(a^1_{kp} a^2_{pl}\big)\bigg)_{k,l=1}^N\\
                                &= \bigg(\sum^N_{p=1} \Big[ \partial(a^1_{kp}) \cdot a^2_{pl} + a^1_{kp} \cdot \partial(a^2_{pl}) \Big] \bigg)_{k,l=1}^N\\
                                &= \bigg(\sum^N_{p=1} \partial(a^1_{kp})\bigg)_{k,p=1}^N \cdot a^2 + a^1 \cdot \bigg(\sum^N_{p=1} \partial (a^2_{pl})\bigg)_{p,l=1}^N\\
                                &= \partial^{(N)}(a^1) \cdot a^2 + a^1 \cdot \partial^{(N)} (a^2),
\end{align*}
which confirms that $\partial^{(N)}$ is an $M_N(\M)$-valued derivation on $M_N(\A)$, as asserted.
\end{proof}

We focus now on the particular case where the $\A$-bimodule $\M$ is given as $\A \otimes \A$, i.e., as the algebraic tensor product of $\A$ over $\C$ with itself. Note that $\A \otimes \A$ forms both
\begin{itemize}
 \item an algebra with respect to the multiplication that is defined by bilinear extension of $$(a_1 \otimes a_2) (b_1 \otimes b_2) = (a_1 b_1) \otimes (a_2 b_2);$$
 \item an $\A$-bimodule with respect to the left and right action of $\A$ that are defined by $$a_1 \cdot (b_1 \otimes b_2) \cdot a_2 = (a_1 b_1) \otimes (b_2 a_2).$$
\end{itemize}
Accordingly, $M_N(\A \otimes \A)$ can be seen both as a complex unital algebra and an $M_N(\A)$-bimodule.

On each $M_N(\A)$, there is a binary operation
$$\odot:\ M_N(\A) \times M_N(\A) \to M_N(\A \otimes \A)$$
that is given by
$$a^1 \odot a^2 = \bigg(\sum_{p=1}^N a^1_{kp} \otimes a^2_{pl}\bigg)_{k,l=1}^N$$
for any $a^1=(a^1_{kl})_{k,k=1}^N$ and $a^2 = (a^2_{kl})_{k,l=1}^N$ in $M_N(\A)$. We have then the relation
$$M_N(\A) \odot M_N(\A) = M_N(\A \otimes \A),$$
since $(e^{k,p} a_1) \odot (e^{p,l} a_2) = e^{k,l} (a_1 \otimes a_2)$ for all $a_1,a_2\in \A$ and all $1\leq k,p,l \leq N$; here, we denote by $e^{k,l}$ the matrix unit in $M_N(\C)$, i.e., the matrix whose entries are all $0$ except the $(k,l)$-entry which is $1$.

The binary operation $\odot$ is compatible with the algebra structure and the $\A$-bimodule structure of $M_N(\A\otimes\A)$ in the sense that
$$a^1 \cdot m \cdot a^2 = (a^1 \odot \1_N) m (\1_N \odot a^2)$$
for all $a^1,a^2 \in M_N(\A)$ and all $m\in M_N(\A \otimes \A)$, where $\1_N$ denotes the identity element in $M_N(\C) \subset M_N(\A)$.

Finally, we point out that for any other $\A$-bimodule $\M$, there is an operation $\sharp: (\A \otimes \A) \times \M \to \M$ which is defined by bilinear extension of $(a_1 \otimes a_2) \sharp m = a_1 \cdot m \cdot a_2$ for $a_1,a_2\in\A$ and $m\in\M$; this extends naturally to the matricial setup as an operation
$$\sharp:\ M_N(\A \otimes \A) \times \M \to M_N(\M)$$
that is defined by $u \sharp m := (u_{kl} \sharp m)_{k,l=1}^N$ for any $u=(u_{kl})_{k,l=1}^N$ in $M_N(\A \otimes \A)$ and $m\in\M$.

\subsection{Noncommutative derivatives}

We focus now on the case of noncommutative derivatives, which underly free analysis as the suitable noncommutative counterpart of classical derivatives.

\subsubsection{Matrices of noncommutative polynomials}

Let $\C\langle x_1,\dots,x_n\rangle$ be the complex unital algebra of noncommutative polynomials in $n$ (formal) non-commuting variables $x_1,\dots,x_n$.
Note that $\C\langle x_1,\dots,x_n\rangle$ becomes a $\ast$-algebra with respect to the involution $\ast$ that is determined by the condition that $1^\ast = 1$ and $x_i^\ast = x_i$ for $i=1,\dots,n$.

If $X=(X_1,\dots,X_n)$ is any $n$-tuple consisting of elements $X_1,\dots,X_n$ in any complex unital algebra $\A$, then we may define the \emph{evaluation map $\ev_X$} as the unique algebra homomorphism
$$\ev_X:\ \C\langle x_1,\dots,x_n\rangle \to \A$$
that is unital and satisfies $\ev_X(x_i)=X_i$ for $i=1,\dots,n$; its image, which is the unital subalgebra of $\A$ that is generated by $X_1,\dots,X_n$, will be denoted by $\C\langle X_1,\dots,X_n\rangle$.

Fix $N\in\N$. For any given $P\in M_N(\C\langle x_1,\dots,x_n\rangle)$, we will mostly write $P(X_1,\dots,X_n)$ or $P(X)$ instead of $\ev_X^{(N)}(P)$.
Correspondingly, for each $Q\in  M_N(\C\langle x_1,\dots,x_n\rangle \otimes \C\langle x_1,\dots,x_n\rangle)$, we abbreviate $(\ev_X \otimes \ev_X)^{(N)}(Q)$ by $Q(X_1,\dots,X_n)$ or just $Q(X)$.

Note that since $\C\langle x_1,\dots,x_n\rangle$ is a $\ast$-algebra, also $M_N(\C\langle x_1,\dots,x_n\rangle)$ forms naturally a $\ast$-algebra; thus, if $\A$ is a unital complex $\ast$-algebra, then $P(X)^\ast = P^\ast(X)$ for each $P\in M_n(\C\langle x_1,\dots,x_n\rangle)$ and all $n$-tuples $X=(X_1,\dots,X_n)$ that consist of selfadjoint elements $X_1,\dots,X_n$ in $\A$.

\subsubsection{Noncommutative derivatives and their matricial extension}

On the algebra $\C\langle x_1,\dots,x_n\rangle$, we may introduce the so-called \emph{non-commutative derivatives} $\partial_1,\dots,\partial_n$ as the unique derivations
$$\partial_j:\ \C\langle x_1,\dots,x_n\rangle \to \C\langle x_1,\dots,x_n\rangle \otimes \C\langle x_1,\dots,x_n\rangle,\qquad j=1,\dots,n,$$
with values in the $\C\langle x_1,\dots,x_n\rangle$-bimodule $\C\langle x_1,\dots,x_n\rangle \otimes \C\langle x_1,\dots,x_n\rangle$ that satisfy the condition $\partial_j x_i = \delta_{i,j} 1 \otimes 1$ for $i,j=1,\dots,n$.

For any $N\in\N$, the noncommutative derivatives extend according to Lemma \ref{lem:derivation_amplification} to $M_N(\C\langle x_1,\dots,x_n\rangle \otimes \C\langle x_1,\dots,x_n\rangle)$-valued derivations
$$\partial_j^{(N)}:\ M_N(\C\langle x_1,\dots,x_n\rangle) \to M_N(\C\langle x_1,\dots,x_n\rangle \otimes \C\langle x_1,\dots,x_n\rangle),\qquad j=1,\dots,n.$$

In the next subsection, we study unbounded linear operators that are induced by those amplifications of the noncommutative derivatives.

\subsection{A matricial extension of the $L^2$-theory for free differential operators}

Let $(\M,\tau)$ be a tracial $W^\ast$-probability space (i.e., a von Neumann algebra $\M$ that is endowed with a faithful normal tracial state $\tau: M \to \C$) and consider $n$ selfadjoint noncommutative random variables $X_1,\dots,X_n\in \M$.
Throughout the following, we will denote in such cases by $\M_0\subseteq \M$ the von Neumann subalgebra that is generated by $X_1,\dots,X_n$; in order to simplify the notation, the restriction of $\tau$ to $\M_0$ will be denoted again by $\tau$.

\subsubsection{Conjugate systems and non-microstates free Fisher information}

In \cite{Voi98}, Voiculescu associated to the tuple $(X_1,\dots,X_n)$ the so-called \emph{non-microstates free Fisher information $\Phi^\ast(X_1,\dots,X_n)$}; note that, while he assumed for technical reasons in addition that $X_1,\dots,X_n$ do not satisfy any non-trivial algebraic relation over $\C$, it was shown in \cite{MSW17} that this constraint is not needed as an a priori assumption on $(X_1,\dots,X_n)$ but is nonetheless enforced a posteriori by some general arguments. We call $(\xi_1,\dots,\xi_n) \in L^2(\M_0,\tau)^n$ a \emph{conjugate system for $(X_1,\dots,X_n)$}, if the \emph{conjugate relation}
$$\tau\big(\xi_j P(X_1,\dots,X_n)\big) = (\tau \ootimes \tau)\big((\partial_j P)(X_1,\dots,X_n)\big)$$
holds for each $j=1,\dots,n$ and for all noncommutative polynomials $P\in\C\langle x_1,\dots,x_n\rangle$, where $\tau \ootimes\tau$ denotes the faithful normal tracial state that is induced by $\tau$ on the von Neumann algebra tensor product $\M \ootimes \M$. The conjugate relation implies that such a conjugate system, in case of its existence, is automatically unique; thus, one can define
$$\Phi^\ast(X_1,\dots,X_n) := \sum^n_{j=1} \|\xi_j\|_2^2$$
if a conjugate system $(\xi_1,\dots,\xi_n)$ for $(X_1,\dots,X_n)$ exists and $\Phi^\ast(X_1,\dots,X_n) := \infty$ if there is no conjugate system for $(X_1,\dots,X_n)$.

\subsubsection{Free differential operators}

Suppose now that $\Phi^\ast(X_1,\dots,X_n) < \infty$ holds and let $(\xi_1,\dots,\xi_n)$ be the conjugate system for $(X_1,\dots,X_n)$. It was shown in \cite{MSW17} that $\ev_X: \C\langle x_1,\dots,x_n\rangle \to \C\langle X_1,\dots,X_n\rangle$ constitutes under this hypothesis an isomorphism, so that the noncommutative derivatives induce unbounded linear operators
$$\partial_j:\ L^2(\M_0,\tau) \supseteq D(\partial_j) \to L^2(\M_0 \ootimes \M_0, \tau \ootimes \tau)$$
with domain $D(\partial_j) := \C\langle X_1,\dots,X_n\rangle$. Since $\partial_j$ is densely defined, we may consider the adjoint operators
$$\partial_j^\ast:\ L^2(\M_0 \ootimes \M_0, \tau \ootimes \tau) \supseteq D(\partial_j^\ast) \to L^2(\M_0,\tau)$$
and we conclude from the conjugate relations that $1\otimes 1 \in D(\partial_j^\ast)$ with $\partial_j^\ast(1\otimes 1) = \xi_j$.

In a similar way, we may treat their matricial amplifications. Let us fix $N\in\N$. We consider then the $W^\ast$-probability spaces
$$\big(M_N(\M), \tr_N \circ \tau^{(N)}\big) \qquad \text{and} \qquad \big(M_N(\M \ootimes \M), \tr_N \circ (\tau \ootimes \tau)^{(N)}\big).$$
Here, $\tr_N: M_N(\C) \to \C$ stands for the usual normalized trace on $M_N(\C)$, $\tau^{(N)}: M_N(\M) \to M_N(\C)$ for the amplification of $\tau$ to $M_N(\M)$, and $(\tau\ootimes\tau)^{(N)}: M_N(\M \ootimes \M) \to M_N(\C)$ for the amplification of $\tau \ootimes\tau$.
The norms in the induced $L^2$-spaces will both be denoted by $\|\cdot\|_2$ as their meaning will always be clear from the context.

The matricial amplifications of the noncommutative derivatives induce unbounded linear operators
$$\partial_j^{(N)}:\ L^2(M_N(\M_0),\tr_N\circ\tau^{(N)}) \supseteq D(\partial_j^{(N)}) \to L^2(M_N(\M_0 \ootimes \M_0), \tr_N\circ(\tau\ootimes\tau)^{(N)})$$
with domain $D(\partial_j^{(N)}) = M_N(\C\langle X_1,\dots,X_n\rangle)$; the domain is dense in $L^2(M_N(\M_0),\tr_N\circ\tau^{(N)})$ and forms furthermore a $\ast$-subalgebra of $M_N(\M_0)$.

If restricted to its domain, each of the unbounded linear operator $\partial_j$ gives a $\C\langle X_1,\dots,X_n\rangle \otimes \C\langle X_1,\dots,X_n\rangle$-valued derivation on $\C\langle X_1,\dots,X_n\rangle$. Thus, Lemma \ref{lem:derivation_amplification} says that the matricial amplification $\partial_j^{(N)}$ restricts to an $M_N(\C\langle X_1,\dots,X_n\rangle \otimes \C\langle X_1,\dots,X_n\rangle)$-valued derivation on $M_N(\C\langle X_1,\dots,X_n\rangle)$.

These matricial amplifications enjoy properties very similar to the scalar-valued versions. In particular, they are \emph{real} in the sense that
\begin{equation}\label{eq:real_amplification}
\partial_j^{(N)}(P^\ast) = \big(\partial_j^{(N)} P \big)^\dagger
\end{equation}
holds for all $P\in D(\partial_j^{(N)})$, where the involution $\dagger$ on $M_N(\M_0 \ootimes \M_0)$ is defined by 
$$Q^\dagger = (Q^\dagger_{lk})_{k,l=1}^N \qquad \text{for any $Q = (Q_{kl})_{k,l=1}^N \in M_N(\M_0 \ootimes \M_0)$}$$
as the natural extension of the involution $\dagger$ on $\M_0 \ootimes \M_0$ that is determined by $(P_1 \otimes P_2)^\dagger = P_2^\ast \otimes P_1^\ast$ for arbitrary $P_1,P_2\in \M_0$.
The validity of \eqref{eq:real_amplification} follows from the corresponding statement in the scalar-valued setting; indeed, if $P=(P_{kl})_{k,l=1}^N$ in $D(\partial_j^{(N)})$ is given, we may easily check that
$$\partial_j^{(N)}(P^\ast) = \big(\partial_j (P_{lk}^\ast) \big)_{k,l=1}^N = \big( (\partial_j P_{lk})^\dagger \big)_{k,l=1}^N =  \big(\partial_j^{(N)} P \big)^\dagger.$$

\subsubsection{Voiculescu's formulas}

The assumption $\Phi^\ast(X_1,\dots,X_n) < \infty$ guarantees for each $j=1,\dots,n$ that $1 \otimes 1 \in D(\partial_j^\ast)$ with $\partial^\ast_j(1\otimes 1) = \xi_j$ and moreover that
$$\C\langle X_1,\dots,X_n\rangle \otimes \C\langle X_1,\dots,X_n\rangle \subseteq D(\partial^\ast_j).$$
Here, we prove that an analogous statement holds for their matricial amplifications; note that those unbounded linear operators are densely defined, so that their adjoints exist.

\begin{lemma}\label{lem:adjoints_amplifications}
For each $j=1,\dots,n$, we have that
$$M_N(\C\langle X_1,\dots,X_n\rangle \otimes \C\langle X_1,\dots,X_n\rangle) \subseteq D\big((\partial^{(N)}_j)^\ast\big).$$
If any $Q =(Q_{kl})_{k,l=1}^N \in M_N(\C\langle X_1,\dots,X_n\rangle \otimes \C\langle X_1,\dots,X_n\rangle)$ is given, then we have more precisely
$$(\partial^{(N)}_j)^\ast Q = \big(\partial_j^\ast Q_{kl}\big)_{k,l=1}^N$$
and in particular
$$(\partial_j^{(N)})^\ast (\1_N \odot \1_N) = \1_N \xi_j.$$
\end{lemma}

\begin{proof}
Let any $Q =(Q_{kl})_{k,l=1}^N$ in $M_N(\C\langle X_1,\dots,X_n\rangle \otimes \C\langle X_1,\dots,X_n\rangle)$ be given. Then, since each $Q_{kl}$ belongs to $D(\partial^\ast_j)$, we may deduce that
$$\langle \partial^{(N)}_j P, Q \rangle = \sum^N_{k,l=1} \langle \partial_j P_{kl}, Q_{kl} \rangle = \sum^N_{k,l=1} \langle P_{kl}, \partial_j^\ast Q_{kl}\rangle = \langle P, \big(\partial_j^\ast Q_{kl}\big)_{k,l=1}^N\rangle$$
for all $P=(P_{kl})_{k,l=1}^N$ in $M_N(\C\langle X_1,\dots,X_N\rangle)$, which shows $Q \in D((\partial^{(N)}_j)^\ast)$ and the asserted formula for $(\partial^{(N)}_j)^\ast Q$.
If we apply the latter observations to $\1_N \odot \1_N$, we conclude that it belongs to $D((\partial^{(N)}_j)^\ast)$ and $(\partial_j^{(N)})^\ast (\1_N \odot \1_N) = (\delta_{k,l} \partial_j^\ast(1\otimes 1) \big)_{k,l=1}^N = \1_N \xi_j$.
\end{proof}

Voiculescu \cite{Voi98} also derived a formula for $\partial_j^\ast$ on $\C\langle X_1,\dots,X_n\rangle \otimes \C\langle X_1,\dots,X_n\rangle$. In fact, he showed that
$$(\partial_j)^\ast Q = m_{\xi_j}(Q) - m_1 (\id \otimes \tau \otimes \id) (\partial_j \otimes \id + \id \otimes \partial_j)(Q)$$
for all $Q \in \C\langle X_1,\dots,X_n\rangle \otimes \C\langle X_1,\dots,X_n\rangle$, where $m_\eta$ for any element $\eta \in L^2(\M_0,\tau)$ stands for the linear mapping $m_\eta: \M_0 \otimes \M_0 \to L^2(\M_0,\tau)$ that is given by $m_\eta(a_1 \otimes a_2) = a_1 \eta a_2$. Due to the previous Lemma \ref{lem:adjoints_amplifications}, the above formula readily passes to the matricial setting. Indeed, we have
\begin{equation}\label{eq:Voiculescu-formula_amplification}
(\partial^{(N)}_j)^\ast Q = m_{\xi_j}^{(N)}(Q) - m_1^{(N)} (\id \otimes \tau \otimes \id)^{(N)}(\partial_j \otimes \id + \id \otimes \partial_j)^{(N)}(Q)
\end{equation}
for each $Q \in M_N(\C\langle X_1,\dots,X_n\rangle \otimes \C\langle X_1,\dots,X_n\rangle)$. In the sequel, some special instance of that formula will become important, which we thus present in the next proposition.

\begin{proposition}\label{prop:Voiculescu-formula_amplification}
Let $P^1,P^2 \in M_N(\C\langle X_1,\dots,X_n\rangle)$ be given. Then
$$(\partial_j^{(N)})^\ast \big(P^1 \odot P^2\big) = P^1(\1_N \xi_j) P^2 - (\id\otimes\tau)^{(N)}\big(\partial^{(N)}_j P^1\big) P^2 - P^1 (\tau\otimes\id)^{(N)}\big(\partial^{(N)}_j P^2\big).$$
\end{proposition}

\begin{proof}
Consider arbitrary $P^1=(P^1_{kl})_{k,l=1}^N$ and $P^2=(P^2_{kl})_{k,l=1}^N$ in $M_N(\C\langle X_1,\dots,X_n\rangle)$. Then $m_{\xi_j}^{(N)}(P^1 \odot P^2) = P^1 (\1_N \xi_j) P^2$ and
\begin{align*}
\lefteqn{m_1^{(N)} (\id \otimes \tau \otimes \id)^{(N)}(\partial_j \otimes \id + \id \otimes \partial_j)^{(N)}(P^1 \odot P^2)}\\
&= m_1^{(N)} (\id \otimes \tau \otimes \id)^{(N)} (\partial_j \otimes \id + \id \otimes \partial_j)^{(N)} \bigg( \sum^N_{p=1} P^1_{kp} \otimes P^2_{pl} \bigg)_{k,l=1}^N\\
&= m_1^{(N)} (\id \otimes \tau \otimes \id)^{(N)} \bigg( \sum^N_{p=1} \big[ \partial_j P^1_{kp} \otimes P^2_{pl} + P^1_{kp} \otimes \partial_j P^2_{pl} \big] \bigg)_{k,l=1}^N\\
&= m_1^{(N)} \bigg( \sum^N_{p=1} \big[ (\id\otimes\tau)(\partial_j P^1_{kp}) \otimes P^2_{pl} + P^1_{kp} \otimes (\tau\otimes\id)(\partial_j P^2_{pl}) \big] \bigg)_{k,l=1}^N\\
&= \bigg(\sum^N_{p=1} \big[ (\id\otimes\tau)(\partial_j P^1_{kp}) P^2_{pl} + P^1_{kp} (\tau\otimes\id)(\partial_j P^2_{pl}) \big]\bigg)_{k,l=1}^N\\
&= \big((\id\otimes\tau) \circ \partial_j\big)^{(N)} (P^1) P^2 + P^1 \big((\tau\otimes\id)\circ \partial_j \big)^{(N)} (P^2)\\
&= (\id\otimes\tau)^{(N)}\big(\partial^{(N)}_j P^1\big) P^2 + P^1 (\tau\otimes\id)^{(N)}\big(\partial^{(N)}_j P^2\big),
\end{align*}
so that \eqref{eq:Voiculescu-formula_amplification} yields the assertion.
\end{proof}

\subsubsection{Dabrowski's inequalities}

While the simple observation recorded in Lemma \ref{lem:adjoints_amplifications} allowed us to translate directly Voiculescu's results \cite{Voi98} about the adjoints of those unbounded linear operators that are induced by the noncommutative derivatives, establishing Dabrowski's inequalities \cite{Dab10} in the amplified setting is also possible but requires more work; in fact, this is the main strengthening compared to the tools already used in \cite{MSW17}. The precise statement reads as follows.

\begin{proposition}\label{prop:Dab-estimates}
Let $P\in M_N(\C\langle X_1,\dots,X_n\rangle)$ be given. Then
\begin{equation}\label{eq:Dab-estimates_1}
\begin{aligned}
\big\|(\partial_j^{(N)})^\ast (P \odot \1_N) \big\|_2 &\leq \|\xi_j\|_2 \|P\|,\\
\big\|(\partial_j^{(N)})^\ast (\1_N \odot P) \big\|_2 &\leq \|\xi_j\|_2 \|P\|,\\
\end{aligned}
\end{equation}
and
\begin{equation}\label{eq:Dab-estimates_2}
\begin{aligned}
\big\|(\tau\otimes\id)^{(N)} \big(\partial_j^{(N)} P\big)\big\|_2 &\leq 2\|\xi_j\|_2 \|P\|,\\
\big\|(\id\otimes\tau)^{(N)} \big(\partial_j^{(N)} P\big)\big\|_2 &\leq 2\|\xi_j\|_2 \|P\|.\\
\end{aligned}
\end{equation}
\end{proposition}

Our proof relies rather on Dabrowski's proof \cite{Dab10} for the scalar-valued case than on those statements themselves. We need to recall that Dabrowski's arguments yield -- though not stated explicitly in his paper -- that
\begin{equation}\label{eq:Dabrowski-identity}
\| \partial_j^\ast (P \otimes 1) \|_2^2 = \langle \partial_j^\ast \big( (P^\ast P) \otimes 1\big), \xi_j\rangle
\end{equation}
holds under our assumptions for each $P\in\C\langle X_1,\dots,X_n\rangle$. In addition, we will use the following easy result; a proof thereof can be found in \cite{Mai15}.

\begin{lemma}\label{lem:iterative-norm-bound}
Let $(\M,\tau)$ be a $W^\ast$-probability space and let $T: L^2(\M,\tau) \supseteq D(T) \rightarrow L^2(\M,\tau)$ be an unbounded linear operator whose domain $D(T)$ is a unital $\ast$-subalgebra of $\M$. Assume that the following conditions are satisfied:
\begin{enumerate}
 \item\label{it:iterative-norm-bound-1} There exists a constant $C>0$ such that
$$\|T(X)\|^2_2 \leq C \|T(X^\ast X)\|_2 \qquad\text{for all $X\in D(T)$}.$$
 \item\label{it:iterative-norm-bound-2} For each $X\in D(T)$, we have that
$$\limsup_{m\to\infty} \|T(X^m)\|^\frac{1}{m}_2 \leq \|X\|.$$
\end{enumerate}
Then $T$ satisfies $\|T(X)\|_2 \leq C \|X\|$ for all $X\in D(T)$.
\end{lemma}

Note that we will apply Lemma \ref{lem:iterative-norm-bound} for the tracial $W^\ast$-probability space $(M_N(\M_0), \tr_N \circ \tau^{(N)})$.

\begin{proof}[Proof of Proposition \ref{prop:Dab-estimates}]
We will establish first that \eqref{eq:Dabrowski-identity} extends to the matricial setting; more precisely, we claim that
\begin{equation}\label{eq:Dabrowski-identity-amplification}
\big\| (\partial_j^{(N)})^\ast \big(P \odot \1_N\big) \big\|_2^2 = \langle (\partial_j^{(N)})^\ast \big((P^\ast P) \odot \1_N\big), \1_N \xi_j\rangle
\end{equation}
holds for each $P\in M_N(\C\langle X_1,\dots,X_n\rangle)$. For seeing that, take any $P=(P_{kl})_{k,l=1}^N$ in $M_N(\C\langle X_1,\dots,X_n\rangle)$ and write, with respect the matrix units $e^{k,l}$, $1\leq k,l\leq N$, in $M_N(\C)$,
$$P = \sum^N_{k,l=1} e^{k,l} P_{kl} \qquad\text{and}\qquad P^\ast P = \sum^N_{k,l=1} e^{k,l} \bigg(\sum^N_{p=1} P_{pk}^\ast P_{pl}\bigg).$$
With the help of Lemma \ref{lem:adjoints_amplifications}, we see that accordingly
\begin{align*}
(\partial_j^{(N)})^\ast(P \odot \1_N) &= \sum^N_{k,l=1} e^{k,l} \partial_j^\ast(P_{kl} \otimes 1) \qquad\text{and}\\
(\partial_j^{(N)})^\ast( (P^\ast P) \odot \1_N) &= \sum^N_{k,l=1} e^{k,l} \bigg(\sum^N_{p=1} \partial_j^\ast\big( (P_{pk}^\ast P_{pl}) \otimes 1\big)\bigg).
\end{align*} 
Note that for all $1\leq k,l,k',l'\leq N$ and $\eta,\eta'\in L^2(\M_0,\tau)$
$$\langle e^{k,l} \eta, e^{k',l'} \eta' \rangle = \tr_N(e^{k,l} e^{l'k'}) \langle \eta, \eta' \rangle = \frac{1}{N} \delta_{l,l'} \delta_{k,k'} \langle \eta, \eta' \rangle,$$
so that
\begin{align*}
\big\| (\partial_j^{(N)})^\ast \big(P \odot \1_N\big) \big\|_2^2
&= \langle \sum^N_{k,l=1} e^{k,l} \partial_j^\ast(P_{kl} \otimes 1), \sum^N_{k',l'=1} e^{k',l'} \partial_j^\ast(P_{k'l'} \otimes 1) \rangle\\
&= \frac{1}{N} \sum^N_{k,l=1} \| \partial_j^\ast (P_{kl} \otimes 1) \|_2^2\\
&\stackrel{\eqref{eq:Dabrowski-identity}}{=} \frac{1}{N} \sum^N_{k,l=1} \langle \partial_j^\ast \big( (P^\ast_{kl} P_{kl}) \otimes 1\big), \xi_j\rangle
\end{align*}
and on the other hand
\begin{align*}
\langle (\partial_j^{(N)})^\ast \big((P^\ast P) \odot \1_N\big), \1_N \xi_j\rangle
&= \langle \sum^N_{k,l=1} e^{k,l} \bigg(\sum^N_{p=1} \partial_j^\ast\big( (P_{pk}^\ast P_{pl}) \otimes 1\big)\bigg), \1_N \xi_j\rangle\\
&= \frac{1}{N} \sum^N_{p,l=1} \langle \partial_j^\ast\big( (P_{pl}^\ast P_{pl}) \otimes 1\big), \xi_j \rangle.
\end{align*}
Thus, comparing both results yields the asserted identity \eqref{eq:Dabrowski-identity-amplification}.

Next, we apply the Cauchy-Schwarz inequality to the right hand side of \eqref{eq:Dabrowski-identity-amplification}; in this way, noting that $\|\1_N \xi_2\|_2 = \|\xi_j\|_2$, we obtain that
$$\big\| (\partial_j^{(N)})^\ast \big(P \odot \1_N\big) \big\|_2^2 \leq \big\| (\partial_j^{(N)})^\ast \big((P^\ast P) \odot \1_N\big) \big\|_2  \|\xi_j\|_2.$$
In terms of the unbounded linear operator $T$ that is given by
$$T:\ L^2(\M_0,\tau) \supseteq D(T) \to L^2(\M_0,\tau),\quad P \mapsto (\partial_j^{(N)})^\ast \big(P \odot \1_N\big)$$
with domain $D(T) = M_N(\C\langle X_1,\dots,X_n\rangle)$, the latter can be rewritten as
$$\| T(P) \|_2^2 \leq \| T(P^\ast P) \|_2 \|\xi_j\|_2.$$

We intend to apply Lemma \ref{lem:iterative-norm-bound} to $T$; while the previously obtained estimate verifies the condition required in Item \ref{it:iterative-norm-bound-1}, it remains to check Item \ref{it:iterative-norm-bound-2}. For that purpose, let us fix $P\in M_N(\C\langle X_1,\dots,X_n\rangle)$ and $m\in N$. With the help of Proposition \ref{prop:Voiculescu-formula_amplification}, using also the fact that $\partial^{(N)}_j$ is a derivation, we derive that
\begin{align*}
T(P^m) &= (\partial^{(N)})^\ast\big(P^m \odot \1_N\big)\\
       &= P^m (\1_N \xi_j) - (\id\otimes\tau)^{(N)}\big(\partial^{(N)}_j P^m\big)\\
       &= P^m (\1_N \xi_j) - \sum_{k=1}^m (\id\otimes\tau)^{(N)}\big(P^{k-1} \cdot (\partial_j^{(N)} P) \cdot P^{m-k}\big).
\end{align*}
Applying $\|\cdot\|_2$ and using the triangle inequality as well as the fact that $(\id\otimes\tau)^{(N)}$ is a contraction, this yields that
$$\| T(P^m) \|_2 \leq \|P\|^m \|\xi_j\|_2 + m \|P\|^{m-1} \|\partial_j^{(N)} P\|_2,$$
from which we immediately get that
$$\limsup_{m\rightarrow\infty} \|T(P^m)\|_2^{\frac{1}{m}} \leq \|P\|.$$

Therefore, we may now apply Lemma \ref{lem:iterative-norm-bound} to $T$ in the setting of $(M_N(\M_0), \tr_N \circ \tau^{(N)})$, which gives
$$\| T(P) \|_2 \leq \|\xi_j\|_2 \|P\|$$
and thus the first estimate in \eqref{eq:Dab-estimates_1}; the second one can be obtained similarly or by using the fact that $\partial_j^{(N)}$ is real in the sense of \eqref{eq:real_amplification}.

The estimates in \eqref{eq:Dab-estimates_2} can be deduced from \eqref{eq:Dab-estimates_1} with the help of Proposition \ref{prop:Voiculescu-formula_amplification}. Indeed, we may infer from Proposition \ref{prop:Voiculescu-formula_amplification} that
\begin{align*}
(\id\otimes\tau)^{(N)}\big(\partial^{(N)}_j P\big) &= P (\1_N \xi_j) - (\partial_j^{(N)})^\ast \big(P \odot \1_N\big) \qquad\text{and}\\
(\tau\otimes\id)^{(N)}\big(\partial^{(N)}_j P\big) &= (\1_N \xi_j) P - (\partial_j^{(N)})^\ast \big(\1_N \odot P\big),
\end{align*}
from which \eqref{eq:Dab-estimates_2} follows after applying the triangle inequality and using \eqref{eq:Dab-estimates_1}.
\end{proof}

In fact, we will need in the sequel some slight extension of Proposition \ref{prop:Dab-estimates}. This is the content of the following corollary.

\begin{corollary}\label{cor:key-estimates}
Let $P^1,P^2\in M_N(\C\langle X_1,\dots,X_n\rangle)$ be given. For $j=1,\dots,n$, we have that
\begin{equation}\label{eq:key-estimates_1}
\big\|(\partial_j^{(N)})^\ast (P^1 \odot P^2) \big\|_2 \leq 3 \|\xi_j\|_2 \|P^1\| \|P^2\|,
\end{equation}
and
\begin{equation}\label{eq:key-estimates_2}
\begin{aligned}
\big\| (\id\otimes\tau)^{(N)} \big((\partial_j^{(N)} P^1) \cdot P^2\big) \big\|_2 & \leq 4 \|\xi_j\|_2 \|P^1\| \|P^2\|,\\
\big\| (\tau\otimes\id)^{(N)} \big(P^1 \cdot (\partial_j^{(N)} P^2)\big) \big\|_2 & \leq 4 \|\xi_j\|_2 \|P^1\| \|P^2\|.
\end{aligned}
\end{equation}
\end{corollary}

\begin{proof}
Take any $P^1,P^2 \in M_N(\C\langle X_1,\dots,X_n\rangle)$. By using Proposition \ref{prop:Voiculescu-formula_amplification}, we see that
\begin{align*}
(\partial_j^{(N)})^\ast (P^1 \odot P^2)
 &= P^1 (\1_N \xi_j) P^2 - (\id\otimes\tau)^{(N)}\big(\partial^{(N)}_j P^1\big) P^2 - P^1 (\tau\otimes\id)^{(N)}\big(\partial^{(N)}_j P^2\big)\\
 &= (\partial_j^{(N)})^\ast(P^1 \odot \1_N) P^2 - P^1 (\tau\otimes\id)^{(N)}\big(\partial^{(N)}_j P^2\big).
\end{align*}
Finally, applying the estimates that were established in Proposition \ref{prop:Dab-estimates} yields that
\begin{align*}
\big\|(\partial_j^{(N)})^\ast (P^1 \odot P^2) \big\|_2
 &\leq \big\|(\partial_j^{(N)})^\ast(P^1 \odot \1_N)\big\|_2 \|P^2\| + \|P^1\| \big\|(\tau\otimes\id)^{(N)}\big(\partial^{(N)}_j P^2\big)\big\|_2\\
 &\leq 3 \|\xi_j\|_2 \|P^1\| \|P^2\|,
\end{align*}
which verifies \eqref{eq:key-estimates_1}. For proving \eqref{eq:key-estimates_2}, we proceed as follows: since $\partial_j^{(N)}$ forms a derivation on its domain $D(\partial_j^{(N)})$, we may use the strategy of ``integration by parts'' in order to derive that
\begin{align*}
(\id\otimes\tau)^{(N)}\big((\partial_j^{(N)} P^1) P^2\big)
 &= (\id\otimes\tau)^{(N)}\big(\partial_j^{(N)} (P^1 P^2)\big) - (\id\otimes\tau)^{(N)}\big( P^1 (\partial_j^{(N)} P^2) \big)\\
 &= (\id\otimes\tau)^{(N)}\big(\partial_j^{(N)} (P^1 P^2)\big) - P^1 (\id\otimes\tau)^{(N)}\big((\partial_j^{(N)} P^2) \big)
\end{align*}
for arbitrary $P^1,P^2\in M_N(\C\langle X_1,\dots,X_n\rangle)$, from which we may easily deduce with the help of \eqref{eq:Dab-estimates_2} that
\begin{align*}
\big\| (\id\otimes\tau)^{(N)}\big((\partial_j^{(N)} P^1) P^2\big)\big\|_2
 &\leq \big\|(\id\otimes\tau)^{(N)}\big(\partial_j^{(N)} (P^1 P^2)\big) \big\|_2 + \|P^1\| \big\|(\id\otimes\tau)^{(N)}\big((\partial_j^{(N)} P^2) \big)\|_2\\
 &\leq 4 \|\xi_j\|_2 \|P^1\| \|P^2\|.
\end{align*}
This is the first of the inequalities that are stated in \eqref{eq:key-estimates_2}, the second one can be proven analogously.
\end{proof}

\subsubsection{Non-microstates free entropy and free entropy dimension}

It was shown in \cite{Voi98} that arbitrarily small perturbations of any tuple $(X_1,\dots,X_n)$ of selfadjoint operators in $\M$ by freely independent semicircular elements lead to finite non-microstates free Fisher information. Indeed, if $S_1,\dots,S_n$ are semicircular elements in $\M$ which are freely independent among themselves and also free from $\{X_1,\dots,X_n\}$, then $(X_1+\sqrt{t}S_n,\dots,X_n+\sqrt{t}S_n)$ admits a conjugate system for each $t>0$ and we have the estimates (cf. \cite[Corollary 6.14]{Voi98})
\begin{equation}\label{eq:Fisher_perturbation}
\frac{n^2}{C^2 + nt} \leq \Phi^\ast(X_1+\sqrt{t}S_1,\dots,X_n+\sqrt{t}S_n) \leq \frac{n}{t} \quad\text{for all $t>0$},
\end{equation}
where $C^2 := \tau(X_1^2 + \dots + X_n^2)$.
Based on this observation, Voiculescu introduced in \cite{Voi98} the \emph{non-microstates free entropy} $\chi^\ast(X_1,\dots,X_n)$ of $X_1,\dots,X_n$ by
$$\chi^\ast(X_1,\dots,X_n) := \frac{1}{2} \int^\infty_0\Big(\frac{n}{1+t}-\Phi^\ast(X_1+\sqrt{t}S_1,\dots,X_n+\sqrt{t}S_n)\Big)\, dt + \frac{n}{2}\log(2\pi e).$$
Note that the left inequality in \eqref{eq:Fisher_perturbation} implies in particular that (cf. \cite[Proposition 7.2]{Voi98})
$$\chi^\ast(X_1,\dots,X_n) \leq \frac{n}{2}\log(2\pi e n^{-1} C^2).$$

The \emph{non-microstates free entropy dimension} $\delta^\ast(X_1,\dots,X_n)$ is now defined in terms of the non-microstates free entropy $\chi^\ast$ by
$$\delta^\ast(X_1,\dots,X_n) := n - \liminf_{\epsilon \searrow 0} \frac{\chi^\ast(X_1+\sqrt{\epsilon}S_1,\dots,X_n+\sqrt{\epsilon}S_n)}{\log(\sqrt{\epsilon})}.$$
In the case $n=1$ of a single operator $X=X^\ast\in\M$, we infer from \cite[Proposition 6.3]{Voi94} and the fact that the microstates entropy as introduced in \cite{Voi94} coincides in this case with the non-microstates entropy (cf. \cite[Proposition 7.6]{Voi98} that
\begin{equation}\label{eq:entropy_dimension_single_operator}
\delta^\ast(X) = 1 - \sum_{t\in\R} \mu_X(\{t\})^2.
\end{equation}
Here, $\mu_X$ denotes the \emph{analytic distribution of $X$}, i.e., the unique Borel probability measure $\mu_X$ on $\R$ that satisfies $\int_\R t^k\, d\mu_X(t) = \tau(X^k)$ for all $k\in\N_0$.

We note that in \cite{CS05} some variant of $\delta^\ast(X_1,\dots,X_n)$ was introduced, namely
$$\delta^\star(X_1,\dots,X_n) := n - \liminf_{t\searrow0} t \Phi^\ast(X_1+\sqrt{t}S_1,\dots,X_n+\sqrt{t}S_n),$$
whose defining expression is formally obtained by applying L'Hospital's rule to the $\liminf$ appearing in the definition of $\delta^\ast(X_1,\dots,X_n)$. We point out that $0 \leq \delta^\star(X_1,\dots,X_n) \leq n$ due to \eqref{eq:Fisher_perturbation}.

Furthermore, it was shown in \cite[Lemma 4.1]{CS05} that $\delta^\ast(X_1,\dots,X_n) \leq \delta^\star(X_1,\dots,X_n)$ holds, so that the condition $\delta^\star(X_1,\dots,X_n)=n$ is weaker than $\delta^\ast(X_1,\dots,X_n)=n$. Conceptually, we are mostly interested in situations where $\delta^\ast(X_1,\dots,X_n)=n$ holds, but our statements will be proven under the weaker assumption $\delta^\star(X_1,\dots,X_n)=n$, which is easier to work with since the associated quantity
$$\alpha(X_1,\dots,X_n) := n - \delta^\star(X_1,\dots,X_n) = \liminf_{t\searrow0} t \Phi^\ast(X_1+\sqrt{t}S_1,\dots,X_n+\sqrt{t}S_n)$$
emerges very naturally in our considerations.

\subsection{A matricial extension of the reduction argument}

The approach of \cite{MSW17}, where the authors aimed at proving absence of atoms for analytic distributions arising from evaluations of noncommutative polynomials in variables having maximal non-microstates free entropy dimension by excluding zero divisors, relied eminently on some ``reduction argument'' established in \cite[Proposition 3.9]{MSW17}.
The goal of this subsection is the following Theorem \ref{thm:Fisher-growth}, which constitutes a matricial analogue the aforementioned result.

\begin{theorem}\label{thm:Fisher-growth}
Let $(\M,\tau)$ be a tracial $W^\ast$-probability space and let $X_1,\dots,X_n\in \M$ be $n$ selfadjoint noncommutative random variables. Consider any matrix $P\in M_N(\C\langle x_1,\dots,x_n\rangle)$ of noncommutative polynomials and suppose that there are elements $u,v \in M_N(\M_0)$ such that both
$$P(X_1,\dots,X_n) u = 0 \qquad\text{and}\qquad P(X_1,\dots,X_n)^\ast v = 0.$$
Then, with the abbreviation $X=(X_1,\dots,X_n)$, we have that
$$\bigg(\sum^n_{j=1} \big|\langle v^\ast \cdot (\partial_j^{(N)} P)(X) \cdot u, Y^1 \odot Y^2\rangle\big|^2\bigg)^{\frac{1}{2}} \leq 4 \kappa_X(P;u,v) \alpha(X)^{\frac{1}{2}} \|Y^1\| \|Y^2\|$$
holds for all $Y^1,Y^2\in M_N(\C\langle X_1,\dots,X_n\rangle)$, where
$$\kappa_X(P;u,v) := \bigg(\sum^n_{j=1} \big\| (\partial^{(N)}_j P)(X) \cdot u\big\|_2^2\bigg)^{\frac{1}{2}} \|v\| + \|u\| \bigg(\sum^n_{j=1} \big\| (\partial^{(N)}_j P^\ast)(X) \cdot v\big\|^2_2\bigg)^{\frac{1}{2}}.$$
As a consequence, if $\delta^\star(X_1,\dots,X_n) = n$, then
$$v^\ast \cdot (\partial_j^{(N)} P)(X_1,\dots,X_n) \cdot u = 0 \qquad\text{for $j=1,\dots,n$}.$$
\end{theorem}

The proof of Theorem \ref{thm:Fisher-growth} will be given in Paragraph \ref{subsubsec:proof_Fisher-growth}. Like in \cite{MSW17}, it relies on a similar result, namely Theorem \ref{thm:Fisher-bound}, that holds under the stronger assumption of finite free Fisher information; this is the content of the next paragraph.

\subsubsection{A preliminary version of Theorem \ref{thm:Fisher-growth}}

We want to prove a preliminary version of Theorem \ref{thm:Fisher-growth} that holds under the stronger assumption of finite Fisher information. The precise statement reads as follows.

\begin{theorem}\label{thm:Fisher-bound}
Let $(\M,\tau)$ be a tracial $W^\ast$-probability space and let $X_1,\dots,X_n\in \M$ be $n$ selfadjoint noncommutative random variables. For all $P\in M_N(\C\langle X_1,\dots,X_n\rangle)$ and all $u,v\in M_N(\M_0)$, we have
\begin{equation}\label{eq:Fisher-bound}
|\langle v^\ast \cdot (\partial_j^{(N)} P) \cdot u, Q^1 \odot Q^2\rangle| \leq 4 \|\xi_j\|_2 \bigl(\|Pu\|_2 \|v\| + \|u\| \|P^\ast v\|_2\bigr) \|Q^1\| \|Q^2\|
\end{equation}
for all $Q^1,Q^2\in M_N(\C\langle X_1,\dots,X_n\rangle)$ and $j=1,\dots,n$. In particular, we have
\begin{equation}\label{eq:Fisher-bound_sum}
\begin{aligned}
\lefteqn{\sum^n_{j=1} |\langle v^\ast \cdot (\partial_j^{(N)} P) \cdot u, Q^1 \odot Q^2\rangle|^2}\\
 & \qquad \leq 16 \bigl(\|P u\|_2 \|v\| + \|u\| \|P^\ast v\|_2\bigr)^2 \Phi^\ast(X_1,\dots,X_n) \|Q^1\|^2 \|Q^2\|^2
\end{aligned}
\end{equation}
for all $Q^1,Q^2\in M_N(\C\langle X_1,\dots,X_n\rangle)$.
\end{theorem}

In addition to the results obtained in the previous paragraphs, we will use in proof of Theorem \ref{thm:Fisher-bound} the following fact, which is a direct consequence of Kaplansky's density theorem.

\begin{lemma}\label{lem:Kaplansky}
For any $w\in M_N(\M_0)$, there exists a sequence $(w_k)_{k\in\N}$ in $M_N(\C\langle X_1,\dots,X_n\rangle)$ such that
$$\sup_{k\in\N} \|w_k\| \leq \|w\| \qquad\text{and}\qquad \lim_{k\to \infty} \|w_k-w\|_2 = 0.$$
\end{lemma}

\begin{proof}
By Kaplansky's density theorem, each element $w\in M_N(\M_0)$ can be approximated with respect to the strong operator topology by some net $(w_\lambda)_{\lambda\in\Lambda}$ in $M_N(\C\langle X_1,\dots,X_n\rangle)$ satisfying $\|w_\lambda\| \leq \|w\|$. Now, since the net $(w_\lambda)_{\lambda\in\Lambda}$ converges to $w$ in the strong operator topology, it also converges in the $L^2$-topology to $w$. Thus, we may choose a subsequence $(w_{\lambda(k)})_{k\in\N}$ that converges to $w$ in the $L^2$-sense, which does the job.
\end{proof}

\begin{proof}[Proof of Theorem \ref{thm:Fisher-bound}]
First, we consider the case $u,v\in M_N(\C\langle X_1,\dots,X_n\rangle)$. If $Q^1,Q^2\in M_N(\C\langle X_1,\dots,X_n\rangle)$ are given, we may check then that
\begin{align*}
\lefteqn{\langle Pu, (\partial_j^{(N)})^\ast\big( (v Q^1) \odot Q^2\big)\rangle}\\
&= \langle \partial_j^{(N)}(Pu), v \cdot (Q^1 \odot Q^2)\rangle\\
&= \langle (\partial_j^{(N)} P) \cdot u,  v \cdot (Q^1 \odot Q^2)\rangle + \langle P \cdot (\partial_j^{(N)} u), v \cdot (Q^1 \odot Q^2)\rangle\\
&= \langle v^\ast \cdot (\partial_j^{(N)} P) \cdot u, Q^1 \odot Q^2\rangle + \langle (\partial_j^{(N)} u) \cdot (Q^2)^\ast, (P^\ast v Q^1) \odot \1_N\rangle\\
&= \langle v^\ast \cdot (\partial_j^{(N)} P) \cdot u, Q^1 \odot Q^2\rangle + \langle (\id\otimes\tau)^{(N)}\big((\partial_j^{(N)} u) \cdot (Q^2)^\ast\big), (P^\ast v) Q^1\rangle
\end{align*}
holds. Rearranging the above equation, applying the triangle inequality, and using Corollary \ref{cor:key-estimates} yields that
\begin{align*}
\lefteqn{\big|\langle v^\ast \cdot (\partial_j^{(N)} P) \cdot u, Q^1 \odot Q^2\rangle\big|}\\
&\leq \big|\langle Pu, (\partial_j^{(N)})^\ast\big( (v Q^1) \odot Q^2)\rangle\big| + \big|\langle (\id\otimes\tau)^{(N)}\big((\partial_j^{(N)} u) \cdot (Q^2)^\ast\big), (P^\ast v) Q^1\rangle\big|\\
&\leq \|Pu\|_2 \big\|(\partial_j^{(N)})^\ast\big( (v Q^1) \odot Q^2)\big\|_2 + \big\|(\id\otimes\tau)^{(N)}\big((\partial_j^{(N)} u) \cdot (Q^2)^\ast\big)\big\|_2 \|(P^\ast v) Q^1\|_2\\
&\leq 3 \|\xi_j\|_2 \|Pu\|_2 \|v Q^1\| \|Q^2\| + 4 \|\xi_j\|_2 \|u\| \|(Q^2)^\ast\| \|P^\ast v Q^1\|_2\\
&\leq 4 \|\xi_j\|_2 \bigl(\|Pu\|_2 \|v\| + \|u\| \|P^\ast v\|_2\bigr) \|Q^1\| \|Q^2\|,
\end{align*}
which is \eqref{eq:Fisher-bound}. The validity of \eqref{eq:Fisher-bound} in the general case of arbitrary $u,v\in M_N(\M_0)$ is due to Lemma \ref{lem:Kaplansky}. This shows the first part of the statement.

The second inequality \eqref{eq:Fisher-bound_sum} follows by taking squares on both sides of \eqref{eq:Fisher-bound}, summing over all $j=1,\dots,n$, and using that $\Phi^\ast(X_1,\dots,X_n) = \sum^n_{j=1}\|\xi_j\|^2_2$.
\end{proof}

\subsubsection{Proof of Theorem \ref{thm:Fisher-growth}}
\label{subsubsec:proof_Fisher-growth}

Let $(\M,\tau)$ be a tracial $W^\ast$-probability space and consider $n$ selfadjoint noncommutative random variables $X_1,\dots,X_n\in \M$.

With no loss of generality, we may assume that $\M$ contains $n$ normalized semicircular elements $S_1,\dots,S_n$ such that $\{X_1,\dots,X_n\},\{S_1\},\dots,\{S_n\}$ are freely independent; the following lemma records an important consequence of that assumption.

\begin{lemma}\label{lem:isometry}
The linear mapping
$$\Theta:\ M_N(\M_0 \otimes \M_0)^n \to M_N(\M),\quad (Q^1,\dots,Q^n) \mapsto \sum^n_{j=1} Q^j \sharp S_j$$
extends to an isometry
$$\hat{\Theta}:\ L^2(M_N(\M_0 \ootimes \M_0), \tr_N \circ (\tau \ootimes \tau)^{(N)})^n \to L^2(M_N(\M),\tr_N \circ \tau^{(N)})$$
\end{lemma}

\begin{proof}
If $Q^1,\dots,Q^n \in M_N(\M_0 \otimes \M_0)$ are given, say $Q^j = (Q^j_{kl})_{k,l=1}^N$ for $j=1,\dots,n$, then
$$\bigg\|\sum^n_{j=1} Q^j \sharp S_j\bigg\|_2^2
= \sum_{i,j=1}^n \langle Q^i \sharp S_i, Q^j \sharp S_j\rangle
= \frac{1}{N} \sum_{i,j=1}^n \sum_{k,l=1}^N \langle Q^i_{kl} \sharp S_i, Q^j_{lk} \sharp S_j\rangle.$$
Now, the assumed freeness gives us that $\langle Q^i_{kl} \sharp S_i, Q^j_{lk} \sharp S_j\rangle = \delta_{i,j} \langle Q^j_{kl}, Q^j_{lk}\rangle$. Hence, we see that
$$\bigg\|\sum^n_{j=1} Q^j \sharp S_j\bigg\|_2^2
= \frac{1}{N} \sum_{i,j=1}^n \sum_{k,l=1}^N \langle Q^i_{kl} \sharp S_i, Q^j_{lk} \sharp S_j\rangle
= \frac{1}{N} \sum_{j=1}^n \sum_{k,l=1}^N \langle Q^j_{kl}, Q^j_{lk} \rangle
= \sum^n_{j=1} \|Q^j\|_2^2,$$
i.e., we have that $\|\Theta(Q^1,\dots,Q^n)\|_2^2 = \sum^n_{j=1} \|Q^j\|_2^2$. This confirms that $\Theta$ admits an isometric extension of the desired form.
\end{proof}

Now, let us define for each $t\geq 0$ the variables
$$X_j^t := X_j + \sqrt{t} S_j \qquad\text{for $j=1,\dots,n$}$$
and denote by $\M_t := \vN(X_1^t,\dots,X_n^t)$ the von Neumann subalgebra of $\M$ that they generate; in the case $t=0$, this is in accordance with our previous definition of $\M_0$. Furthermore, we abbreviate $X^t=(X_1^t,\dots,X_n^t)$ for each $t\geq 0$, so that in particular $X^0=X=(X_1,\dots,X_n)$.

Since $\M_t$ is a von Neumann subalgebra of $\M$, there is a unique trace-preserving conditional expectation $\E_t$ from $\M$ onto $\M_t$. Note that $\E_t^{(N)}$ gives then the unique trace-preserving conditional expectation from $M_N(\M)$ to $M_N(\M_t)$.

\begin{lemma}\label{lem:t-limits}
Take any $P\in M_N(\C\langle x_1,\dots,x_n\rangle)$ and suppose that there is an element $w\in M_N(\M_0)$ such that $P(X) w = 0$. Then
$$\lim_{t\searrow 0} \frac{1}{\sqrt{t}} P(X^t)w = \sum^n_{j=1} \big((\partial_j^{(N)} P)(X) \cdot w \big) \sharp S_j$$
and 
$$\lim_{t\searrow 0} \frac{1}{\sqrt{t}} \|P(X^t)w\|_2 = \bigg(\sum^n_{j=1} \big\|(\partial_j^{(N)} P)(X) \cdot w \big\|_2^2\bigg)^{1/2}.$$
Moreover, if we consider $w_t := \E_t^{(N)}[w] \in \M_t$, then
$$\limsup_{t\searrow0}  \frac{1}{\sqrt{t}} \|P(X^t) w_t\|_2 \leq \bigg(\sum^n_{j=1} \big\|(\partial_j^{(N)} P)(X) \cdot w \big\|_2^2\bigg)^{1/2}.$$
\end{lemma}

\begin{proof}
We write $P=(P_{kl})_{k,l=1}^N$. Then each $t\mapsto P_{k,l}(X^t)$ is a polynomial in $\sqrt{t}$ with coefficients coming from $\M$ and we have that
$$P_{kl}(X^t) = P_{kl}(X) + \sqrt{t} \sum^n_{j=1} (\partial_j P_{kl})(X) \sharp S_j + \omega_{kl}(t) \qquad\text{for all $t\geq 0$},$$
where $\omega_{kl}: [0,\infty) \to \M$ is a polynomial in $\sqrt{t}$ that satisfies $\lim_{t \searrow 0} \frac{1}{\sqrt{t}} \|\omega_{kl}(t)\| = 0$. Put $\omega = (\omega_{kl})_{k,l=1}^N: [0,\infty) \to M_N(\M)$. Then
$$P(X^t) = P(X) + \sqrt{t} \sum^n_{j=1} (\partial_j^{(N)} P)(X) \sharp S_j + \omega(t) \qquad\text{for all $t\geq 0$},$$
and since $P(X) w = 0$, the latter identity gives
$$\frac{1}{\sqrt{t}} P(X^t) w = \sum^n_{j=1} \big((\partial_j^{(N)} P)(X) \cdot w\big) \sharp S_j + \frac{1}{\sqrt{t}} \omega(t) w \qquad\text{for all $t\geq 0$}.$$
Because also $\lim_{t \searrow 0} \frac{1}{\sqrt{t}} \|\omega(t) w\| = 0$, we infer that
$$\lim_{t\searrow 0} \frac{1}{\sqrt{t}} P(X^t)w = \sum^n_{j=1} \big((\partial_j^{(N)} P)(X) \cdot w \big) \sharp S_j,$$
as claimed. With the help of Lemma \ref{lem:isometry}, we may deduce now that
\begin{align*}
\lim_{t\searrow 0} \frac{1}{\sqrt{t}} \|P(X^t)w\|_2
&= \bigg\|\sum^n_{j=1} \big((\partial_j^{(N)} P)(X) \cdot w \big) \sharp S_j\bigg\|_2\\
&= \hat{\Theta}\big((\partial_1^{(N)} P)(X) \cdot w, \dots, (\partial_n^{(N)} P)(X) \cdot w\big)\\
&= \bigg(\sum^n_{j=1} \big\|(\partial_j^{(N)} P)(X) \cdot w \big\|_2^2\bigg)^{1/2},
\end{align*}
which is the second formula. The last formula follows from the second one, since
$$P(X^t) w_t = P(X^t) \E_t^{(N)}[w] = \E_t^{(N)}[P(X^t)w]$$
and hence
$$\|P(X^t) w_t\|_2 = \|\E_t^{(N)}[P(X^t)w]\|_2 \leq \|P(X^t)w\|_2$$
holds. This concludes the proof.
\end{proof}

We are now prepared to prove Theorem \ref{thm:Fisher-growth}.

\begin{proof}[Proof of Theorem \ref{thm:Fisher-growth}]
Let us take any $P\in M_N(\C\langle x_1,\dots,x_n\rangle)$ and let us suppose that there are elements $u,v \in M_N(\M_0)$ such that both $P(X) u = 0$ and $P(X)^\ast v = 0$ holds. For each $t\geq0$, we introduce $u_t := \E_t^{(N)}[u]$ and $v_t := \E_t^{(N)}[v]$, which are by construction elements of $\M_t$ that satisfy $\|u_t\| \leq \|u\|$ and $\|v_t\| \leq \|v\|$.

We infer from Lemma \ref{lem:t-limits} that
$$\limsup_{t\searrow0} \frac{1}{\sqrt{t}} \|P(X^t) u_t\|_2 \leq \bigg(\sum^n_{j=1} \big\|(\partial_j^{(N)} P)(X) \cdot u \big\|_2^2\bigg)^{1/2},$$
and since $P^\ast(X) v = P(X)^\ast v = 0$, also that
$$\limsup_{t\searrow0} \frac{1}{\sqrt{t}} \|P^\ast(X^t) v_t\|_2 \leq \bigg(\sum^n_{j=1} \big\|(\partial_j^{(N)} P^\ast)(X) \cdot v \big\|_2^2\bigg)^{1/2},$$
which can be reformulated as
$$\limsup_{t\searrow0} \frac{1}{\sqrt{t}} \|P(X^t)^\ast v_t\|_2 \leq \bigg(\sum^n_{j=1} \big\|(\partial_j^{(N)} P)(X)^\dagger \cdot v \big\|_2^2\bigg)^{1/2}.$$

Now, take any $Q^1,Q^2\in M_N(\C\langle x_1,\dots,x_n\rangle)$. Since $\Phi^\ast(X^t)<\infty$, we obtain by the inequality \eqref{eq:Fisher-bound_sum} in Theorem \ref{thm:Fisher-bound} that
\begin{align*}
\lefteqn{\sum^n_{j=1} |\langle v_t^\ast \cdot (\partial_j^{(N)} P)(X^t) \cdot u_t, Q^1(X^t) \odot Q^2(X^t)\rangle|^2}\\
 & \quad \leq 16 \bigl(\|P(X^t) u_t\|_2 \|v_t\| + \|u_t\| \|P(X^t)^\ast v_t\|_2\bigr)^2 \Phi^\ast(X^t) \|Q^1(X^t)\|^2 \|Q^2(X^t)\|^2\\
 & \quad \leq 16 \bigg(\frac{1}{\sqrt{t}}\|P(X^t) u_t\|_2 \|v\| + \|u\| \frac{1}{\sqrt{t}} \|P(X^t)^\ast v_t\|_2\bigg)^2 \big(t\Phi^\ast(X^t)\big) \|Q^1(X^t)\|^2 \|Q^2(X^t)\|^2.
\end{align*}
Hence, we may conclude that
\begin{align*}
\liminf_{t\searrow 0} \sum^n_{j=1} |\langle v_t^\ast \cdot (\partial_j^{(N)} P)(X^t) \cdot u_t&, Q^1(X^t) \odot Q^2(X^t)\rangle|^2\\
                           & \leq 16 \kappa_X(P;u,v)^2 \alpha(X) \|Q^1(X)\|^2 \|Q^2(X)\|^2.
\end{align*}

In order to establish Theorem \ref{thm:Fisher-growth}, it thus only remains to show that
\begin{equation}\label{eq:Fisher-growth_liminf}
\begin{aligned}
\liminf_{t\searrow 0}\ \sum^n_{j=1} |\langle v_t^\ast \cdot (\partial_j^{(N)} P)(X^t) & \cdot u_t, Q^1(X^t) \odot Q^2(X^t)\rangle|^2\\
&= \sum^n_{j=1} |\langle v^\ast \cdot (\partial_j^{(N)} P)(X) \cdot u, Q^1(X) \odot Q^2(X)\rangle|^2.
\end{aligned}
\end{equation}
This is indeed sufficient, since each $Y\in M_N(\C\langle X_1,\dots,X_n\rangle)$ is of the form $Y=Q(X)$ for some $Q\in M_N(\C\langle x_1,\dots,x_n\rangle)$.

We first note that $\E_t \ootimes \E_t$ gives the unique trace-preserving conditional expectation from $\M \ootimes \M$ to $\M_t \ootimes \M_t$, so that $(\E_t \ootimes \E_t)^{(N)}$ is the unique trace-preserving conditional expectation from $M_N(\M \ootimes \M)$ to $M_N(\M_t \ootimes \M_t)$. Using this fact, we may check that
\begin{align*}
\lefteqn{\langle v_t^\ast \cdot (\partial_j^{(N)} P)(X^t) \cdot u_t, Q^1(X^t) \odot Q^2(X^t)\rangle}\\
&\qquad = \langle \E_t^{(N)}[v^\ast] \cdot (\partial_j^{(N)} P)(X^t) \cdot \E_t^{(N)}[u], Q^1(X^t) \odot Q^2(X^t)\rangle\\
&\qquad = \langle (\E_t \ootimes \E_t)^{(N)}[v^\ast \cdot (\partial_j^{(N)} P)(X^t) \cdot u], Q^1(X^t) \odot Q^2(X^t)\rangle\\
&\qquad = \langle v^\ast \cdot (\partial_j^{(N)} P)(X^t) \cdot u, Q^1(X^t) \odot Q^2(X^t)\rangle
\end{align*}
and the latter expression is actually a complex polynomial in $\sqrt{t}$. Altogether, this shows that
$$\lim_{t\searrow0}\ \langle v_t^\ast \cdot (\partial_j^{(N)} P)(X^t) \cdot u_t, Q^1(X^t) \odot Q^2(X^t)\rangle = \langle v^\ast \cdot (\partial_j^{(N)} P)(X) \cdot u, Q^1(X) \odot Q^2(X)\rangle,$$
so that
\begin{align*}
\lim_{t\searrow 0}\ \sum^n_{j=1} |\langle v_t^\ast \cdot (\partial_j^{(N)} P)(X^t) \cdot u_t&, Q^1(X^t) \odot Q^2(X^t)\rangle|^2\\
&= \sum^n_{j=1} |\langle v^\ast \cdot (\partial_j^{(N)} P)(X) \cdot u, Q^1(X) \odot Q^2(X)\rangle|^2
\end{align*}
holds, from which \eqref{eq:Fisher-growth_liminf} follows. This completes the proof of Theorem \ref{thm:Fisher-growth}.
\end{proof}

\subsubsection{A reformulation of Theorem \ref{thm:Fisher-growth}}

In the spirit of \cite{MSW17}, the reduction argument provided by Theorem \ref{thm:Fisher-growth} will be used later on actually in the following reformulation.

\begin{corollary}\label{cor:reduction}
Let $(\M,\tau)$ be a tracial $W^\ast$-probability space and let $X_1,\dots,X_n\in \M$ be $n$ selfadjoint noncommutative random variables that satisfy the condition
$$\delta^\star(X_1,\dots,X_n)=n.$$
Consider any matrix $P\in M_N(\C\langle x_1,\dots,x_n\rangle)$ of noncommutative polynomials and suppose that there is a projection $p \in M_N(\M_0)$ such that
$$P(X_1,\dots,X_n) p = 0.$$
Then there exists a projection $q\in M_N(\M_0)$ with the property $(\tr_N \circ \tau^{(N)})(q) \geq (\tr_N \circ \tau^{(N)})(p)$ and such that $P(X_1,\dots,X_n)^\ast q = 0$ and
$$q \cdot (\partial_j^{(N)} P)(X_1,\dots,X_n) \cdot p = 0 \qquad\text{for $j=1,\dots,n$}.$$
\end{corollary}

Before giving the proof of Corollary \ref{cor:reduction}, we need to recall the following well-known result.

\begin{lemma}\label{lem:kernels}
Let $X$ be an element of any tracial $W^\ast$-probability space $(\M,\tau)$ over some complex Hilbert space $H$. Let $p_{\ker(X)}$ and $p_{\ker(X^\ast)}$ denote the orthogonal projections onto $\ker(X)$ and $\ker(X^\ast)$, respectively.

The projections $p_{\ker(X)}$ and $p_{\ker(X^\ast)}$ belong both to $\M$ and satisfy
$$\tau(p_{\ker(X)}) = \tau(p_{\ker(X^\ast)}).$$
Thus, in particular, if $\ker(X)$ is non-zero, then also $\ker(X^\ast)$ is a non-zero subspace of $H$.
\end{lemma}

The interested reader can find a detailed proof of that statement in \cite{MSW17}.

\begin{proof}[Proof of Corollary \ref{cor:reduction}]
Lemma \ref{lem:kernels} tells us that the projections $p_{\ker(P(X))}$ and $p_{\ker(P(X)^\ast)}$ both belong to $\M_0$ and satisfy $(\tr_N \circ \tau^{(N)})(p_{\ker(P(X))}) = (\tr_N \circ \tau^{(N)})(p_{\ker(P(X)^\ast)})$. We put $q := p_{\ker(P(X)^\ast)}$; note that in particular $P(X)^\ast q = 0$.
From $P(X) p = 0$, it follows that $\ran(p) \subseteq \ker(P(X))$. The projection $p_{\ker(P(X))}$ thus satisfies $p \leq p_{\ker(P(X))}$, so that
$$(\tr_N \circ \tau^{(N)})(p) \leq (\tr_N \circ \tau^{(N)})(p_{\ker(P(X))}) = (\tr_N \circ \tau^{(N)})(q)$$
by the positivity of $\tr_N\circ\tau^{(N)}$. Now, we may apply Theorem \ref{thm:Fisher-growth}, which then yields also the assertion
$$q \cdot (\partial_j^{(N)} P)(X_1,\dots,X_n) \cdot p = 0 \qquad\text{for $j=1,\dots,n$}$$
due to our assumption that $\delta^\star(X_1,\dots,X_n)=n$.
\end{proof}

\section{Regularity of matrices with linear entries}
\label{sec:regularity_linear_matrices}

Since recently, operator-valued semicircular elements are attracting much attention motivated by far reaching applications in random matrix theory. Those are noncommutative random variables of the form
$$\bS = b_0 + b_1 S_1 + \dots + b_n S_n$$
with selfadjoint coefficient matrices $b_0,b_1,\dots,b_n \in M_N(\C)$ and a tuple $(S_1,\dots,S_n)$ of freely independent semicircular elements. In some impressive series of publications (see, for instance, \cite{EKYY13,AjEK18,AEK18} and the references collected therein), a deep understanding of the regularity properties of their distributions was gained. These achievements rely on a very detailed analysis of the so-called \emph{Dyson equation}, which is some kind of quadratic equation on $M_N(\C)$ (or, more generally, on von Neumann algebras) for their operator-valued Cauchy transforms that determines in particular their scalar-valued Cauchy transforms and hence their analytic distributions $\mu_{\bS}$.

It turns out that analytic properties of $\mu_{\bS}$ strongly depend on the algebraic properties of the coefficient matrices $b_0,b_1,\dots,b_n$. In \cite{AjEK18}, the associated \emph{self-energy operator} (or \emph{quantum operator} in the terminology of \cite{GGOW16} used in Proposition \ref{prop:rank-decreasing})
$$\cL:\ M_N(\C) \to M_N(\C),\qquad b \mapsto \sum^n_{j=1} b_j b b_j$$
was supposed to be \emph{flat} in the sense that there are constants $c_1,c_2>0$ such that
\begin{equation}\label{eq:flatness}
c_1 \tr_N(b) \1_N \leq \cL(b) \leq c_2 \tr_N(b) \1_N \qquad\text{for all positive semidefinite $b\in M_N(\C)$}. 
\end{equation}
Note that $b_0$, the so-called \emph{bare matrix}, plays a special role and accordingly does not show up in $\cL$. We will come back to the flatness condition in Section \ref{sec:Hoelder_continuity}.

In this section, we consider more generally operator-valued elements of the form
$$\X = b_0 + b_1 X_1 + \dots + b_n X_n,$$
where, on the analytic side, we allow $(X_1,\dots,X_n)$ to be any tuple of selfadjoint noncommutative random variables that satisfies $\delta^\star(X_1,\dots,X_n)=n$. On the algebraic side, we significantly relax the flatness condition by requiring only that the associated linear polynomial
$$P = b_0 + b_1 x_1 + \dots + b_n x_n \in M_N(\C\langle x_1,\dots,x_n\rangle)$$
is full, where $\C\langle x_1,\dots,x_n\rangle$ denotes as before the ring of noncommutative polynomials in the (formal) non-commuting variables $x_1,\dots,x_n$.

Let us point out that, due to Proposition \ref{prop:rank-decreasing}, the homogeneous part $b_1 x_1 + \dots + b_n x_n$ of $P$ is full if and only if $\cL$ is nowhere rank-decreasing. Now, if $\cL$ is flat, then the lower estimate in \eqref{eq:flatness} enforces $\cL$ to be nowhere rank-decreasing, so that $b_1 x_1 + \dots + b_n x_n$ and hence (as one sees, for instance with the help of Proposition \ref{prop:shrunk_subspace}) also $P$ must be full. Therefore, flatness of $\cL$ is indeed a much stronger requirement than fullness of $P$.

Our motivation to study regularity properties for such operator-valued elements $\X$ has in fact two sources. On the one hand, it is natural to ask which of the results that were obtained for operator-valued semicircular elements survive in this generality -- especially because the description in terms of the Dyson equation is no longer available. On the other hand, these operators are at the core of our present investigations, since they are intimately related via the linearization machinery to questions about evaluations of noncommutative rational functions.

Our first main theorem reads as follows.

\begin{theorem}\label{thm:main-1}
Suppose that
\begin{enumerate}
 \item $b_0,b_1,\dots,b_n$ are (not necessarily selfadjoint) matrices in $M_N(\C)$ for which
 $$P=b_0 + b_1 x_1 + \dots + b_n x_n \in M_N(\C\langle x_1,\dots,x_n\rangle)$$
 is full over $\C\langle x_1,\dots,x_n\rangle$;
 \item $X_1,\dots,X_n$ are selfadjoint noncommutative random variables in some tracial $W^\ast$-probability space $(\M,\tau)$ that satisfy $$\delta^\star(X_1,\dots,X_n)=n.$$
\end{enumerate}
Put $\M_0 := \vN(X_1,\dots,X_n) \subseteq \M$ and consider the operator
$$P(X_1,\dots,X_n) = b_0 + b_1 X_1 + \dots + b_n X_n \in M_N(\M_0)$$
If now $p\in M_N(\M_0)$ is any projection satisfying $P(X_1,\dots,X_n) p = 0$, then necessarily $p=0$.
\end{theorem}

The proof of Theorem \ref{thm:main-1} relies crucially on the following easy fact, which will allow us to reduce inductively the dimension of the considered matrices.

\begin{lemma}\label{lem:projection_reduction}
Let $(\M,\tau)$ be tracial $W^\ast$-probability space. Suppose that $p\in M_N(\M)$ is a projection that satisfies
$$\tilde{p} := \tau^{(N)}(p) = \begin{pmatrix} 0 & \tilde{p}_{12}\\ \tilde{p}_{21} & \tilde{p}_{22}\end{pmatrix}$$
with the block $\tilde{p}_{22}$ belonging to $M_{N-1}(\C)$ and all other blocks being of appropriate size. Then necessarily $\tilde{p}_{12} = 0$ and $\tilde{p}_{21} = 0$ and we have that
$$p = \begin{pmatrix} 0 & 0\\ 0 & p_{22}\end{pmatrix}$$
with a projection $p_{22} \in M_{N-1}(\M)$.
\end{lemma}

\begin{proof}
Let us consider the block decomposition of $p$ of the form
$$p = \begin{pmatrix} p_{11} & p_{12}\\ p_{21} & p_{22} \end{pmatrix},$$
where the block $p_{22}$ belongs to $M_{N-1}(\M)$ and all other blocks are of appropriate size. Since $p$ is supposed to be a projection, it satisfies $p^\ast=p$ and $p^2 = p$. Then $p=p^\ast$ yields in particular that
\begin{equation}\label{eq:projection_reduction-1}
p_{11} = p_{11}^\ast \qquad\text{and}\qquad p_{12} = p_{21}^\ast,
\end{equation}
and from $p = p^2$ it follows that
\begin{equation}\label{eq:projection_reduction-2}
p_{11} = p_{11}^2 + p_{12} p_{21}.
\end{equation}
Combining these facts gives us that
\begin{equation}\label{eq:projection_reduction-3}
p_{11} = p_{11}^2 + p_{12} p_{21} = p_{11}^\ast p_{11} + p_{21}^\ast p_{21} \geq p_{11}^\ast p_{11}.
\end{equation}
Now, we invoke our assumption on $\tilde{p}=\tau^{(N)}(p)$, which says that $\tau(p_{11})=0$; by \eqref{eq:projection_reduction-3} and the positivity of $\tau$, this enforces that $\tau(p_{11}^\ast p_{11}) = 0$, and since $\tau$ is moreover faithful, we may infer that $p_{11} = 0$. Hence, with the help of \eqref{eq:projection_reduction-2}, we see that $p_{21}^\ast p_{21} = 0$, which gives $p_{21} = 0$ and thus by \eqref{eq:projection_reduction-1} also $p_{12} = p_{21}^\ast = 0$.
In summary, this shows that $p$ has the desired form; furthermore, it follows that $\tilde{p}_{12} = 0$ and $\tilde{p}_{21} = 0$, as asserted.
\end{proof}

\begin{proof}[Proof of Theorem \ref{thm:main-1}]
Our proof proceeds by mathematical induction on the matrix-size $N$. We suppose that $X_1,\dots,X_n$ are selfadjoint elements in a tracial $W^\ast$-probability space $(\M,\tau)$ that satisfy $\delta^\star(X_1,\dots,X_n)=n$; like above, we put $\M_0:=\vN(X_1,\dots,X_n) \subseteq \M$ and we abbreviate $X:=(X_1,\dots,X_n)$. We want to show the validity of the following assertion for each $N\geq 1$:
$$A(N) \quad \begin{cases} \quad \begin{minipage}[c]{0.79\textwidth} If $b_0,b_1,\dots,b_n$ are matrices in $M_N(\C)$ for which the associated polynomial $$P = b_0 + b_1 x_1 + \dots + b_n x_n$$ in $M_N(\C\langle x_1,\dots,x_n\rangle)$ is full over $\C\langle x_1,\dots,x_n\rangle$, then the only projection $p\in M_N(\M_0)$ that satisfies $P(X) p = 0$ is $p=0$.\end{minipage} \end{cases}$$ 
First of all, we note the following: whenever matrices $b_0,b_1,\dots,b_n\in M_N(\C)$ and a projection $p\in M_N(\M_0)$ are given such that $P = b_0 + b_1 x_1 + \dots + b_n x_n$ satisfies $P(X) p = 0$, then Corollary \ref{cor:reduction} shows the existence of another projection $q\in M_N(\M_0)$ that satisfies $(\tr_N \circ \tau^{(N)})(q) \geq (\tr_N \circ \tau^{(N)})(p)$, $P(X)^\ast q = 0$, and  
$$q \cdot (\partial^{(N)}_j P)(X) \cdot p = 0 \qquad\text{for $j=1,\dots,n$}.$$
By definition of $P$, clearly $\partial_j^{(N)} P = b_j \odot \1_N$, so that
$$(q b_j) \odot p = 0 \qquad\text{for $j=1,\dots,n$}.$$
In particular, since $P(X)^\ast q = 0$ is equivalent to $q P(X) = 0$,
$$0 = \sum^n_{j=1} \big( (q b_j) \odot p \big) \sharp X_j = \sum^n_{j=1} (q b_j X_j) \odot p =  (q P(X)) \odot p - (q b_0) \odot p = - (q b_0) \odot p,$$
so that in summary
$$(q b_j) \odot p = 0 \qquad\text{for $j=0,1,\dots,n$}.$$
Now, applying $(\tau\ootimes\tau)^{(N)}$ on both sides of the last equation gives
\begin{equation}\label{eq:coefficients}
\tilde{q} b_j \tilde{p} = 0 \qquad\text{for $j=1,\dots,n$},
\end{equation}
where we put $\tilde{p} := \tau^{(N)}(p)$ and $\tilde{q} := \tau^{(N)}(q)$. 

Let us treat $A(1)$ first. The assumption that $P$ is full means that at least one of the scalar coefficients $b_0,b_1\dots,b_n$ is non-zero; but \eqref{eq:coefficients} then tells us that $\tilde{q}=0$ or $\tilde{p}=0$; since $\tau$ is faithful, this means that $q=0$ or $p=0$, which by construction of $q$ implies in either case that $p=0$. This proves the validity of $A(1)$.

Now, suppose that $A(N-1)$ for some $N\geq 2$ is already proven; we want to establish that also $A(N)$ holds true. Again, we start our considerations at \eqref{eq:coefficients}, which yields after consulting Corollary \ref{cor:fullness_contraction} that
\begin{equation}\label{eq:fullness_condition}
\rank(\tilde{p}) + \rank(\tilde{q}) \leq N.
\end{equation}
If we would have that $\rank(\tilde{q}) = 0$, i.e., $\tilde{q}=0$, then the faithfulness of $\tau^{(N)}$ would imply that $q=0$; then, we would be done, since $q=0$ enforces by construction of $q$ that $p_{\ker(P(X))}=0$ and thus $p=0$, which is what we wished to show. Accordingly, it suffices to treat the case $\rank(\tilde{q}) \neq 0$. Since $\tilde{q}$ is a scalar matrix, we thus have that $\rank(\tilde{q}) \geq 1$ and the previous observation \eqref{eq:fullness_condition} yields that $\rank(\tilde{p}) \leq N-1$.
Since the scalar matrix $\tilde{p}$ is selfadjoint, we find a unitary matrix $U$ in $M_N(\C)$ such that
$$\tilde{p}' := U^\ast \tilde{p} U = \begin{pmatrix} 0 & 0\\ 0 & \tilde{p}_{22}' \end{pmatrix} \qquad\text{with $\tilde{p}'_{22} \in M_{N-1}(\C)$}.$$
Thus, according to Lemma \ref{lem:projection_reduction}, the projection $p' := U^\ast p U \in M_N(\M_0)$ enjoys itself a block decomposition of the form
$$p' = \begin{pmatrix} 0 & 0\\ 0 & p'_{22} \end{pmatrix}$$
with a projection $p'_{22} \in M_{N-1}(\M_0)$. Now, due to Proposition \ref{prop:full minor}, there exists a unitary matrix (in fact, a permutation matrix) $V\in M_N(\C)$, such that the matrix $P' := V P U \in M_N(\C\langle x_1,\dots,x_n\rangle)$ enjoys a block decomposition of the form
$$P' = \begin{pmatrix} P_{11}' & P_{12}'\\ P_{21}' & P_{22}' \end{pmatrix}$$
with a full block $P_{22}' \in M_{N-1}(\C\langle x_1,\dots,x_n\rangle)$. Then
$$0 = V (P(X) p) = (V P(X) U) (U^\ast p U) = \begin{pmatrix} P_{11}'(X) & P_{12}'(X)\\ P_{21}'(X) & P_{22}'(X) \end{pmatrix} \begin{pmatrix} 0 & 0\\ 0 & p_{22}' \end{pmatrix} = \begin{pmatrix} 0 & P_{12}'(X) p_{22}'\\ 0 & P_{22}'(X) p_{22}' \end{pmatrix}$$
implies that $P'_{22}(X) p_{22}' = 0$. Due to $A(N-1)$, if follows that $p_{22}'=0$ and hence $p'=0$, so that we obtain finally $p = U p' U^\ast = 0$. Thus, the validity of $A(N)$ is shown.
\end{proof}

In the particular case where the matrices $b_0,b_1,\dots,b_n$ and thus the operator $P(X)$ are selfadjoint, we may conclude from Theorem \ref{thm:main-1} that the analytic distribution $\mu_{P(X)}$ of $P(X)$ cannot have an atom at $0$, i.e., that $\mu_{P(X)}(\{0\})=0$. Under the fullness assumption only, atoms at all other points, however, cannot be excluded; if $P$ is for instance a constant selfadjoint polynomial, i.e., $P = P^\ast \in M_N(\C) \subset M_N(\C\langle x_1,\dots,x_n\rangle)$, then fullness implies that the scalar matrix $P$ is invertible and thus has only non-zero eigenvalues, but its distribution is nonetheless purely atomic with atoms sitting at each eigenvalue of $P$. We thus ask the following questions:
\begin{enumerate}
 \item\label{it:question_1} Under which additional conditions on $P$ can we exclude atoms in $\mu_{P(X)}$?
 \item\label{it:question_2} What happens if we drop the fullness condition?
\end{enumerate}
We will see that under the assumption $\delta^\star(X_1,\dots,X_n)=n$ the positions where atoms appear can be characterized in purely algebraic terms; for that purpose, we give the following definition in the generality of Definition \ref{def:inner-rank_full}.

\begin{definition}
Let $\A$ be a unital complex algebra. For each square matrix $A$ over $\A$, say $A\in M_N(\A)$ for some $N\in\N$, we define
$$\rho^\full_\A(A) := \big\{\lambda \in \C \bigm| \text{$A - \lambda \1_N$ is full over $\A$}\big\},$$
where $\1_N$ stands for the unital element in $M_N(\A)$, and $\sigma^\full_\A(A) := \C \setminus \rho^\full_\A(A)$.
\end{definition}

This definition is clearly modeled according to the familiar notion of resolvent sets and spectra for elements in unital Banach algebras. They show, however, very different properties since the underlying notion of invertibility is here of purely algebraic nature. We specialize our considerations now to the for us relevant case $\A = \C\langle x_1,\dots,x_n\rangle$.

\begin{lemma}\label{lem:full_spectrum_properties}
Let any $P \in M_N(\C\langle x_1,\dots,x_n \rangle)$ of the form $P = b_0 + b_1 x_1 + \cdots + b_n x_n$ with $b_0,b_1,\dots,b_n\in M_N(\C)$ be given. Then the following statements hold:
\begin{enumerate}
 \item\label{it:full_spectrum_properties_1} We have that $\sigma^\full_{\C\langle x_1,\dots,x_n\rangle}(P) \subseteq \sigma(b_0)$, where $\sigma(b_0)$ is the usual spectrum of $b_0$, which consists of all eigenvalues ob $b_0$.
 \item\label{it:full_spectrum_properties_2} If the homogeneous part $P - b_0 = b_1 x_1 + \cdots + b_n x_n$ of $P$ is full, then $\sigma^\full_{\C\langle x_1,\dots,x_n\rangle}(P) = \emptyset$.
\end{enumerate}
\end{lemma}

\begin{proof}
Let $\lambda\in\sigma_{\C\langle x_1,\dots,x_n\rangle}^{\full}(P)$ be given. By definition, this means that $P-\lambda \1_N$ is not full, so that Theorem \ref{thm:hollow and full} guarantees the existence of invertible matrices $U,V\in M_N(\C)$ for which
$$U (P - \lambda \1_N) V = U (b_0 - \lambda \1_N) V + \sum_{j=1}^n (U b_j V) x_j$$
is hollow. Due to linearity, this enforces both $U (b_0 - \lambda \1_N) V$ and $\sum_{j=1}^n (U b_j V) x_j$ to be hollow.
Now, on the one hand, it follows that neither $U (b_0 - \lambda \1_N) V$ nor $b_0 - \lambda \1_N$, thanks to the invertibility of $U$ and $V$, can be invertible; thus, we infer that $\lambda \in \sigma(b_0)$, which shows the validity of \ref{it:full_spectrum_properties_1}.
On the other hand, we see that neither $\sum_{j=1}^n (U b_j V) x_j$ nor $\sum_{j=1}^n b_j x_j$, by the invertibility of $U$ and $V$, can be full; thus, if the homogeneous part of $P$ is assumed to be full, that contradiction rules out the existence of $\lambda\in\sigma_{\C\langle x_1,\dots,x_n\rangle}^{\full}(P)$, which proves \ref{it:full_spectrum_properties_2}. 
\end{proof}

The following statement generalizes Theorem \ref{thm:main-1} and in turn answers question \ref{it:question_2}.

\begin{theorem}\label{thm:main-2}
Suppose that
\begin{enumerate}
 \item $b_0,b_1,\dots,b_n$ are selfadjoint matrices in $M_N(\C)$;
 \item $X_1,\dots,X_n$ are selfadjoint elements in a tracial $W^\ast$-probability space $(\M,\tau)$ that satisfy $$\delta^\star(X_1,\dots,X_n)=n.$$
\end{enumerate}
Then the analytic distribution $\mu_\X$ of
$$\X :=  b_0 + b_1 X_1 + \dots + b_n X_n,$$
seen as an element in the tracial $W^\ast$-probability space $(M_N(\M),\tr_N \circ \tau^{(N)})$, has atoms precisely at the points of
$$\sigma^\full_{\C\langle x_1,\dots,x_n\rangle}(b_0 + b_1 x_1 + \dots + b_n x_n).$$
\end{theorem}

\begin{proof}
We first prove that $\mu_\X$ can have atoms only at the points in $\sigma^\full_{\C\langle x_1,\dots,x_n\rangle}(b_0 + b_1 x_1 + \dots + b_n x_n)$; this is an immediate consequence of Theorem \ref{thm:main-1}: if $\mu_\X$ has an atom at $\lambda$, then we can find a projection $0 \neq p \in M_N(\M_0)$ such that $(\X - \lambda \1_N)p = 0$; hence, the matrix $(b_0 - \lambda \1_N) + b_1 x_1 + \dots + b_n x_n$ cannot be full over $\C\langle x_1,\dots,x_n\rangle$, since otherwise Theorem \ref{thm:main-1} would enforce $p$ to be $0$, which is excluded by our choice of $p$; accordingly, we must have that $\lambda \in \sigma^\full_{\C\langle x_1,\dots,x_n\rangle}(b_0 + b_1 x_1 + \dots + b_n x_n)$, which proves the first part of the assertion.

In order to prove the converse direction, let us take any $\lambda \in \sigma^\full_{\C\langle x_1,\dots,x_n\rangle}(b_0 + b_1 x_1 + \dots + b_n x_n)$. By definition, this means that $(b_0 - \lambda \1_N) + b_1 x_1 + \dots + b_n x_n$ is not full over $\C\langle x_1,\dots,x_n\rangle$ and hence can be written as a product $R_1 R_2$ of an $N \times r$ matrix $R_1$ and an $r \times N$ matrix $R_2$ with entries in $\C\langle x_1,\dots,x_n\rangle$ for some integer $1 \leq r < N$. We may enlarge $R_1$ and $R_2$ to square matrices $\hat{R}_1$ and $\hat{R}_2$, respectively, in $M_N(\C\langle x_1,\dots,x_n\rangle)$ by filling up with zeros as
$$\hat{R}_1 = \begin{pmatrix} R_1 & \0_{N \times (N-r)} \end{pmatrix} \qquad\text{and}\qquad \hat{R}_2 = \begin{pmatrix} R_2 \\ \0_{(N-r) \times N} \end{pmatrix}.$$
Obviously, this does not affect the factorization, i.e., we still have that
$$(b_0 - \lambda \1_N) + b_1 x_1 + \dots + b_n x_n = \hat{R}_1 \hat{R}_2.$$
The latter identity over $\C\langle x_1,\dots,x_n\rangle$ remains valid after evaluation in $X=(X_1,\dots,X_n)$; we thus have that
\begin{equation}\label{eq:main2-factorization}
\X - \lambda \1_N = (b_0 - \lambda \1_N) + b_1 X_1 + \dots + b_n X_n = \hat{R}_1(X) \hat{R}_2(X).
\end{equation}
Now, let us take a look at the operator $\hat{R}_2(X)$. The projection $p_{\ker(\hat{R}_2(X)^\ast)}$ onto the kernel of its adjoint $\hat{R}_2(X)^\ast$ is obviously non-zero as we have that $\hat{R}_2(X)^\ast = \begin{pmatrix} R_2(X)^\ast & \0_{N \times (N-r)} \end{pmatrix}$; more precisely, we have that $(\tr_N \circ \tau^{(N)})(p_{\ker(\hat{R}_2(X)^\ast)}) \geq \frac{N-r}{N}$. From Lemma \ref{lem:kernels}, it thus follows that also  $p_{\ker(\hat{R}_2(X))}$ is non-zero with $(\tr_N \circ \tau^{(N)})(p_{\ker(\hat{R}_2(X))}) \geq \frac{N-r}{N}$. In particular, since the factorization \eqref{eq:main2-factorization} gives that $\ker(\hat{R}_2(X)) \subseteq \ker(\X - \lambda \1_N)$, we see that $p_{\ker(\X - \lambda \1_N)}$ is non-zero with $(\tr_N \circ \tau^{(N)})(p_{\ker(\X - \lambda \1_N)}) \geq \frac{N-r}{N}$. Consequently, $\mu_\X$ has an atom at the given point $\lambda$ with $\mu_\X(\{\lambda\}) \geq \frac{N-r}{N}$.
\end{proof}

\begin{remark}\label{rem:value_entropy_dimension}
The proof of Theorem \ref{thm:main-2} shows in addition that the size of the atom of $\mu_\X$ at any given point $\lambda \in \sigma^\full_{\C\langle x_1,\dots,x_n\rangle}(P)$, where we abbreviate $P := b_0 + b_1 x_1 + \dots + b_n x_n$, can be controlled by
\begin{equation}\label{eq:entropy_dimension_bound}
\mu_\X(\{\lambda\}) = (\tr_N \circ \tau^{(N)})(p_{\ker(\X - \lambda \1_N)}) \geq \frac{1}{N}\big(N - \rho(P-\lambda \1_N)\big),
\end{equation}
where $\rho(P-\lambda \1_N)$ denotes the inner rank of the non-full matrix $P-\lambda \1_N$ as introduced in Definition \ref{def:inner-rank_full}. According to \eqref{eq:entropy_dimension_single_operator}, the non-microstates free entropy dimension $\delta^\ast(\X)$ of the selfadjoint noncommutative random variable $\X = P(X_1,\dots,X_n)$ in $(M_N(\M),\tr_N \circ \tau^{(N)})$ can be estimated as
$$\delta^\ast(\X) = 1 - \sum_{\lambda \in \sigma^\full_{\C\langle x_1,\dots,x_n\rangle}(P)} \mu_\X(\{\lambda\})^2 \leq 1 - \frac{1}{N^2} \sum_{\lambda \in \sigma^\full_{\C\langle x_1,\dots,x_n\rangle}(P)} \big(N - \rho(P-\lambda\1_N)\big)^2.$$
Later, in Corollary \ref{cor:full_entropy_dimension_implies_Atiyah}, we will see that $\rho(P-\lambda\1_N) = \rank(\X - \lambda \1_N)$. From this, we infer $$\rho(P-\lambda\1_N) = N (\tr_N \circ \tau^{(N)})(p_{\overline{\im(\X-\lambda\1_N)}}) = N \big(1 - (\tr_N \circ \tau^{(N)})(p_{\ker(\X - \lambda \1_N)})\big),$$
so that in fact equality holds in \eqref{eq:entropy_dimension_bound}; consequently, we can also improve our previous estimate for $\delta^\ast(\X)$ to
$$\delta^\ast(\X) = 1 - \frac{1}{N^2} \sum_{\lambda \in \sigma^\full_{\C\langle x_1,\dots,x_n\rangle}(P)} \big(N - \rho(P-\lambda\1_N)\big)^2,$$
which expresses the non-microstates free entropy dimension $\delta^\ast(\X)$ in terms of purely algebraic quantities associated to $P$.
\end{remark}

Finally, we observe that the situation becomes particularly nice when fullness is imposed on the purely linear part $b_1 x_1 + \dots + b_n x_n$ of $b_0 + b_1 x_1 + \dots + b_n x_n$; this provides the answer to our question \ref{it:question_1}.

\begin{theorem}\label{thm:main-3}
Suppose that
\begin{enumerate}
 \item $b_1,\dots,b_n$ are selfadjoint matrices in $M_N(\C)$ for which
 $$b_1 x_1 + \dots + b_n x_n \in M_N(\C\langle x_1,\dots,x_n\rangle)$$
 is full over $\C\langle x_1,\dots,x_n\rangle$; let $b_0\in M_N(\C)$ be any other selfadjoint matrix;
 \item $X_1,\dots,X_n$ are selfadjoint elements in a tracial $W^\ast$-probability space $(\M,\tau)$ that satisfy $$\delta^\star(X_1,\dots,X_n)=n.$$
\end{enumerate}
Then the analytic distribution $\mu_\X$ of
$$\X :=  b_0 + b_1 X_1 + \dots + b_n X_n,$$
seen as an element in the tracial $W^\ast$-probability space $(M_N(\M),\tr_N \circ \tau^{(N)})$, has no atoms.
\end{theorem}

\begin{proof}
This relies crucially on Lemma \ref{lem:full_spectrum_properties} \ref{it:full_spectrum_properties_2}: if the homogeneous part $b_1 x_1 + \dots + b_n x_n$ is full, then $\sigma^\full_{\C\langle x_1,\dots,x_n\rangle}(b_0 + b_1 x_1 + \dots + b_n x_n) = \emptyset$. Therefore, according to Theorem \ref{thm:main-2}, the measure $\mu_\X$ cannot have atoms.
\end{proof}

\section{Noncommutative Rational functions and rational closure}

In this section, we will give an introduction to noncommutative rational functions. One crucial fact in this section is that, for each rational function, we can associate it with some representation using full matrices over linear polynomials: this is known as ``linearization trick''. Hence Theorem \ref{thm:main-1} can be applied to rational functions, which will allow us to exclude zero divisors for rational functions under the same assumption as in Section 4. Furthermore, a construction called rational closure is also introduced; this will play an important role when we consider the Atiyah property in the next section.

\subsection{Noncommutative rational expressions and rational functions}

A noncommutative rational expression, intuitively speaking, is obtained by taking repeatedly sums, products and inverses, starting from scalars and some formal non-commuting variables, without taking care about possible cancellations or resulting mathematical inconsistencies. For example, we allow $0^{-1}$ and $(x-x)^{-1}$ as valid and different expressions, though they don't make any sense when we try to treat them as functions. (We will take care of this problem, when we talk about the domain of such expressions.) A formal definition for rational expressions can be achieved in terms of graphs as follows.

\begin{definition}
A \emph{rational expression} in variables $\{x_{1},\dots,x_{n}\}$ is a finite directed acyclic graph, i.e., a finite directed graph with no directed cycles, with labels on in-degree zero vertices and some edges, satisfying the following rules:
\begin{enumerate}
\item the in-degree zero vertices, i.e., vertices having no edge directed to them, are chosen from the set $\C$ and $\{x_{1},\dots,x_{n}\}$, that is, each vertex of in-degree zero can be identified with a complex number or some variable in $\{x_{1},\dots,x_{n}\}$; all other vertices are of in-degree $1$ or $2$;
\item there is only one vertex that has out-degree zero, i.e., there is no edge directed from it; all other vertices are of out-degree $1$;
\item for any pair of edges directed to some vertex of in-degree $2$, they can be additionally labelled by \emph{left} and \emph{right}.
\end{enumerate}
\end{definition}

The definition is more or less self-explanatory, that is, for any given ``rational expression'', we can read it by the above language: the variables and coefficients in the expression are given by some vertices of in-degree zero; for each $^{-1}$ applied to a vertex (which represents some variable or number), we add an directed edge from it to a new vertex without any label; for each $+$ applied to two vertices, we add two directed edges from them to a new vertex without any labels; for each $\times$ applied to two vertices, we add two directed edges from them to a new vertex with left and right labels to determine the order of multiplication; we proceed in such a way until we arrive at a vertex which corresponds to the desired ``rational expression''. For example, the rational expression $y\left(xy\right)^{-1}x$ is given by the following graph:
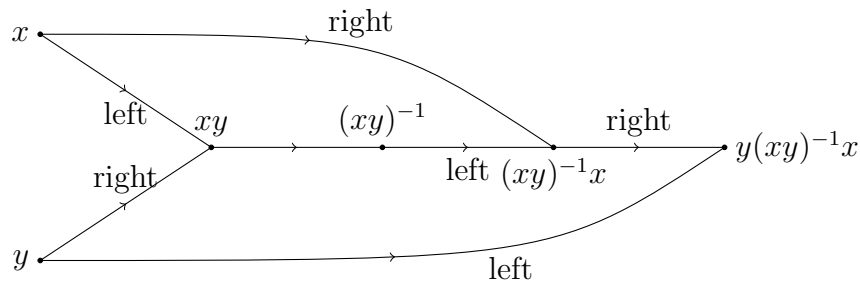
\begin{center}
\begin{tikzpicture}
  \draw [fill] (0,0) circle [radius=0.03];
  \node [left] at (0,0) {$x$};
  \draw [fill] (0,-3) circle [radius=0.03];
  \node [left] at (0,-3) {$y$};
  \draw [fill] (2.25,-1.5) circle [radius=0.03];
  \node [above] at (2.25,-1.5) {$xy$};
  \draw [fill] (4.5,-1.5) circle [radius=0.03];
  \node [above] at (4.5,-1.5) {$(xy)^{-1}$};
  \draw [fill] (6.75,-1.5) circle [radius=0.03];
  \node [below] at (6.75,-1.5) {$(xy)^{-1}x$};
  \draw [fill] (9,-1.5) circle [radius=0.03];
  \node [right] at (9,-1.5) {$y(xy)^{-1}x$};
  \draw [->-=.5] (0,0) to node [below] {left} (2.25,-1.5) ;
  \draw [->-=.5] (0,-3) to node [above] {right} (2.25,-1.5);
  \draw [->-=.5] (2.25,-1.5) to (4.5,-1.5);
  \draw [->-=.5] (4.5,-1.5) to node [below] {left} (6.75,-1.5);
  \draw [->-=.5] (6.75,-1.5) to node [above] {right} (9,-1.5);
  \draw [->-=.5] (0,0) .. controls (4.5,0) .. node [above] {right} (6.75,-1.5);
  \draw [->-=.5] (0,-3) .. controls (6.75,-3) .. node [below] {left} (9,-1.5);
\end{tikzpicture}
\end{center}

With this definition, taking sums, products and inverses of rational expression are clear: we adjoin two rational expression by adding two edges from their unique vertices of out-degree zero to a new vertex, then the resulting new graph is the sum, or product if these two edges are labelled by left and right; if we add a new edge from the vertex of out-degree zero to a new vertex, then the resulting graph is the inverse. This definition of rational expressions is known as \emph{circuits}, or \emph{noncommutative arithmetic circuits with division}. We refer to \cite{HW15} and the references collected therein for this notion and related topics.

Our goal in this subsection is to introduce the noncommutative rational functions, which should be some smallest division ring (or skew field) containing noncommutative polynomials; similar to the case of commutative polynomials, or any integral domain in general. So it may be tempting to imitate the construction as the commutative case; that is, every commutative rational function can be written in the form $pq^{-1}$, where $p$ and $q$ are polynomials. So we may hope that this also holds for the noncommutative case. But this doesn't work any more; for example, considering a function such as $xy^{-1}x$, where $x$ and $y$ are viewed as noncommutative polynomials, there is no way to put this in the form $pq^{-1}$, simply due to the noncommutativity of the variables. Therefore, in the noncommutative case there is no hope to represent a rational function just by two polynomials.

Then it is natural to go back to the notion of rational expressions given at the beginning of this section. A basic idea is to view rational functions as equivalence classes of rational expressions. But then we also need to identify the rational expressions which are trivial
or mathematically inconsistent. For example, as a rational function, we should have $y\left(xy\right)^{-1}x=yy^{-1}x^{-1}x=1$, since each non-zero polynomial is invertible as a rational function; so the rational expression $1-y(xy)^{-1}x$ actually represents the zero function and thus the rational expression $(1-y(xy)^{-1}x)^{-1}$ represents nothing
meaningful. Actually, there can be quite complicated rational expressions  to represent a simple rational function; for example,
\[(x-y^{-1})^{-1}-x^{-1}-(xyx-x)^{-1}\]
also can be reduced to zero by arithmetic operations though it may not be obvious. A way to overcome this difficulty is to define the equivalence classes by evaluations.

\begin{definition}
\label{def:evaluation of expr}Let $\A$ be any unital algebra. For any noncommutative rational expression $r$ in variables $\{x_{1},\dots,x_{n}\}$, we define its $\A$\emph{-domain} $\dom_{\A}(r)$ together with its \emph{evaluation} $\ev_{X}(r)$ for any $X=(X_{1},\dots,X_{n})\in\dom_{\A}(r)$ by the following rules:
\begin{enumerate}
\item For any $\lambda\in\C$, we put $\dom_{\A}(r)=\A^{n}$ and $\ev_{X}(\lambda)=\lambda1$, where $1$ is the unit of algebra $\A$;
\item For $i=1,\dots,n$, we put $\dom_{\A}(x_{i})=\A^{n}$ and $\ev_{X}(x_{i})=X_{i}$;
\item For two rational expressions $r_{1}$, $r_{2}$, we have
\[\dom_{\A}(r_{1}\cdot r_{2})=\dom_{\A}(r_{1}+r_{2})=\dom_{\A}(r_{1})\cap\dom_{\A}(r_{2})\]
and
\[\begin{array}{l}\ev_{X}(r_{1}\cdot r_{2})=\ev_{X}(r_{1})\cdot\ev_{X}(r_{2}),\\\ev_{X}(r_{1}+r_{2})=\ev_{X}(r_{1})+\ev_{X}(r_{2});\end{array}\]
\item For a rational expression $r$, we have
\[\dom_{\A}(r^{-1})=\{X\in\dom_{\A}(r)\bigm| \ev_{X}(r)\text{ is invertible in }\A\}\]
and
\[\ev_{X}(r^{-1})=\ev_{X}(r)^{-1}.\]
\end{enumerate}
We also abbreviate $r(X):=\ev_{X}(r)$ for any given rational expression $r$ and $X\in\dom_{\A}(r)$.
\end{definition}

Therefore, when a unital algebra $\A$ is given, equivalence classes of rational expressions can be defined by an equivalence relation: two rational expressions $r_{1}$ and $r_{2}$ are called \emph{$\A$-evaluation equivalent} if $\dom_{\A}(r_{1})\cap\dom_{\A}(r_{2})\neq\emptyset$ and $r_{1}(X)=r_{2}(X)$ for all $X\in\dom_{\A}(r_{1})\cap\dom_{\A}(r_{2})$. Then it remains to choose appropriate $\A$ such that rational functions can be well-defined as $\A$-evaluation equivalent classes of rational expressions. This approach was first achieved by Amitsur in his paper \cite{Ami66} by evaluating rational expressions on some large auxiliary skew field; it turns out that the evaluation on matrices of all sizes is also sufficient, which was proved in \cite{KV12}.

Besides the difficulty to construct such a skew field of rational functions containing polynomials as a subring, there is another significant difference between the noncommutative and commutative cases. That is, noncommutative polynomials actually can be embedded into more than one skew fields that are non-isomorphic; see \cite{KV12} for some examples. However, it turns out that there exists an unique skew field which has some ``universality''. In order to make this notion precise, we take some definitions from \cite[Section 7.2]{Coh06}.

\begin{definition}
Let $\mathcal{R}$ be a ring.
\begin{enumerate}
\item An $\mathcal{R}$-\emph{ring} (respectively, \emph{$\mathcal{R}$-field}) is a ring (respectively, (skew) field) $\mathcal{K}$ together with a homomorphism $\phi:\mathcal{R}\rightarrow\mathcal{K}$.
\item An $\mathcal{R}$-field $\mathcal{K}$ is called \emph{epic} if there is no proper subfield of $\mathcal{K}$ containing the image $\phi(\mathcal{R})$.
\item An epic $\mathcal{R}$-field $\mathcal{K}$, for which $\phi$ is injective, is called\emph{ field of fractions of} $\mathcal{R}$.
\end{enumerate}
\end{definition}

Since we want to compare different $\mathcal{R}$-fields, it is natural to consider the homomorphisms between $\mathcal{R}$-rings which respect the $\mathcal{R}$-ring structure. That is, for a homomorphism $f:\mathcal{K}\rightarrow\mathcal{L}$ between two $\mathcal{R}$-rings $\mathcal{K}$ and $\mathcal{L}$ with homomorphisms $\phi_{\mathcal{K}}:\mathcal{R}\rightarrow\mathcal{K}$ and $\phi_{\mathcal{L}}:\mathcal{R}\rightarrow\mathcal{L}$, if $f\circ\phi_{\mathcal{K}}=\phi_{\mathcal{L}}$, then we say $f$ is an $\mathcal{R}$-\emph{ring homomorphism}. However, this requirement enforces $f$ to be an isomorphism whenever $\mathcal{K}$ and $\mathcal{L}$ are two epic $\mathcal{R}$-fields. Hence we need to consider more general maps.

\begin{definition}
Let $\mathcal{K}$ and $\mathcal{L}$ be $\mathcal{R}$-fields. A \emph{specialization} from $\mathcal{K}$ to $\mathcal{L}$ is an $\mathcal{R}$-ring homomorphism $f:\mathcal{K}_{f}\rightarrow\mathcal{L}$, where $\mathcal{K}_{f}$ is a minimal subring of $\mathcal{K}$ satisfying
\begin{itemize}
\item $\mathcal{K}_{f}$ contains the image of $\mathcal{R}$,
\item all elements of $\{x\in\mathcal{K}_{f}\bigm|f(x)\neq0\}$ are invertible in $\mathcal{K}_{f}$.
\end{itemize}
\end{definition}

This definition is slightly modified from the one in \cite[Section 7.2]{Coh06} for simplicity. With the help of specializations we can now clarify a universal property for epic $\mathcal{R}$-fields.

\begin{definition}
An epic $\mathcal{R}$-field $\mathcal{U}$ is called a \emph{universal $\mathcal{R}$-field} if for any epic $\mathcal{R}$-field $\mathcal{K}$ there is a unique specialization $\mathcal{U}\rightarrow\mathcal{K}$. If $\mathcal{U}$ is in addition a field of fractions of $\mathcal{R}$, then we call $\mathcal{U}$ the \emph{universal field of fractions of} $\mathcal{R}$.
\end{definition}

In other words, an epic $\mathcal{R}$-field $\mathcal{U}$ is universal if for any other epic $\mathcal{R}$-field $\mathcal{K}$, the corresponding $\phi_{\mathcal{K}}$ factorizes through a specialization $f$ from $\mathcal{U}$ to $\mathcal{K}$, i.e.,
\begin{center}
\begin{tikzpicture}
  \node (R) at (0,0) {$\mathcal{R}$};
  \node (U) at (2,0) {$\mathcal{U}$};
  \node (K) at (2,-2) {$\mathcal{K}$.};
  \draw[->] (R) to node [above] {$\phi_{\mathcal{U}}$} (U);
  \draw[->] (R) to node [left] {$\phi_{\mathcal{K}}$} (K);
  \draw[->] [dotted] (U) to node [right] {$f$} (K);
\end{tikzpicture}
\end{center}
Actually, in this case, $\ker f$ is a maximal ideal of $\mathcal{U}_{f}$ and hence by the definition of $\mathcal{U}_{f}$, $\mathcal{U}_{f}/\ker f$ is a field, isomorphic to a subfield of $\mathcal{K}$ containing $\phi_{\mathcal{K}}(\mathcal{R})$; so if $\mathcal{K}$ is epic, then this field is actually isomorphic to $\mathcal{K}$. Therefore, from a universal $\mathcal{R}$-field one can obtain any other epic $\mathcal{R}$-field by a specialization; and by this universal property a universal $\mathcal{R}$-field, if it exists, is unique up to isomorphism.

In our particular case, though it is highly non-trivial, the universal field of fractions of $\C\left\langle x_{1},\dots,x_{n}\right\rangle $ indeed exists; it is denoted by $\C\plangle x_{1},\dots,x_{n}\prangle$, and sometimes it is also simply called the \emph{free (skew) field}. We have already mentioned two ways of constructing the free field, by evaluating rational expressions on some auxiliary algebras; yet there is also another approach to construct the free field by generalizing the idea of \emph{localization} to the non-commutative case (see \cite[Chapter 7]{Coh06} for details). Recall that, for a commutative unital ring $\mathcal{R}$ and a given set $S\subseteq\mathcal{R}$ which is closed under multiplication and contains $1$, localization allows us to construct another ring $\mathcal{R}_{S}$ together with a homomorphism $\phi:\mathcal{R}\rightarrow\mathcal{R}_{S}$ such that all elements in the image $\phi(\mathcal{R})$ are invertible in $\mathcal{R}_{S}$. Cohn discovered that one can replace the set $S$ by a set of matrices $\Sigma$ over $\mathcal{R}$ and construct a \emph{universal localization} $\mathcal{R}_{\Sigma}$, that is, a ring with a homomorphism $\phi:\mathcal{R}\rightarrow\mathcal{R}_{\Sigma}$ such that all elements in the image $\phi(\Sigma)$ are invertible as matrices over $\mathcal{R}_{\Sigma}$, and any other ring with such a homomorphism can be factorized through $\mathcal{R}_{\Sigma}$. Moreover, if we take the set $\Sigma$ to be the set of all full matrices over $\mathcal{R}$ and if $\Sigma$ satisfies some ``multiplicative closure'' property, then this universal localization $\mathcal{R}_{\Sigma}$ turns out to be the universal field of fractions of $\mathcal{R}$. Actually, in \cite[Theorem 7.5.13]{Coh06}, Cohn gives a list of characterizations for rings that can be embedded into universal fields of fractions. With the help of our lemmas from the appendix we can check one of those characterizations for our ring of non-commutative polynomials.

\begin{lemma}
Let $\Sigma$ be the set of full matrices over $\C\left\langle x_{1},\dots,x_{n}\right\rangle $. Then $\Sigma$ it is lower multiplicative, i.e., $\1\in\Sigma$ and
\[\begin{pmatrix}A & \0\\C & B\end{pmatrix}\in\Sigma\]
for all $A,B\in\Sigma$ and each matrix $C$ over $\C\left\langle x_{1},\dots,x_{d}\right\rangle $ of appropriate size.
\end{lemma}

\begin{proof}
From Lemma \ref{lem:product and sum}, we know that the diagonal sum of two full matrices is full again, so it's clear that $\1\in\Sigma$ and
\[\begin{pmatrix}A & \0\\\0 & B\end{pmatrix}\in\Sigma\]
for any two full matrices $A\in M_{k}(\C\left\langle x_{1},\dots,x_{n}\right\rangle )$ and $B\in M_{l}(\C\left\langle x_{1},\dots,x_{n}\right\rangle )$. Actually, the argument in the proof of Lemma \ref{lem:product and sum} can be applied directly to
\[\begin{pmatrix}A & \0\\C & B\end{pmatrix},\]
since the left lower block $C$ doesn't play any role therein; hence this matrix is also full.
\end{proof}

So Theorem 7.5.13 of \cite{Coh06} yields the following.

\begin{theorem}
Let $\Sigma$ be the set of full matrices over $\C\left\langle x_{1},\dots,x_{n}\right\rangle $, then the universal localization for $\Sigma$ is the universal field of fractions $\C\plangle x_{1},\dots,x_{n}\prangle$
of $\C\left\langle x_{1},\dots,x_{n}\right\rangle $. Moreover, the inner rank of a matrix over $\C\left\langle x_{1},\dots,x_{n}\right\rangle $ stays the same if the matrix is considered as a matrix over $\C\plangle x_{1},\dots,x_{n}\prangle$.
\end{theorem}

Therefore, in the following, for a matrix $A$ over polynomials, we do not need to distinguish between its inner rank over polynomials and its inner rank over rational functions; this common inner rank is denoted by $\rho\left(A\right)$.

\subsection{\label{sec:linearization}Linearization for rational functions}

This localization in the last subsection tells us that a full matrix $A$ over $\C\left\langle x_{1},\dots,x_{n}\right\rangle $ is invertible as a matrix over $\C\plangle x_{1},\dots,x_{n}\prangle$. So each entry in $A^{-1}$ is a \emph{rational function}, i.e., an element in the free field $\C\plangle x_{1},\dots,x_{n}\prangle$. Therefore, for any row vector $u$ and any column vector $v$ over $\C$, $uA^{-1}v$ is a rational function as it is a linear combination of some rational functions. Actually, in the construction of the universal localization, we add new elements more or less in this way to extend $\C\left\langle x_{1},\dots,x_{n}\right\rangle$ to its universal localization, which turns out to be the free field $\C\plangle x_{1},\dots,x_{n}\prangle$; so we can expect the converse should also be true, that is, any rational function $r$ can be written in the form $r=uA^{-1}v$, by some full matrix $A$ over polynomials with two scalar-valued vectors $u$ and $v$.
Moreover, this matrix $A$ can be chosen to be linear, though the dimension of $A$ may increase for exchange. This culminates in the following definition borrowed from \cite{CR99}.

\begin{definition}\label{def:linear representation}
Let $r$ be a rational function. A \emph{linear representation of} $r$ is a tuple $\rho=(u,A,v)$ consisting of a linear full matrix $A\in M_{k}(\C\left\langle x_{1},\dots,x_{n}\right\rangle )$, a row vector $u\in M_{1,k}(\C)$ and a column vector $M_{k,1}(\C)$ such that $r=uA^{-1}v$.
\end{definition}

In \cite{CR99}, such linear representations were used to give an alternative construction of the free field. That indeed each element in the free field admits a linear representation is a fundamental result, which is a direct consequence of the approach of \cite{CR99}, but follows also from the general theory presented in \cite{Coh06}; see also \cite{Vol18}.

\begin{theorem}\label{thm:linear representation}
Each rational function $r \in \C\plangle x_1,\dots,x_n\prangle$ admits a linear representation in the sense of Definition \ref{def:linear representation}.
\end{theorem}

The idea of realizing rational noncommutative functions by inverses of linear matrices has been known for more than fifty years; and was rediscovered several times in many distant realms, such as automaton theory and non-commutative rational series, and many other branches of mathematics as well as computer science and engineering. Under the name ``linearization trick'', it was introduced to the community of free probability by the work of Haagerup and Thorbj\o rnsen \cite{HT05} and Haagerup, Schultz, and Thorbj\o rnsen \cite{HST06}, building on earlier operator space versions; for the latter see in particular the work of Pisier \cite{Pis18}.

For the special case of noncommutative polynomials, similar concepts were developed by Anderson \cite{And12,And13,And15} and were used in \cite{BMS17} in order to study evaluations of noncommutative polynomials in noncommutative random variables by means of operator-valued free probability theory. Later, in \cite{HMS18}, these methods were generalized to noncommutative rational expressions, based on the following variant of Definition \ref{def:linear representation}; it is taken from \cite[Section 5]{HMS18}, but with the sign changed for convenience.

\begin{definition}
Let $r$ be a rational expression in variables $x_{1},\dots,x_{n}$. A \emph{formal linear representation $\rho=(u,A,v)$ of $r$} consists of a linear matrix $A$ over $\C\left\langle x_{1},\dots,x_{n}\right\rangle $, a row vector $u$ and a column vector $v$ over $\C$ such that for any unital algebra $\A$,
\[\dom_{\A}(r)\subseteq\{X\in\A^{n}\bigm|A\left(X\right)\text{ is invertible in }\A\}\]
and
\[r(X)=uA(X)^{-1}v\]
for any tuple $X\in\dom_{\A}(r)$.
\end{definition}

The following explicit algorithm that was stated in \cite[Section 5]{HMS18} establishes the existence of a formal linear representation for any rational expression, thus yielding a perfect analogue of Theorem \ref{thm:linear representation}. 

\begin{algorithm}\label{linearization}
A formal linear representation $\rho=(u,A,v)$ of a rational expression $r$ can be constructed by using successively the following rules:
\begin{enumerate}
\item For scalars $\lambda\in\C$ and the variables $x_{j}$, $j=1,\dots,n$, formal linear representations are given by
\[\rho_{\lambda}:=\left(\begin{pmatrix}0 & 1\end{pmatrix},\begin{pmatrix}-\lambda & 1\\1 & 0\end{pmatrix},\begin{pmatrix}0\\1\end{pmatrix}\right),\]
and
\[\rho_{x_{j}}:=\left(\begin{pmatrix}0 & 1\end{pmatrix},\begin{pmatrix}-x_{j} & 1\\1 & 0\end{pmatrix},\begin{pmatrix}0\\1\end{pmatrix}\right).\]
\item If $\rho_{1}=(u_{1},A_{1},v_{1})$ and $\rho_{2}=(u_{2},A_{2},v_{2})$ are two formal linear representations for rational expressions $r_{1}$ and $r_{2}$, respectively, then
\[\rho_{1}\oplus\rho_{2}:=\left(\begin{pmatrix}u_{1} & u_{2}\end{pmatrix},\begin{pmatrix}A_{1} & \0\\\0 & A_{2}\end{pmatrix},\begin{pmatrix}v_{1}\\v_{2}\end{pmatrix}\right)\]
gives a formal linear representation of $r_{1}+r_{2}$ and
\[\rho_{1}\odot\rho_{2}:=\left(\begin{pmatrix}\0 & u_{1}\end{pmatrix},\begin{pmatrix}-v_{1}u_{2} & A_{1}\\A_{2} & \0\end{pmatrix},\begin{pmatrix}\0\\v_{2}\end{pmatrix}\right)\]
gives a formal linear representation of $r_{1}\cdot r_{2}$.
\item If $\rho=(u,A,v)$ is a formal linear representation for rational expression $r$, then
\[\rho^{-1}:=\left(\begin{pmatrix}1 & \0\end{pmatrix},\begin{pmatrix}0 & u\\v & A\end{pmatrix},\begin{pmatrix}1\\\0\end{pmatrix}\right)\]
gives a formal linear representation of $r^{-1}$.
\end{enumerate}
\end{algorithm}

Our main interest in this section is in rational functions rather than rational expressions and actually we don't really need to deal with rational expressions through the paper. So we will not say more on this algorithm here, but a detailed proof can be found in \cite[Section 5]{HMS18} or \cite[Chapter III]{Mai17}.
We only want to highlight that, due to their excellent evaluation properties, formal linear representations are closely related to linear representations as introduced in Definition \ref{def:linear representation}.
Of course, as we just consider rational expressions in the above algorithm, it may happen that the linear matrix $A$ is not full, since rational expressions like $0^{-1}$ are allowed. However, for rational expressions that ``represent'' rational functions, their formal linear representations automatically produce linear matrices $A$ that are full; this is explained in \cite[Chapter III]{Mai17}.
Indeed, if a rational function is seen like in \cite{KV12} as an equivalence class of regular rational expressions with respect to matrix evaluation equivalence (where a rational expression $r$ is said to be \emph{regular} if $\dom_{M_n(\C)}(r) \neq \emptyset$ holds for at least one $n\in\N$), then any formal linear representation $\rho=(u,A,v)$ of any of its representatives $r$ carries a full matrix $A$, because $A(X)$ is due to the defining property of $\rho$ an invertible matrix for each $X\in \dom_{M_n(\C)}(r)$. In this way, one recovers the the fundamental result Theorem \ref{thm:linear representation} on the existence of linear representations for rational functions.

\subsection{Evaluation of rational functions}\label{sec:evaluation}

Let $X=(X_{1},\dots,X_{n})$ be a tuple of elements in a unital algebra $\A$, then its evaluation map $\ev_{X}$ from $\C\left\langle x_{1},\dots,x_{n}\right\rangle $ to $\A$ is well-defined as a homomorphism. We have also seen that the evaluation of rational expressions can be defined naturally with $\A$-domains considered in Definition \ref{def:evaluation of expr}. Then the question is how can we define the evaluation for rational functions. However, unfortunately, the evaluation can not be well-defined for all algebras without additional assumptions. Here is an example which illustrates the problem: considering $\A=B\left(H\right)$ for some infinite dimensional separable Hilbert space, let $l$ denote the one-sided left-shift operator, then $l^{\ast}$ is the right-shift operator and we have $l\cdot l^{\ast}=1$ but $l^{\ast}\cdot l\neq1$; so it is clear that the evaluation of the rational expression $r(x,y)=y\left(xy\right)^{-1}x$ is $r(l,l^{\ast})=l^{\ast}l\neq1$; however, since this rational expression also represents the rational function $1$ there is no consistent way to define its value for the arguments $l$ and $l^*$. So it's natural to consider algebras in which a left inverse is also a right inverse to avoid such a problem; actually, we require algebras to be stably finite in order to make sure that we have a well-defined evaluation.

\begin{theorem}
\label{thm:evaluation}Let $\A$ be a stably finite algebra, then for any rational function $r$ in the free field $\C\plangle x_{1},\dots,x_{n}\prangle$, we have a well-defined $\A$-domain $\dom_{\A}(r)\subseteq\A^{n}$ and an evaluation $r(X)$ for any $X\in\dom_{\A}(r)$.
\end{theorem}

Actually, the converse also holds in some sense, see Theorem 7.8.3 in the book \cite{Coh06}, which is stated in some other terminologies. When rational functions are treated as equivalence classes of rational expressions evaluated on matrices of all sizes, see also \cite[Theorem 6.1]{HMS18} for a proof of the same theorem. For reader's convenience, here we give a proof under our setting.

\begin{definition}
\label{def:evaluation}For a linear representation $\rho=(u,A,v)$, we define its $\A$\emph{-domain}
\[\dom_{\A}(\rho)=\{X\in\A^{n}\bigm|A(X)\text{ is invertible as a matrix over }\A\};\]
and for a given rational function $r$, we define its $\A$\emph{-domain}
\[\dom_{\A}(r)=\bigcup_{\rho}\dom_{\A}(\rho),\]
where the union is taken over all possible linear representations of $r$. Then we define the \emph{evaluation} of $r$ at a tuple $X\in\dom_{\A}(r)$ by
\[\Ev_{X}(r)=r(X)=uA(X)^{-1}v\]
for any linear representation $\rho=(u,A,v)$ satisfying $X\in\dom_{\A}(\rho)$.
\end{definition}

Of course, as the choice of the linear representations for a rational function is not unique, we have to prove that different choices always give the same evaluation.

\begin{proof}[Proof of Theorem \ref{thm:evaluation}]
Let $\rho_{1}=(u_{1},A_{1},v_{1})$ and $\rho_{2}=(u_{2},A_{2},v_{2})$ be two linear representations of a rational function $r$ such that
\[r=u_{1}A_{1}^{-1}v_{1}=u_{2}A_{2}^{-1}v_{2}.\]
We need to prove that for any $X\in\dom_{\A}(\rho_{1})\cap\dom_{\A}(\rho_{2})$, we have $u_{1}A_{1}(X)^{-1}v_{1}=u_{2}A_{2}(X)^{-1}v_{2}$. It is not difficult to verify that the tuple
\[\left(\begin{pmatrix}u_{1} & u_{2}\end{pmatrix},\begin{pmatrix}A_{1} & \0\\\0 & -A_{2}\end{pmatrix},\begin{pmatrix}v_{1}\\v_{2}\end{pmatrix}\right)\]
is a linear representation of zero in free field; hence it suffices to prove that, for any linear representation $\rho=(u,A,v)$ of zero, we have $uA(X)^{-1}v=0$ for any $X\in\dom_{\A}(A)$. Now suppose that $uA(X)^{-1}v\neq0$ for some linear representation $\rho=(u,A,v)$ of the zero function; then
\[\begin{pmatrix}0 & u\\v & A(X)\end{pmatrix}\in M_{k+1}(\A)\]
has inner rank $k+1$ over $\A$ by Proposition \ref{prop:invertible minor} as $\A$ is stably finite. However, this is impossible: by the embedding of polynomials into rational functions, we see that $\C\left\langle x_{1},\dots,x_{n}\right\rangle $ is stably finite; so we can apply the same proposition to show that
\[\begin{pmatrix}0 & u\\v & A\end{pmatrix}\in M_{k+1}(\C\left\langle x_{1},\dots,x_{n}\right\rangle )\]
has inner rank $k$ as $uA^{-1}v=0$; and thus it has a rank factorization over $\C\left\langle x_{1},\dots,x_{n}\right\rangle $, which leads to the same factorization over $\A$ as the evaluation of polynomials is always well-defined as a homomorphism.
\end{proof}

We close this subsection by remarking that this definition of evaluation is consistent with the usual notion of evaluation. That is, given any polynomial $p$, in order to see that the above definition coincides with the usual one, we should find a linear representation $\rho=(u,A,v)$ such that $uA(X)^{-1}v$ equals $p(X)$, the usual evaluation of $p$ at $X$, for any $X\in\A^{n}$; and actually such a linear representation can be constructed by following the first two rules in Algorithm \ref{linearization} for formal linear representations. Furthermore, from the last rules of this algorithm, we can also see that the arithmetic operations between rational functions give the corresponding arithmetic operations between their evaluations.

\subsection{\label{sec:rational closure}Rational closure}

In this subsection, we introduce another construction besides rational functions, which is based on similar idea as localizations; but it will allow us to consider some general situations when we study Atiyah properties in the next section.

\begin{definition}\label{def:rational closure}
Let $\phi:\C\left\langle x_{1},\dots,x_{n}\right\rangle \rightarrow\A$ be a homomorphism into a unital algebra $\A$ and let us denote by $\Sigma_{\phi}$ the set of all matrices whose images are invertible under the matricial amplifications of $\phi$, i.e.,
\[\Sigma_{\phi}=\bigcup_{k=1}^{\infty}\{A\in M_{k}(\C\left\langle x_{1},\dots,x_{n}\right\rangle )\bigm|\phi^{(k)}(A)\text{ is invertible in }M_{k}(\A)\}.\]
The \emph{rational closure of} $\C\left\langle x_{1},\dots,x_{n}\right\rangle $ with respect to $\phi$, denoted by $\mathcal{R}_{\phi}$, is the set of all entries of inverses of matrices in the image of $\Sigma_{\phi}$ under $\phi$.
\end{definition}

We actually only consider the case when $\phi$ is given by the evaluation of some tuple of elements in $\A$, in the next section. For discussion of more general cases, see \cite[Section 7.1]{Coh06}.

\begin{lemma}
(See \cite[Proposition 7.1.1 and Theorem 7.1.2]{Coh06}) The rational closure $\mathcal{R}_{\phi}$ for any given homomorphism $\phi$ is a subalgebra of $\A$ containing the image of $\C\left\langle x_{1},\dots,x_{n}\right\rangle $.
\end{lemma}

By the definition of rational closures, we can see that for each element $r\in\mathcal{R}_{\phi}$, there is a $k\times k$ matrix $A$ whose image $\phi^{(k)}(A)$ is invertible such that $r$ is a entry of $\phi^{(k)}(A)^{-1}$; hence we can choose some scalar-valued row and column vectors $u$ and $v$ such that
\[r=u\phi^{(k)}(A)^{-1}v.\]
Therefore, the proof of this lemma can go the same way as Algorithm \ref{linearization}, though here no linearity is involved. So we refer to \cite[Section 5]{HMS18} for a detailed proof which can easily be adapted to our setting.

Unlike the rational functions, the rational closure is not a division ring in general. But it has a nice property about inverses: if an element $r\in\mathcal{R}_{\phi}$ is invertible in $\A$, then $r^{-1}\in\mathcal{R}_{\phi}$. Actually, this can be seen from the last rule in Algorithm \ref{linearization}: if $r=u\phi^{(k)}(A)^{-1}v$ for some matrix $A$ over $\C\left\langle x_{1},\dots,x_{n}\right\rangle $ and scalar-valued row and column vectors $u$ and $v$, then
\[r^{-1}=\begin{pmatrix}1 & \0\end{pmatrix}\begin{pmatrix}0 & u\\v & \phi^{(k)}(A)\end{pmatrix}^{-1}\begin{pmatrix}1\\\0\end{pmatrix},\]
where the invertiblity of matrix
\[\begin{pmatrix}0 & u\\v & \phi^{(k)}(A)\end{pmatrix}\]
follows from the invertiblity of $r=u\phi^{(k)}(A)^{-1}v$ in $\A$ by the following well-known lemma about Schur complements.

\begin{lemma}
\label{lem:Schur complement}Suppose that $\A$ is a unital algebra. Let $k,l\in\N$, $A\in M_{k}(\A)$, $B\in M_{k\times l}(\A)$, $C\in M_{l\times k}(\A)$ and $D\in M_{l}(\A)$ such that $D$ is invertible. Then the matrix
\[\begin{pmatrix}A & B\\C & D\end{pmatrix}\]
is invertible in $M_{k+l}(\A)$ if and only if the Schur complement $A-BD^{-1}C$ is invertible in $M_{k}(\A)$.
\end{lemma}

Then a closely related notion is the smallest subalgebra that has this property, namely, is closed under taking inverses, as following.

\begin{definition}
\label{def:division closure}Let $\mathcal{R}$ be a subalgebgra of $\A$. The \emph{division closure of $\mathcal{R}$} in $\A$ is the smallest subalgebra $\D$ of $\A$ containing $\mathcal{R}$ which is closed under taking inverses in $\A$, i.e., if $d\in\D$ is invertible in $\A$, then $d^{-1}\in\D$.
\end{definition}

From the definition it follows that the rational closure $\mathcal{R}_{\phi}$ for some homomorphism $\phi:\mathbb{C}\left\langle x_{1},\dots,x_{n}\right\rangle \rightarrow\A$ always contains the division closure of the image of $\mathbb{C}\left\langle x_{1},\dots,x_{n}\right\rangle $. In order to study when the rational closure actually is equal to the division closure, we consider a recursive structure for the division closure. Namely, now we begin with $\mathcal{R}_{0}:=\phi(\mathbb{C}\left\langle x_{1},\dots,x_{n}\right\rangle )$, then we set
\[\mathcal{R}_{0}^{-1}:=\{p^{-1}\bigm|p\in\mathcal{R}_{0}\text{ is invertible in }\A\}\subseteq\A\]
and define $\mathcal{R}_{1}$ as the subalgebra of $\A$ generated by the set $\mathcal{R}_{0}\cup\mathcal{R}_{0}^{-1}$. Let $\D$ be the division closure of $\mathcal{R}_{0}$, then clearly we have $\mathcal{R}_{0}^{-1}\subseteq\D$ by the definition of the division closure and thus $\mathcal{R}_{1}\subseteq\D$. It is not difficult to see that we can repeat this procedure to obtain $\mathcal{R}_{k}\subseteq\D$ for $k=1,2,\dots$ and thus
\[\mathcal{R}_{\infty}:=\bigcup_{k=1}^{\infty}\mathcal{R}_{k}\subseteq\D.\]
On the other hand, if an element $r\in\mathcal{R}_{k}$ for some $k\in\N$ is invertible in $\A$, then $r^{-1}\in\mathcal{R}_{k+1}$ and thus $\mathcal{R}_{\infty}$ is closed under taking inverses in $\A$; hence we have $\mathcal{R}_{\infty}=\D$.

Therefore, if we want to prove $\mathcal{R}_{\phi}=\D$, we should try to prove $\mathcal{R}_{\phi}=\mathcal{R}_{\infty}$. A nice criterium for this is given by the following lemma.

\begin{lemma}
\label{lem:division closure}Let $\phi$ be a homomorphism from $\mathbb{C}\left\langle x_{1},\dots,x_{n}\right\rangle $ to a unital algebra $\A$. If the rational closure $\mathcal{R}_{\phi}$ is a division ring, then $\mathcal{R}_{\phi}=\mathcal{R}_{\infty}=\D$.
\end{lemma}

The proof of this lemma can be done step by step as the proof of Theorem 2.4 (which shows that the free field has a similar recursive structure as above) in \cite{Yin18}, with the role of the free field therein replaced by the rational closure $\mathcal{R}_{\phi}$. It relies on representing the elements of rational closure as entries of inverses of matrices over polynomials and also Lemma \ref{lem:Schur complement} on the Schur complement.

Of course, the question when the rational closure $\mathcal{R}_{\phi}$ becomes a division ring is not an easy one. We will discuss some situation in the second half of Section \ref{sec:Atiyah}. For that purpose, we need the following
technical lemma.

\begin{lemma}
\label{lem:rational closure}Let $\phi$ be a homomorphism from $\C\left\langle x_{1},\dots,x_{n}\right\rangle $ to a unital algebra $\A$. If the rational closure $\mathcal{R}_{\phi}$ is a division ring, then the set of all square full matrices over $\mathcal{R}_{\phi}$ is closed under products and diagonal sums.
\end{lemma}

Actually, this lemma can be easily deduced from the following lemma.

\begin{lemma}
\label{lem:fullness and invertibility}Let $\A$ be a unital algebra.
\begin{enumerate}
\item If $\A$ is a division ring, then $\A$ is stably finite and any square full matrix is invertible.
\item If $\A$ is stably finite, then any invertible matrix is full.
\end{enumerate}
\end{lemma}

\begin{proof}
Firstly, we want to prove (i). So we suppose that $\A$ is a division ring and we want to prove that any square full matrix is invertible by induction on the dimension $k$. For $k=1$, $a\in\A$ full just means $a\neq0$, hence $a$ is invertible since $\A$ is a division ring. Now suppose that any full $\left(k-1\right)\times\left(k-1\right)$ matrix is invertible, we want to show it's true for any $k\times k$ full matrix. Let $A\in M_{k}\left(\A\right)$ be full, then left multiplied by a permutation matrix, $A$ can be written as
\[A=\begin{pmatrix}a & c\\d & B\end{pmatrix}\]
such that $a\neq0$, $B\in M_{k-1}\left(\A\right)$, $c\in M_{1,k}\left(\A\right)$ and $d\in M_{k,1}\left(\A\right)$; otherwise, $A$ would be hollow as the first column would be zero. Now, since $a$ is invertible, left multiplied by the invertible matrix
\[\begin{pmatrix}1 & \0 \\-da^{-1} & \1_{k-1}\end{pmatrix},\]
$A$ can be assumed to be of the form
\[A=\begin{pmatrix}a & c\\\0 & B\end{pmatrix}.\]
By Lemma \ref{lem:minor}, $B$ must be full as we drop the first row; then by induction, $B$ is invertible and so does $A$ because of
\[\begin{pmatrix}a & c\\\0 & B\end{pmatrix}\begin{pmatrix}a^{-1} & -a^{-1}cB^{-1}\\\0 & B^{-1}\end{pmatrix}=\begin{pmatrix}a^{-1} & -a^{-1}cB^{-1}\\\0 & B^{-1}\end{pmatrix}\begin{pmatrix}a & c\\\0 & B\end{pmatrix}=\1_{k}.\]

Next, we want to prove that $\A$ is stably finite. For that purpose, we want to show by induction on $k$ that any right inverse is also a left inverse in $M_{k}\left(\A\right)$. If $k=1$, it is true as $\A$ is a division ring. Now let $A$ and $B$ be matrices in $M_{k}\left(\A\right)$ such that $AB=\1_{k}$, then it suffices to prove $A$ is full; because, by what we just proved, this implies that $A$ is invertible, and so the right inverse $B$ is also a left inverse. Now, assume that $A$ is not full, then there is a rank factorization
\[A=\begin{pmatrix}C\\C'\end{pmatrix}\begin{pmatrix}D & D'\end{pmatrix},\]
where $r<k$, and $C,D\in M_{r}\left(\A\right)$, $C'\in M_{k-r,r}\left(\A\right)$, and $D'\in M_{r,k-r}\left(\A\right)$. Write
\[B=\begin{pmatrix}B_{1} & B_{2}\\B_{3} & B_{4}\end{pmatrix},\]
where $B_{1}\in M_{r}\left(\A\right)$, $B_{4}\in M_{k-r}\left(\A\right)$ and $B_{2}$, $B_{3}$ are matrices over $\A$ of appropriate sizes, then we have
\[\1_{k}=AB=\begin{pmatrix}C\\C'\end{pmatrix}\begin{pmatrix}D & D'\end{pmatrix}\begin{pmatrix}B_{1} & B_{2}\\B_{3} & B_{4}\end{pmatrix}=\begin{pmatrix}C\\C'\end{pmatrix}\begin{pmatrix}DB_{1}+D'B_{3} & DB_{2}+D'B_{4}\end{pmatrix}.\]
So we have $C(DB_{1}+D'B_{3})=\1_{r}$, namely, $C$ has $DB_{1}+D'B_{3}$ as its right inverse. This yields that $C$ is invertible by the induction, and so does $DB_{1}+D'B_{3}$. Then we obtain $DB_{2}+D'B_{4}=\0$ and $C'=\0$, hence
\[\1_{k}=AB=\begin{pmatrix}C\\\0\end{pmatrix}\begin{pmatrix}DB_{1}+D'B_{3} & \0\end{pmatrix}=\begin{pmatrix}\1_{r} & \0\\\0 & \0\end{pmatrix}.\]
This gives a contradiction and hence $A$ is full.

Finally, we want to show the part (ii), that is, if $\A$ is stably finite, then invertible matrices are full. This follows from Lemma \ref{lem:full identity}, which says that any identity matrix is full when $\A$ is stably finite: Let $A\in M_{k}(\A)$ be invertible, i.e., there exists $B\in M_{k}(\A)$ such that $AB=\1_{k}$, so this factorization of $\1_{k}$ is a rank factorization since $\1_{k}$ is full; then by the part (iii) of Lemma \ref{lem:inner rank}, $A$ and $B$ are full.
\end{proof}

\section{Applications to Atiyah properties}

\subsection{Preliminaries for affiliated operators}

Let $\left(\M,\tau\right)$ be a tracial $W^{\ast}$-probability space as before. In this section, we denote by $\A$ the set of all closed and densely defined linear operators affiliated with $\M$, which is known to be a $\ast$-algebra containing $\M$. An important and well-known fact is that the polar decomposition also holds in this case.

\begin{lemma}
Let $X$ be a closed densely defined operator on some Hilbert space $H$, then we have $X=U\left|X\right|$, where $\left|X\right|=(X^{\ast}X)^{\frac{1}{2}}$ is a positive selfadjoint (so necessarily closed densely defined) operator and $U$ is a partial isometry such that $U^{\ast}U=p_{(\ker(X))^{\bot}}$ and $UU^{\ast}=p_{\overline{\im(X)}}$. Moreover, $X$ is affiliated with $\M$ if and only if $U\in\M$ and $\left|X\right|$ is affiliated with $\M$.
\end{lemma}

Therefore, we have the analogue of Lemma \ref{lem:kernels}:

\begin{lemma}
\label{lem:kernels unbdd}Given $X\in\A$, let $p_{\ker(X)}$ and $p_{\overline{\im(X)}}$ denote the orthogonal projections onto $\ker(X)$ and the closure of $\im(X)$, respectively. Then they belong both to $\M$ and satisfy
\[\tau(p_{\ker(X)})+\tau(p_{\overline{\im(X)}})=1.\]
\end{lemma}

Moreover, for each integer $N$, we can also consider the matricial extension $M_{N}(\A)$, which is the $\ast$-algebra of closed and densely defined linear operators affiliated to the $W^{\ast}$-probability space $(M_{N}(\M),\tr_{N}\circ\tau^{(N)})$. Therefore, with the help of the polar decomposition and also its matricial extended version, we can show that $\A$ is also stably finite; thus the evaluation of rational functions is well-defined on $\A$ by Theorem \ref{thm:evaluation}.

\begin{lemma}
$\A$ is stably finite.
\end{lemma}

\begin{proof}
It suffices to prove that any $X,Y\in\A$ with $XY=1$ implies $YX=1$, since $M_{N}(\A)$ is also a $\ast$-algebra of affiliated operators for each $N$. First, let $X,Y\in\A$ with $XY=1$ and $X=X^{\ast}$. Then $Y^{\ast}X=1$, hence $Y^{\ast}=Y^{\ast}XY=Y$, and thus $YX=1$. Next, we consider now arbitrary $X,Y\in\A$ with $XY=1$. By the polar decomposition, we can write $Y=U\left|Y\right|$ with a partial isometry $U\in\M$ and $\left|Y\right|\in\A$. Note that $\left|Y\right|$ is selfadjoint and satisfies $XU\left|Y\right|=XY=1$ in $\A$, so by the previous argument we have $\left|Y\right|XU=1$. Then $U$ is injective and so it must be unitary by the previous lemma. Hence $Y$ is also injective and thus invertible with inverse $X$.
\end{proof}

Now, consider an element $P\in M_{N}(\A)$, then it is invertible in $M_{N}(\A)$ if and only if $\tr_{N}\circ\tau^{(N)}(p_{\ker(P)})=0$. So with this setting, Theorem \ref{thm:main-1} asserts that if a tuple $X=(X_{1},\dots,X_{n})$ of selfadjoint random variables satisfies $\delta^{\star}(X_{1},\dots,X_{n})=n$, then a linear full matrix $P$ over $\C\left\langle x_{1},\dots,x_{n}\right\rangle $ gives an invertible evaluation $P(X)$ in $M_{N}(\A)$. In other words, a linear matrix $P$ of inner rank $N$ has $\tr_{N}\circ\tau^{(N)}(p_{\overline{\im(P(X))}})=1$. Therefore, if we use the unnormalized trace $\Tr_{N}$ instead of $\tr_{N}$, then these two quantities coincide. So we define the rank of a matrix over $\A$ by this unnormalized trace as following.

\begin{definition}
For any $P\in M_{N}(\A)$, we define its \emph{rank} as
\[\rank(P)=\Tr_{N}\circ\tau^{(N)}(p_{\overline{\im(P)}}).\]
\end{definition}

One of the main goals of this section is to show that this equality of these two ranks for full matrices is not a coincidence: as these two quantities both describe the invertibility of the matrices in some sense, they are naturally equal to each other once we choose some nice operators like in Theorem \ref{thm:main-1}. Furthermore, we will show that this equality holds not only for full matrices but actually for all matrices with arbitrary inner rank. Moreover, we also want to prove that the equality is equivalent to some kind of Atiyah property. For that purpose, we need the following two lemmas.

\begin{lemma}
\label{invertible and rank}$P\in M_{N}(\A)$ is invertible if and only if $\rank(P)=N$.
\end{lemma}

It is just a rephrased statement of Lemma \ref{lem:kernels unbdd} with this notion of rank. And this rank doesn't change when multiplied by invertible matrices over $\A$.

\begin{lemma}
\label{inverstible matrix preserve rank}(See \cite[Lemma 2.3]{Lin93}) If $Q$ is invertible in $M_{N}(\A)$, then $\rank(P)=\rank(PQ)=\rank(QP)$ for any $P\in M_{N}(\A)$.
\end{lemma}

\subsection{\label{sec:Atiyah}Atiyah properties}

Following the notion in \cite{SS15} with some adaptation, we have the following definition:

\begin{definition}
Let $X=(X_{1},\dots,X_{n})$ be a tuple with elements from a tracial $W^{\ast}$-probability space $(\M,\tau)$, and consider the evaluation map $\ev_{X}:\C\left\langle x_{1},\dots,x_{n}\right\rangle \rightarrow\M$. If for any matrix $P\in M_{N}(\C\left\langle x_{1},\dots,x_{n}\right\rangle )$, we have
\[\rank(P(X))\in\N,\]
then we say $X$ has the \emph{strong Atiyah property}.
\end{definition}

The presence of this property is one of the various formulations of the Atiyah conjecture, which arose in the work \cite{Ati74} and asks whether some analytic $L^{2}$-Betti numbers are always rational numbers for certain Riemannian manifolds. A priori this rank can be any real number in $[0,N]$ if defined as above, so it's not a trivial property if all these numbers are integer for some given operators. In this terminology, free groups (or precisely, the free group generators in their group von Neumann algebras) have been proved to have the strong Atiyah property in Linnell's paper \cite{Lin93}; see \cite[Chapter 10]{Luc02} for more detailed discussion, including some counterexamples, and references therein on Atiyah conjecture for group algebras. In the context of free probability, a tuple of non-atomic, freely independent random variables is also proven to have the strong Atiyah property in \cite{SS15} by Shlyakhtenko and Skoufranis.

In fact, with the help of the strong Atiyah property as well as the construction of rational closure and some techniques from Cohn's theory, in \cite{Lin93} Linnell also shows that for the free group, there exists some division ring, as a subring of the $\ast$-algebra of affiliated operators, containing the corresponding group algebra. Inspired by this result, in this section, we want to show that, if $X=(X_{1},\dots,X_{n})$ satisfies that $\delta^{\star}(X_{1},\dots,X_{n})=n$, then $X$ has the strong Atiyah property. Once the strong Atiyah property holds, as proved in \cite{Lin93}, we know that then the rational closure is a division ring. (The validity of Linnell's argument in this much more general context was pointed out to us by Ken Dykema and James Pascoe.) Actually, we want to establish the equivalence of these two properties; in addition, we also want to connect this with the question whether the inner rank of a matrix of polynomials is equal to the rank of its evaluation of the corresponding operators in some finite von Neumann algebra.

These equivalences are established in two settings: in the first one, we consider the evaluation of all rational functions; in the second one, we consider the rational closure of the $\ast$-algebra generated by a given tuple of operators. More precisely, in the first list of equivalent properties, one of them is that all non-zero rational functions are invertible, or alternatively, have no zero divisors; such a result is a natural generalization of the result that all polynomials have non zero divisors for operators with full entropy dimension, see \cite{MSW17} and \cite{CS16}. We achieve this result by linear representations for rational functions, introduced in Section 5, combining with Theorem \ref{thm:main-1} that a linear full matrix cannot have zero divisors when evaluated at operators with full entropy dimension. In this case, a matrix of polynomials has its rank of the evaluation equal to its inner rank.

Moreover, we find that the equality of these two notions of ranks, as a property for the operators in consideration, is stronger than the strong Atiyah property: the strong Atiyah property only asks the rank of a matrix over operators to be integers, but in Theorem \ref{thm:Atiyah-1}, we ask the rank to be exactly the corresponding inner rank (which is an integer by definition). So, for the strong Atiyah property, it is possible that the rank is an integer but doesn't equal the inner rank.

Therefore, in our second list (Theorem \ref{thm:Atiyah-2}), we want to establish equivalent characterizations for the strong Atiyah property. Instead of evaluating rational functions, we consider the rational closure: in this case, we can show that the ranks of the evaluations are equal to inner ranks over the rational closure; and the latter can be a division algebra which is not isomorphic to the free field. In general, there is a gap between these two cases, which will be shown by an example.

Now, let $(\M,\tau)$ be a tracial $W^{\ast}$-probability space and we consider the evaluation map $\ev_{X}:\C\left\langle x_{1},\dots,x_{n}\right\rangle \rightarrow\A$ given by a tuple $X=(X_{1},\dots,X_{n})$ in $\M^n$, where $\A$ is the $\ast$-algebra of affiliated operators introduced in the previous section. Our first main theorem in this section is the following.

\begin{theorem}
\label{thm:Atiyah-1}The following statements are equivalent:
\begin{enumerate}
\item For any $N\in\N$ and $P\in M_{N}(\C\left\langle x_{1},\dots,x_{n}\right\rangle )$ we have: if $P$ is linear and full, then $P(X)\in M_{N}(\A)$ is invertible.
\item For any $N\in\N$ and $P\in M_{N}(\C\left\langle x_{1},\dots,x_{n}\right\rangle )$ we have: if $P$ is full, then $P(X)\in M_{N}(\A)$ is invertible.
\item For any $N\in\N$ and $P\in M_{N}(\C\left\langle x_{1},\dots,x_{n}\right\rangle )$ we have: $\rank(P(X))=\rho(P)$.
\item We have $X\in\dom_\A(r)$ for each $r\in \C\plangle x_{1},\dots,x_{n}\prangle$ and $\Ev_{X}$ as introduced in Definition \ref{def:evaluation} induces an injective homomorphism $\Ev_{X}: \C\plangle x_{1},\dots,x_{n}\prangle\rightarrow\A$ that extends the evaluation map $\ev_{X}:\C\left\langle x_{1},\dots,x_{n}\right\rangle \rightarrow\A$.
\end{enumerate}
Moreover, if the equivalent conditions above are satisfied, then
\begin{equation}\label{eq:rank_equality}
\rank(P(X)) = \rho(P) = \rho_\A(P(X)) \qquad\text{for all $P\in M_{N}(\C\left\langle x_{1},\dots,x_{n}\right\rangle )$},
\end{equation}
where $\rho_{\A}(P(X))$ denotes the inner rank of $P(X)$ over the algebra $\A$.
\end{theorem}

\begin{proof}
It's easy to see that (ii)$\implies$(i) is trivial and (iii)$\implies$(ii) follows from Lemma \ref{invertible and rank}.

(iv)$\implies$(iii): Assume that $P\in M_{N}(\C\left\langle x_{1},\dots,x_{n}\right\rangle )$ has inner rank $\rho(P)=r$, then by Theorem \ref{thm:full minor} (with its requirement checked in Appendix \ref{sec:full minor}), there exist a full $r\times r$ block of $P$. With some permutations of rows and columns (which don't change either the inner rank of $P$ or the rank of its evaluation $P(X)$), we may assume that $P$ is of the form
\[P=\begin{pmatrix}A & B\\C & D\end{pmatrix},\]
where $A\in M_{r}(\C\left\langle x_{1},\dots,x_{n}\right\rangle )$ and other blocks $B$, $C$, $D$ are of appropriate sizes. It can be verified directly that the factorization
\begin{equation}\label{eq:factorization with full minor}
\begin{pmatrix}\1_{r} & \0\\\0 & D-CA^{-1}B\end{pmatrix}=\begin{pmatrix}A^{-1} & \0\\-CA^{-1} & \1_{N-r}\end{pmatrix}\begin{pmatrix}A & B\\C & D\end{pmatrix}\begin{pmatrix}\1_{r} & -A^{-1}B\\\0 & \1_{N-r}\end{pmatrix}
\end{equation}
holds in $M_{N}(\C\plangle x_{1},\dots,x_{n}\prangle)$, since the full matrix $A$ is invertible in $M_{r}(\C\plangle x_{1},\dots,x_{n}\prangle)$ (as discussed in the previous section on rational functions). Note that
\[\begin{pmatrix}A^{-1} & \0\\-CA^{-1} & \1_{N-r}\end{pmatrix},\ \begin{pmatrix}\1_{r} & -A^{-1}B\\\0 & \1_{N-r}\end{pmatrix}\]
are invertible in $M_{N}(\C\plangle x_{1},\dots,x_{n}\prangle)$, hence we have
\[r=\rho(P)=\rho\begin{pmatrix}A & B\\C & D\end{pmatrix}=\rho\begin{pmatrix}\1_{r} & \0\\\0 & D-CA^{-1}B\end{pmatrix}.\]
By Proposition \ref{prop:invertible minor} (the stable finiteness of polynomials follows from the stable finiteness of the free field, which can be seen from Lemma \ref{lem:fullness and invertibility}), we have $D=CA^{-1}B$. As we assume statement (iv), the extended evaluation $\Ev_{X}:\C\plangle x_{1},\dots,x_{n}\prangle\rightarrow \A$, as a homomorphism, yields that $A(X)$ is invertible with the inverse given by the evaluation $A^{-1}(X)$ since
\[\1_r=\Ev_{X}^{(r)}(AA^{-1})=\Ev_{X}^{(r)}(A)\Ev_{X}^{(r)}(A^{-1})=A(X)A^{-1}(X).\]
Therefore, (\ref{eq:factorization with full minor}) leads to the following factorization 
\[\begin{pmatrix}\1_{r} & \0\\\0 & \0\end{pmatrix}=\begin{pmatrix}A^{-1}(X) & \0\\-C(X)A^{-1}(X) & \1_{N-r}\end{pmatrix}\begin{pmatrix}A(X) & B(X)\\C(X) & D(X)\end{pmatrix}\begin{pmatrix}\1_{r} & -A^{-1}(X)B(X)\\\0 & \1_{N-r}\end{pmatrix}.\]
Applying Lemma \ref{inverstible matrix preserve rank}, we obtain
\[\rank(P(X))=\rank\begin{pmatrix}A(X) & B(X)\\C(X) & D(X)\end{pmatrix}=\rank\begin{pmatrix}\1_{r} & \0\\\0 & \0\end{pmatrix}=r.\]

(i)$\implies$(iv): First, recall from Definition \ref{def:evaluation} that a rational function $r$ in the free field $\C\plangle x_{1},\dots,x_{n}\prangle$ satisfies $X\in\dom_\A(r)$ if there is a linear representation $\rho=(u,A,v)$ of $r$ with the property that $X\in \dom_\A(\rho)$, i.e., for which $A(X)$ is invertible; but in fact, each linear representation of $r$ (whose existence is guaranteed by Theorem \ref{thm:linear representation}) has this property due to our assumption (i) as $A$ is full.
Thus, according to Definition \ref{def:evaluation} and Theorem \ref{thm:evaluation}, the evaluation $\Ev_X(r)$ is well-defined for each $r\in \C\plangle x_{1},\dots,x_{n}\prangle$ and thus induces a map $\Ev_X: \C\plangle x_1,\dots,x_n\prangle \to \A$. Now, we infer from the proof of Theorem \ref{thm:evaluation} that the evaluation of rational functions via such linear representations coincides with the usual evaluation of polynomials and respects the arithmetic operations between rational functions. Therefore, the evaluation map $\Ev_X: \C\plangle x_1,\dots,x_n\prangle \to \A$ forms a homomorphism which agrees with $\ev_{X}$ on $\C\left\langle x_{1},\dots,x_{n}\right\rangle$. Moreover, $\Ev_{X}$ has to be injective as a homomorphism from a skew field.

Suppose now that the equivalent conditions are satisfied. Then, for any matrix $P\in M_{N}(\C\left\langle x_{1},\dots,x_{n}\right\rangle )$, we can consider the rank factorization $P(X)=AB$ of $P(X)$ over $\A$, where $A\in M_{N,r}(\A)$ and $B\in M_{r,N}(\A)$ for $r := \rho_\A(P(X)) \leqslant N$. This can be rewritten as $P(X) = \hat{A} \hat{B}$ with the square matrices $\hat{A},\hat{B} \in M_N(\A)$ that are defined by
\[\hat{A} := \begin{pmatrix} A & \0_{N \times (N-r)} \end{pmatrix} \qquad\text{and}\qquad \hat{B} := \begin{pmatrix} B\\ \0_{(N-r) \times N} \end{pmatrix}.\]
From this, we see that $(\tr_N \circ \tau^{(N)})(p_{\ker(\hat{A})}) \geq \frac{N-r}{N}$, so that $\rank(\hat{A}) \leq r$ by Lemma \ref{lem:kernels unbdd}; thus, since $\im(P(X))\subseteq\im(\hat{A})$, it follows that
\[\rank(P(X)) \leqslant \rank(\hat{A})\leqslant r.\]
On the other hand, we may observe that in general
\[r = \rho_{\A}(P(X)) \leqslant \rho(P),\]
because each rank factorization of $P$ yields after evaluation at $X$ a factorization of $P(X)$ over $\A$. Finally, the third property in the theorem gives us
\[\rho(P) = \rank(P(X)).\]
Thus, in summary, the asserted equality \eqref{eq:rank_equality} follows.
\end{proof}

We want to remark that the fourth property in the theorem implies that any non-zero rational function $r$ has no zero divisors for its evaluation $r(X)$: for any right zero divisor $p\in\M$, $r(X)p=0$ yields that $\im(p)\subseteq\ker(r(X))$, but $\ker(r(X))$ is always trivial as $r(X)$ is invertible in $\A$ (where we use the property that the evaluation $\Ev_{X}$ is a homomorphism defined on the whole free field).

In other words, the fourth property also says that the image of the free field under the evaluation map forms a division subring of $\A$ that contains the algebra $\C\left\langle X_{1},\dots,X_{n}\right\rangle$ generated by $X_1,\dots X_n$; therefore, from the fourth property in our above theorem, we may also infer that the division closure $\D$ of $\C\left\langle X_{1},\dots,X_{n}\right\rangle$ is contained in the image of free field $\Ev_{X}(\C\plangle x_{1},\dots,x_{n}\prangle)$. Such a result was first established by Linnell in his paper \cite{Lin93} for free groups, by first proving the Atiyah conjecture for free groups. More precisely, he proved that the division closure (see Definition \ref{def:division closure}) of the group algebra is a division ring (Lemma 3.7 in \cite{Lin93}).

Moreover, we can also consider the rational closure $\mathcal{R}$ (see Definition \ref{def:rational closure}) of $\C\left\langle X_{1},\dots,X_{n}\right\rangle$ with respect to $\ev_{X}$, which contains the image $\Ev_{X}(\C\plangle x_{1},\dots,x_{n}\prangle)$ by the definition of rational closure and the way we define the evaluation map $Ev_{X}$ for rational functions. So we have
\[\D\subseteq\Ev_{X}(\C\plangle x_{1},\dots,x_{n}\prangle)\subseteq\mathcal{R};\]
and actually, these three algebras are equal to each other in our setting, since $\D=\mathcal{R}$ holds as we will see in the following; in fact, this follows from Theorem \ref{thm:Atiyah-2} (which applies if the equivalent conditions of Theorem \ref{thm:Atiyah-1} hold) in combination with Lemma \ref{lem:division closure}.

Therefore, with the equivalences of these properties, we get a complete understanding for the existence of the embedding of the free field into unbounded affiliated operators. Even though the free group case does not follow directly from our result, we can establish this embedding for the big class of operators $(X_1,\dots,X_n)$ that have maximal non-microstates free entropy dimension; in fact, we can prove this under the weaker condition $\delta^\star(X_1,\dots,X_n)=n$. We put it as the following corollary.

\begin{corollary}\label{cor:full_entropy_dimension_implies_Atiyah}
If $X=(X_{1},\dots,X_{n})$ is a tuple of selfadjoint random variables in some tracial $W^{\ast}$-probability space $(\M,\tau)$ with $\delta^{\star}(X_{1},\dots,X_{n})=n$, then
\begin{enumerate}
\item for any non-zero rational function $r$, $r(X)$ is well-defined and invertible as an affiliated operator, and thus $r(X)$ is not zero and has no zero divisors;
\item for any $N\in\N$ and any $N\times N$ matrix $P$ over $\C\left\langle x_{1},\dots,x_{n}\right\rangle $, we have
\[\rank(P(X))=\rho(P)\in\N\cap[0,N].\]
\end{enumerate}
\end{corollary}

In particular, the second statement implies that the strong Atiyah property holds for $X$. Moreover, we also know exactly the dimension of the kernel of non-full matrices evaluated at $X$.

\begin{remark}
For any matrix $P$ over $\C\left\langle x_{1},\dots,x_{n}\right\rangle $,
\[\Tr_N\circ\tau^{(N)}(p_{\ker(P(X))})=N-\rho(P).\]
Therefore, if $P$ is self-adjoint, then the analytic distribution of the self-adjoint operator $P(X)$, with respect to the normalized trace $\tr_{N}\circ\tau^{(N)}$, has an atom at $0$ of measure $1-\rho(P)/N$, whenever $\rho(P)<N$, i.e., $P$ is not full. Furthermore, it allows us to determine the value of the non-microstates free entropy dimension $\delta^\ast(P(X))$ of $P(X)$, as we already explained in Remark \ref{rem:value_entropy_dimension}.
\end{remark}

As mentioned before, a tuple of operators may have the strong Atiyah property but fail the equality in the third property in Theorem \ref{thm:Atiyah-1}. Here we present an example provided by Ken Dykema and James Pascoe.

\begin{example}\label{ex:Dykema_Pascoe}
Consider two freely independent semicircular elements, denoted by $X$ and $Y$, then they satisfy the strong Atiyah property. Let
\[A=Y^{2},\ B=YXY,\ C=YX^{2}Y,\]
then they also have the strong Atiyah property as any polynomial in them can be reduced back to a polynomial in $X$ and $Y$. However, though they don't satisfy any nontrivial polynomial relation, they have a rational relation:
\[BA^{-1}B-C=0\]
in the $\ast$-algebra of affiliated operators. Then definitely they don't satisfy the last property in Theorem \ref{thm:Atiyah-1}; moreover, we can also find some matrix like
\[\begin{pmatrix}A & B\\B & C\end{pmatrix}\]
that has inner rank $2$ if it is viewed as a matrix of formal variables, but has
\[\rank\begin{pmatrix}A & B\\B & C\end{pmatrix}=\rank\begin{pmatrix}Y^{2} & YXY\\YXY & YX^{2}Y\end{pmatrix}=\rank\begin{pmatrix}1 & X\\X & X^{2}\end{pmatrix}=\rho\begin{pmatrix}1 & x\\x & x^{2}\end{pmatrix}=1.\]
Therefore, $(A,B,C)$ violates all the properties in Theorem \ref{thm:Atiyah-1} though it has the strong Atiyah property. Nevertheless, by the following list of equivalent properties, we see that the rank is always equal to the inner rank over the rational closure when the strong Atiyah property holds.
\end{example}

Let $\mathcal{R}$ be the rational closure of $\C\left\langle x_{1},\dots,x_{n}\right\rangle $ with respect to $\ev_{X}:\C\left\langle x_{1},\dots,x_{n}\right\rangle \rightarrow\A$, which is a subalgebra of $\A$. In the following theorem, we consider the inner rank over $\mathcal{R}$ and denote it by $\rho_{\mathcal{R}}$. Similarly as for the inner rank $\rho$, if a matrix $A$ over $\mathcal{R}$ is multiplied by invertible matrices over $\mathcal{R}$, then its inner rank $\rho_{\mathcal{R}}$ stays invariant. We have the following equivalent properties.

\begin{theorem}
\label{thm:Atiyah-2}The following statements are equivalent:
\begin{enumerate}
\item For any $N\in\N$ and any $P\in M_{N}(\C\left\langle x_{1},\dots,x_{n}\right\rangle )$ we have: if $P(X)$ is full over $\mathcal{R}$, then $P(X)\in M_{N}(\A)$ is invertible.
\item For any $N\in\N$ and any $P\in M_{N}(\C\left\langle x_{1},\dots,x_{n}\right\rangle )$we have: $\rank(P(X))=\rho_{\mathcal{R}}(P(X))$.
\item The rational closure $\mathcal{R}$ is a division ring.
\item We have the strong Atiyah property for $X$, i.e., for any $N\in\N$ and any $P\in M_{N}(\C\left\langle x_{1},\dots,x_{n}\right\rangle )$ we have that $\rank(P(X))\in\N$.
\end{enumerate}
\end{theorem}

\begin{proof}
It's easy to see that (ii)$\implies$(i) follows from Lemma \ref{invertible and rank}.

(iii)$\implies$(ii): Assume that the evaluation of $P\in M_{N}(\C\left\langle x_{1},\dots,x_{n}\right\rangle)$ has inner rank $\rho_{\mathcal{R}}(P(X))=r$. By Lemma \ref{lem:rational closure}, the requirements in Proposition \ref{prop:full minor} are satisfied and thus we can apply this proposition to $P(X)$; so there exist a full matrix $A\in M_{r}(\mathcal{R})$ and matrices $B$, $C$ and $D$ over $\mathcal{R}$ of appropriate sizes such that we can write $P(X)$ as
\[P(X)=\begin{pmatrix}A & B\\C & D\end{pmatrix}.\]
From Lemma \ref{lem:fullness and invertibility}, $A$ is invertible and $A^{-1}\in M_{r}(\mathcal{R})$ as $A$ is full; hence the factorization
\begin{equation}
\begin{pmatrix}\1_{r} & \0\\\0 & D-CA^{-1}B\end{pmatrix}=\begin{pmatrix}A^{-1} & \0\\-CA^{-1} & \1_{N-r}\end{pmatrix}\begin{pmatrix}A & B\\C & D\end{pmatrix}\begin{pmatrix}\1_{r} & -A^{-1}B\\\0 & \1_{N-r}\end{pmatrix}\label{eq:inner rank}
\end{equation}
holds in $M_{N}(\mathcal{R})$, and thus
\[r=\rho_{\mathcal{R}}(P(X))=\rho_{\mathcal{R}}\begin{pmatrix}A & B\\C & D\end{pmatrix}=\rho_{\mathcal{R}}\begin{pmatrix}\1_{r} & \0\\\0 & D-CA^{-1}B\end{pmatrix}.\]
As $\mathcal{R}$ is also stably finite (by Lemma \ref{lem:fullness and invertibility}), we can apply Proposition \ref{prop:invertible minor} to see that $D-CA^{-1}B=0$. Therefore, (\ref{eq:inner rank}) turns out to be
\[\begin{pmatrix}\1_{r} & \0\\\0 & \0\end{pmatrix}=\begin{pmatrix}A^{-1} & \0\\-CA^{-1} & \1_{N-r}\end{pmatrix}P(X)\begin{pmatrix}\1_{r} & -A^{-1}B\\\0 & \1_{N-r}\end{pmatrix};\]
then by applying Lemma \ref{inverstible matrix preserve rank} we have
\[\rank(P(X))=\rank\begin{pmatrix}\1_{r} & \0\\\0 & \0\end{pmatrix}=r.\]

(i)$\implies$(iii): For any nonzero $r\in\mathcal{R}$ , there exist a matrix $P\in M_{N}(\C\left\langle x_{1},\dots,x_{n}\right\rangle )$, $u\in M_{1,N}(\C)$, $v\in M_{N,1}(\C)$ such that $P(X)$ is invertible in $M_{N}(\A$) and $r=uP(X)^{-1}v$ (see Section 5.4). It can be verified that the factorization
\[\begin{pmatrix}-r & \0\\\0 & \1_{N}\end{pmatrix}=\begin{pmatrix}-uP(X)^{-1}v & \0\\\0 & \1_{N}\end{pmatrix}=\begin{pmatrix}1 & -uP(X)^{-1}\\\0 & P(X)^{-1}\end{pmatrix}\begin{pmatrix}0 & u\\v & P(X)\end{pmatrix}\begin{pmatrix}1 & \0\\-P(X)^{-1}v & \1_{N}\end{pmatrix}\]
holds in $M_{N+1}(\A)$. Since
\[\begin{pmatrix}1 & -uP(X)^{-1}\\\0 & P(X)^{-1}\end{pmatrix}\text{ and }\begin{pmatrix}1 & \0\\-P(X)^{-1}v & \1_{N}\end{pmatrix}\]
are invertible in $M_{N+1}(\A)$, we have
\[\rank\begin{pmatrix}-r & \0\\\0 & \1_{N}\end{pmatrix}=\rank\begin{pmatrix}0 & u\\v & P(X)\end{pmatrix}.\]
Then, as $\mathcal{R}$ is stably finite (because it is a subalgebra of $\A$ which is stably finite), by Proposition \ref{prop:invertible minor},
\[\begin{pmatrix}0 & u\\v & P(X)\end{pmatrix}\]
is full over $\mathcal{R}$ since $-uP(X)^{-1}v=-r\neq0$; hence this matrix is invertible by the assumption (i) and so is the matrix
\[\begin{pmatrix}-r & \0\\\0 & \1_{N}\end{pmatrix}.\]
So we see that $r$ is invertible. Moreover, note
\[-r^{-1}=-(uP(X)^{-1}v)^{-1}=\begin{pmatrix}1 & \0\end{pmatrix}\begin{pmatrix}0 & u\\v & P(X)\end{pmatrix}^{-1}\begin{pmatrix}1\\\0\end{pmatrix},\]
so $r^{-1}$ is also in the rational closure $\mathcal{R}$, by definition of the rational closure. This shows that $\mathcal{R}$ is a division ring.

Finally, it remains to prove that the assertion (iv) is equivalent to the first three assertions. It is clear that (ii) implies (iv) trivially, as the inner rank is always an integer by definition. Now we want to prove (iii) from assertion (iv) by more or less the same argument as in (i)$\implies$(iii). Given any nonzero $r\in\mathcal{R}$, there exists a matrix $P\in M_{N}(\C\left\langle x_{1},\dots,x_{n}\right\rangle )$, $u\in M_{1,N}(\C)$, $v\in M_{N,1}(\C)$ such that $P(X)$ is invertible in $M_{N}(\A$) and $r=uP(X)^{-1}v$. Consider again the factorization
\[\begin{pmatrix}-r & \0\\\0 & \1_{N}\end{pmatrix}=\begin{pmatrix}1 & -uP(X)^{-1}\\\0 & P(X)^{-1}\end{pmatrix}\begin{pmatrix}0 & u\\v & P(X)\end{pmatrix}\begin{pmatrix}1 &\0\\-P(X)^{-1}v & \1_{N}\end{pmatrix},\]
we have
\[\rank\begin{pmatrix}-r & \0\\\0 & \1_{N}\end{pmatrix}=\rank\begin{pmatrix}0 & u\\v & P(X)\end{pmatrix}.\]
Now, by assertion (iv), we have
\[\rank\begin{pmatrix}-r & \0\\\0 & \1_{N}\end{pmatrix}=\rank\begin{pmatrix}0 & u\\v & P(X)\end{pmatrix}\in\N;\]
combining this with the fact that
\[\rank\begin{pmatrix}-r & \0\\\0 & \1_{N}\end{pmatrix}=\rank r+N,\]
we obtain $\rank(r)\in\{0,1\}$. Then, as $r\neq0$, we have $\rank(r)=1$, and thus $r$ is invertible by Lemma \ref{invertible and rank}.

Moreover, now we have
\[\rank\begin{pmatrix}0 & u\\v & P(X)\end{pmatrix}=\rank\begin{pmatrix}-r & \0\\\0 & \1_{N}\end{pmatrix}=N+1,\]
so
\[\begin{pmatrix}0 & u\\v & P(X)\end{pmatrix}\]
is invertible by Lemma \ref{invertible and rank} and thus
\[-r^{-1}=-\left(uP(X)^{-1}v\right)^{-1}=\begin{pmatrix}1 & \0\end{pmatrix}\begin{pmatrix}0 & u\\v & P(X)\end{pmatrix}^{-1}\begin{pmatrix}1\\\0\end{pmatrix}\]
also lies in the rational closure $\mathcal{R}$.
\end{proof}

If any of the above properties holds for some given tuple of operators, then we also have that the division closure of the algebra generated by these operators forms a division ring; this follows directly from the fact that the rational closure is exactly the division closure in this setting (see Lemma \ref{lem:division closure}). 

Finally, we close this section by examining the group algebra case for our theorems.

\begin{example}
As mentioned before, other important and interesting examples are group algebras. Let $G$ be generated by $n$ elements $g_{1},\dots,g_{n}$, then their images $U_{1},\dots,U_{n}$ under the left regular representation of $G$, are unitary operators which generate the group von Neumann algebra $L(G)$; and there is a trace $\tau$ on $L(G)$ (the vector state deduced from the identity element of $G$) such that $(L(G),\tau)$ is a tracial $W^{\ast}$-probability space. So we can apply Theorem \ref{thm:Atiyah-2} to the generators $U_{1},\dots,U_{n}$, by the evaluation map $\ev:\C\left\langle x_{1},\dots,x_{n},y_{1},\dots,y_{n}\right\rangle \rightarrow\A(G)$ that is defined through
\[\ev(x_{i})=U_{i},\ \ev(y_{i})=U_{i}^{\ast},\ i=1,\dots,n,\]
where $\A(G)$ is the $\ast$-algebra of densely defined operators affiliated with $L(G)$ as usual. In this way, the evaluation of $\C\left\langle x_{1},\dots,x_{n},y_{1},\dots,y_{n}\right\rangle $ is nothing else but the image of $\C[G]$ under the left regular representation. Since we have the unitary relations $U_{i}U_{i}^{\ast}=U_{i}^{\ast}U_{i}=1$, the tuple $(U_{1},\dots,U_{n},U_{1}^{\ast},\dots,U_{n}^{\ast})$ never satisfies any of the properties in Theorem \ref{thm:Atiyah-1}. But, on the other hand, these unitary relations also tell us that we can forget about $U_{i}^{\ast}$ if we treat them as inverses of $U_{i}$. Hence, we can also consider the evaluation map $\ev:\C\left\langle x_{1},\dots,x_{n}\right\rangle \rightarrow\A(G)$ that is defined through
\[\ev(x_{i})=U_{i},\ i=1,\dots,n;\]
then, as $U_{i}^{-1}=U_{i}^{\ast}$, the rational closure $\mathcal{R}(G)$ of the image of $\C\left\langle x_{1},\dots,x_{n}\right\rangle $ have to contain all $U_{i}^{\ast}$. Therefore, as a subalgebra of $\A(G)$, $\mathcal{R}(G)$ contains also the image of the group algebra $\C[G]$. Then there is hope that some property in Theorem \ref{thm:Atiyah-1} can hold; in that case, the evaluation can be extended to the free field $\C\plangle x_{1},\dots,x_{n}\prangle$ with its image being the rational closure $\mathcal{R}(G)$. Actually, free group algebras are known to satisfy these properties in Theorem \ref{thm:Atiyah-1}: in \cite{Lin93} Linnell also proved that for a free group $\mathbb{F}_{n}$, the rational closure of $\mathcal{R}(\mathbb{F}_{n})$ is the universal field of fractions for $\C[\mathbb{F}_{n}]$, which turns out to be the free field $\C\plangle x_{1},\dots,x_{n}\prangle$; hence the fourth property in Theorem \ref{thm:Atiyah-1} is valid for generators $U_{1},\dots,U_{n}$.
\end{example}

\section{Absolute continuity}

In this section, we continue our investigations in the spirit of \cite{EKYY13,AjEK18,AEK18} that we began in Section \ref{sec:regularity_linear_matrices}. We have already seen in Theorem \ref{thm:main-3} that the condition $\delta^\star(X_1,\dots,X_n) = n$, and in particular the stronger version $\delta^\ast(X_1,\dots,X_n) = n$ thereof, allow us to conclude that the analytic distribution $\mu_\X$ of any operator of the form
$$\X = b_0 + b_1 X_1 + \dots + b_n X_n,$$
with selfadjoint matrices $b_1,\dots,b_n$ coming from $M_N(\C)$, cannot have atoms if the associated matrix $b_1 x_1 + \dots + b_n x_n$ in $M_N(\C\langle x_1,\dots,x_n\rangle)$ is full over $\C\langle x_1,\dots,x_n\rangle$. This is in accordance with the common philosophy that both $n-\delta^\star(X_1,\dots,X_n)$ and $n-\delta^\ast(X_1,\dots,X_n)$ measure the ``atomic part'' in the noncommutative distribution of $(X_1,\dots,X_n)$ and are accordingly somehow the weakest regularity conditions that we may impose on the noncommutative distribution of $(X_1,\dots,X_n)$.

The opposite end of the scale of regularity conditions is approached when assuming the existence of a dual system. Indeed, it was shown in \cite{CS16} that this condition allows positive statements about the absolute continuity of analytic distributions with respect to the Lebesgue measure. In this section, we give more evidence to this conceptual point of view by showing that in the case $b_0=0$ the fullness of $b_1 x_1 + \dots + b_n x_n$ guarantees even the absolute continuity of the analytic distribution $\mu_\X$ with respect to the Lebesgue measure.

\subsection{Some notational preliminaries}

Let $\A$ be a unital complex algebra.
The algebraic tensor product $\A \otimes \A$ over $\C$ carries a natural linear involution $\sim: \A \otimes \A \to \A \otimes \A$ that is determined by linear extension of $(a_1 \otimes a_2)^\sim = a_2 \otimes a_1$ for all $a_1,a_2\in\A$. We will refer to $\sim$ as the \emph{flip} on $\A \otimes \A$. Note that $\sim$ naturally extends to a linear involution on $M_N(\M \otimes \M)$, which will be denoted again by the same symbol $\sim$ and is defined by $u^\sim := (u^\sim_{kl})_{k,l=1}^N$ for each $u=(u_{kl})_{k,l=1}^N \in M_N(\A \otimes \A)$.

Now, let $\M$ be an $\A$-bimodule. We have used before that $\sharp: (\A \otimes \A) \times \M \to \M$ extends to an operation
$$\sharp:\ M_N(\A \otimes \A) \times \M \to M_N(\M),\qquad (u_{kl})_{k,l=1}^N \sharp m = (u_{kl} \sharp m)_{k,l=1}^N.$$
In the following, we will use that $\sharp$ extends further to an operation
$$\sharp:\ M_N(\A \otimes \A) \times M_N(\M) \to M_N(\M),\qquad (u_{kl})_{k,l=1}^N \sharp (m_{kl})_{k,l=1}^N = \bigg( \sum^N_{p=1} u_{kp} \sharp m_{pl} \bigg)_{k,l=1}^N,$$
which is obviously compatible with the latter under the canonical embedding $\M \subseteq M_N(\M)$ and thus justifies the usage of the same symbol.

\subsection{Schatten-class operators}

Let $(\H,\langle \cdot,\cdot\rangle)$ be a separable complex Hilbert space. An operator $T: \H \to \H$ is said to be of \emph{trace class}, if for some (and hence for each) orthonormal basis $(e_i)_{i\in I}$ of $\H$ the condition $\sum_{i\in I} \langle |T| e_i, e_i\rangle < \infty$ for $|T| := (T^\ast T)^{1/2}$ is satisfied. It can be shown that in such cases $\sum_{i\in I} \langle T e_i, e_i\rangle$ is an absolutely convergent series, whose value, denoted by $\Tr(T)$, is independent of the concrete choice of $(e_i)_{i\in I}$; we will refer to $\Tr(T)$ as the \emph{trace of $T$}.
Note that in particular each finite rank operator on $\H$ is of trace class.

Clearly, an operator $T$ is of trace class if and only if $|T|$ is of trace class; thus, we may define by $\|T\|_1 := \Tr(|T|)$ a norm $\|\cdot\|_1$ on the linear space $S_1(\H)$ of all trace class operators on $\H$, with respect to which it becomes a Banach space. Note that $|\Tr(T)| \leq \|T\|_1$ for each operator $T\in S_1(\H)$.

More generally, for any $1\leq p < \infty$, we may define $S_p(\H)$ to be the linear space of all bounded operators $T$ on $\H$ for which $|T|^p$ is of trace class; this space also carries a norm, denoted accordingly by $\|\cdot\|_p$, which is defined by $\|T\|_p := \Tr(|T|^p)^\frac{1}{p}$ and with respect to which $S_p(\H)$ becomes a Banach space. We call $S_p(\H)$ the \emph{$p$-th Schatten-class on $\H$}.

Note that each Schatten-class $S_p(\H)$ consists only of compact operators on $\H$. Moreover, each $S_p(\H)$ forms even a two-sided ideal in $B(\H)$ as $\|A T B\|_p \leq \|A\| \|B\| \|T\|_p$ for all $T\in S_p(\H)$ and $A,B\in B(\H)$ holds. For trace class operators $T$, we have that $\Tr(A T) = \Tr(T A)$ for all $A\in B(\H)$, which justifies calling $\Tr$ a trace.

Furthermore, if $p,q \in (1,\infty)$ are given such that $\frac{1}{p} + \frac{1}{q} = 1$ holds, then $ST$ is of trace class whenever $S\in S_p(\H)$ and $T\in S_q(H)$, and in those cases $\|ST\|_1 \leq \|S\|_p \|T\|_q$.

Of particular interest is the class $S_2(\H)$, whose elements are also called \emph{Hilbert-Schmidt operators}. If endowed with the inner product $\langle \cdot,\cdot \rangle$ that is given by $\langle S,T \rangle := \Tr(ST^\ast)$ for all $S,T\in S_2(\H)$, $S_2(\H)$ forms a Hilbert space.

\subsection{Absolute continuity of the spectral measure}

Trace class operators provide some suitable framework to deal with questions concerning absolute continuity of spectral measures. With the following lemma we recall some criterion that was used crucially in \cite{CS16}; see also \cite{Voi79}.

\begin{lemma}\label{lem:absolutely_continuous}
Let $\H$ be a separable Hilbert space. Consider a selfadjoint operator $X\in B(\H)$ and assume that its spectral measure is not Lebesgue absolutely continuous. Then there exists a sequence $(T_n)_{n=1}^\infty$ of finite rank operators on $\H$ having the following properties:
\begin{enumerate}[(i)]
 \item $0 \leq T_n \leq 1$ for all $n\in\N$;
 \item $(T_n)_{n=1}^\infty$ converges weakly to a non-zero spectral projection $p$ of $X$;
 \item $\| [T_n,X] \|_1 \to 0$ as $n\to\infty$.
\end{enumerate}
\end{lemma}

If $X=X^\ast\in B(\H)$ is given, then its spectral measure (i.e., the associated resolution of the identity) is a projection valued measure $E_X$ on the Borel subsets of $\R$ that satisfies
$$X = \int_\R t\, dE_X(t).$$
More precisely, the spectral measure $E_X$ takes values in the von Neumann algebra that is generated by $X$ in $B(\H)$. Clearly, $E_X$ being Lebesgue absolutely continuous means that $E_X(A) = 0$ holds for each Borel subset $A\subset\R$ of Lebesgue measure zero. Thus, if $(\M,\tau)$ is any tracial $W^\ast$-probability space, then the spectral measure $E_X$ of an element $X=X^\ast \in \M$ is Lebesgue absolutely continuous if and only if its analytic distribution $\mu_X = \tau \circ E_X$ is absolutely continuous with respect to the Lebesgue measure on $\R$.

\subsection{Trace formulas}

Throughout the rest of this section, we let $(\M,\tau)$ be a (separable) tracial $W^\ast$-probability space. Via the GNS construction, we obtain $L^2(\M,\tau)$ as the canonical (separable) complex Hilbert space on which $\M$ acts. Let us denote by $J$ Tomita's conjugation operator, i.e., the antilinear operator $J: L^2(\M,\tau) \to L^2(\M,\tau)$ that extends the involution $\ast$ isometrically from $\M$ to $L^2(\M,\tau)$. One easily sees that $J$ satisfies $J=J^\ast=J^{-1}$. Furthermore, we introduce $\Pi_1$ as the orthogonal projection onto the (closed) linear subspace $\C1$ of $L^2(\M,\tau)$; note that $\Pi_1$ is of trace class.

In the following, we denote by $\Tr$ the trace on the trace class operators $S_1(L^2(\M,\tau))$ on the Hilbert space $L^2(\M,\tau)$. Let us recall some formulas that were used in \cite{CS16}.

\begin{lemma}\label{lem:trace_formula}
In the situation described previously, we have
\begin{equation}\label{eq:trace_identity-1} 
\Tr(J X^\ast J \Pi_1 Y) = \tau(XY)
\end{equation}
for all $X,Y\in\M$ and more generally
\begin{equation}\label{eq:trace_identity-2} 
\Tr(J X^\ast J (U \sharp \Pi_1) Y) = \tau(X (U^\sim \sharp Y))
\end{equation}
for all $X\in \M$ and $U\in \M \otimes \M$
\end{lemma}

\begin{proof}
If $(e_i)_{i\in I}$ is any orthonormal basis of the separable Hilbert space $L^2(\M,\tau)$, then
$$\langle J X^\ast J \Pi_1 Y e_i, e_i\rangle = \langle Y e_i, 1\rangle \langle J X^\ast J 1, e_i\rangle = \langle e_i, Y^\ast 1\rangle \langle J X^\ast J 1, e_i\rangle = \langle J X^\ast J 1, \langle Y^\ast 1, e_i\rangle e_i\rangle$$
for all $i\in I$, so that in summary, since $Y^\ast 1 = \sum_{i\in I} \langle Y^\ast 1, e_i\rangle e_i$,
$$\Tr(J X^\ast J \Pi_1 Y) = \sum_{i\in I} \langle J X^\ast J \Pi_1 Y e_i, e_i\rangle = \langle J X^\ast J 1, Y^\ast 1\rangle = \langle X, Y^\ast\rangle = \tau(XY),$$
which is \eqref{eq:trace_identity-1}. For proving \eqref{eq:trace_identity-2}, it clearly suffices to consider an element $U\in \M\otimes\M$ that is of the special form $U=U_1 \otimes U_2$; for such $U$, we may check that
\begin{align*}
\Tr(J X^\ast J (U \sharp \Pi_1) Y) &= \Tr(J X^\ast J (U_1 \Pi_1 U_2) Y)\\
                                   &= \Tr(U_1 J X^\ast J \Pi_1 U_2 Y)\\
                                   &= \Tr(J X^\ast J \Pi_1 (U_2 Y U_1))\\
                                   &= \tau(X (U_2 Y U_1))\\
                                   &= \tau(X (U^\sim \sharp Y)),
\end{align*}
where we used in turn the fact that $J X^\ast J$ commutes with $\M$, the trace property of $\Tr$, and finally the previous formula \eqref{eq:trace_identity-1}. This concludes the proof.
\end{proof}

Now, let any $N\in\N$ be given. We are aiming at an analogue of the previous lemma for $M_N(\M)$.

For that purpose, we represent the von Neumann algebra $M_N(\M)$ on the associated complex Hilbert space $L^2(M_N(\M),\tr_N\circ\tau^{(N)})$. This sounds very natural but is of course not the only option: alternatively, we could represent $M_N(\M)$ on $L^2(\M,\tau)^N$, which would however not have all the needed properties.

Let us denote by $\Tr_N$ the trace on the trace class operators on the separable Hilbert space $L^2(M_N(\M),\tr_N\circ\tau^{(N)})$ and denote by $J_N$ the Tomita operator that extends the involution $\ast$ isometrically from $M_N(\M)$ to $L^2(M_N(\M),\tr_N\circ\tau^{(N)})$. Of course, we could apply Lemma \ref{lem:trace_formula} directly, but the resulting formula would involve $\Pi_{\1_N}$, i.e., the orthogonal projection from $L^2(M_N(\M),\tr_N\circ\tau^{(N)})$ onto its closed linear subspace $\C \1_N$. In contrast, we need a formula involving $\Pi_1 \1_N$ instead.

Note that $(\M,\tau)$ is canonically embedded in $(M_N(\M),\tr_N\circ\tau^{(N)})$. We thus may consider the unique trace preserving conditional expectation $\E_N$ from $M_N(\M)$ to $\M$. Being trace preserving means then explicitly
$$\tr_N \circ \tau^{(N)} = \tau \circ \E_N.$$
Clearly, $\E_N[X] = \frac{1}{N} \sum^N_{k=1} X_{kk}$ for each $X=(X_{kl})_{k,l=1}^N \in M_N(\M)$.

\begin{lemma}\label{lem:trace_formula-matricial}
In the situation described before, we have
$$\Tr_N(J_N X^\ast J_N (U \sharp \Pi_1) Y) = N^2 \tau\big( \E_N[X] \E_N[U^\sim \sharp Y] \big)$$
for $X,Y\in M_N(\M)$ and $U \in M_N(\M \otimes \M)$.
\end{lemma}

\begin{proof}
We choose an orthonormal basis $(e_i)_{i\in I}$ of $L^2(\M,\tau)$. Since the normalized matrix units $(\sqrt{N} e^{k,l})_{(k,l)\in\{1,\dots,N\}^2}$ form an orthonormal basis of $L^2(M_N(\C),\tr_N)$, we may lift the latter to an orthonormal basis $(e^{k,l}_i)_{(k,l,i) \in \{1,\dots,N\}^2 \times I}$ of $L^2(M_M(\M),\tr_N\circ\tau^{(N)})$, where $e^{k,l}_i$ corresponds to $\sqrt{N} e^{kl} \otimes e_i$ under the natural identification of $L^2(M_M(\M),\tr_N\circ\tau^{(N)})$ with the Hilbert space tensor product of $L^2(M_N(\C),\tr_N)$ and $L^2(\M,\tau)$. Then we may compute that
\begin{align*}
J_N X^\ast J_N (U \sharp \Pi_1) Y (e^{kl} \otimes e_i)
&= \sum^N_{r=1} J_N X^\ast J_N (U \sharp \Pi_1) (e^{rl} \otimes (Y_{rk} e_i))\\
&= \sum^N_{q,r=1} J_N X^\ast J_N (e^{ql} \otimes ( (U_{qr} \sharp \Pi_1) Y_{rk} e_i))\\
&= \sum^N_{q,r=1} J_N X^\ast (e^{lq} \otimes ( J (U_{qr} \sharp \Pi_1) Y_{rk} e_i))\\
&= \sum^N_{p,q,r=1} J_N (e^{pq} \otimes (X^\ast_{lp} J (U_{qr} \sharp \Pi_1) Y_{rk} e_i))\\
&= \sum^N_{p,q,r=1} e^{qp} \otimes (J X^\ast_{lp} J (U_{qr} \sharp \Pi_1) Y_{rk} e_i),
\end{align*}
so that
\begin{align*}
\Tr_N(J_N X^\ast J_N (U \sharp \Pi_1) Y)
&= \sum_{k,l=1}^N \sum_{i\in I} \langle J_N X^\ast J_N (U \sharp \Pi_1) Y e^{kl}_i, e^{kl}_i\rangle\\
&= \sum_{k,l=1}^N \sum_{i\in I} \sum^N_{p,q,r=1} N \langle e^{qp}, e^{kl}\rangle  \langle (J X^\ast_{lp} J (U_{qr} \sharp \Pi_1) Y_{rk} e_i, e_i\rangle\\
&= \sum_{k,l=1}^N \sum_{i\in I} \sum^N_{r=1} \langle J X^\ast_{ll} J (U_{kr} \sharp \Pi_1) Y_{rk} e_i, e_i\rangle\\
&= \sum_{k,l=1}^N \sum^N_{r=1} \Tr(J X^\ast_{ll} J (U_{kr} \sharp \Pi_1) Y_{rk}),
\end{align*}
and finally, with the help of Lemma \ref{lem:trace_formula},
\begin{align*}
\Tr_N(J_N X^\ast J_N (U \sharp \Pi_1) Y)
&= \sum_{k,l=1}^N \sum^N_{r=1} \Tr(J X^\ast_{ll} J (U_{kr} \sharp \Pi_1) Y_{rk})\\
&= \sum_{k,l=1}^N \sum^N_{r=1} \tau(X_{ll} U_{kr}^\sim \sharp Y_{rk})\\
&= \sum_{k,l=1}^N \tau(X_{ll} (U^\sim \sharp Y)_{kk})\\
&= N^2 \tau\big( \E_N[X] \E_N[U^\sim \sharp Y] \big),
\end{align*}
as we wished to show.
\end{proof}

\subsection{Dual Systems}

Consider selfadjoint elements $X_1,\dots,X_n\in \M$. We suppose that a \emph{dual system to $(X_1,\dots,X_n)$ in $L^2(\M,\tau)$} exists, i.e., an $n$-tuple $(R_1,\dots,R_n)$ of operators $R_1,\dots,R_n\in B(L^2(\M,\tau))$ such that
$$[R_i,X_j] = \delta_{i,j} \Pi_1 \qquad\text{for all $i,j = 1,\dots,n$}.$$
Note that our definition is taken from \cite{CS16} and thus differs slightly from \cite{Voi98}. More precisely, we have removed the imaginary unit on the right hand side and have flipped the entries of the commutator on the left hand side; accordingly, the operators $R_1,\dots,R_n$ are not selfadjoint like in \cite{Voi98} but satisfy $R_i^\ast = - R_i$ for $i=1,\dots,n$.

It follows from Proposition 5.10 in \cite{Voi98} that the existence of a dual system to $(X_1,\dots,X_n)$ guarantees that $\Phi^\ast(X_1,\dots,X_n) < \infty$. More concretely, the conjugate system $(\xi_1,\dots,\xi_n)$ of $(X_1,\dots,X_n)$ is given by $\xi_j = (R_j - J R_j J)1$ for $j=1,\dots,n$.

If now $P\in \C\langle x_1,\dots,x_n\rangle$ is any noncommutative polynomial, then
$$[R_j,P(X)] = (\partial_j P)(X) \sharp \Pi_1 \qquad\text{for all $j=1,\dots,n$}.$$
Indeed, since $[R_j,\cdot]$ is a derivation on $B(L^2(\M,\tau))$, we get that
$$[R_j,P(X)] = \sum^n_{i=1} (\partial_i P)(X) \sharp [R_j,X_i] = (\partial_j P)(X) \sharp \Pi_1.$$
More generally, if $P\in M_N(\C\langle x_1,\dots,x_n\rangle)$ is given, then
\begin{equation}\label{eq:dual_system-matricial}
[R_j\1_N,P(X)] = (\partial_j^{(N)} P)(X) \sharp \Pi_1 \qquad\text{for all $j=1,\dots,n$}.
\end{equation}
Indeed, if we write $P=(P_{kl})_{k,l=1}^N$, then we see that
$$[R_j\1_N,P(X)] = \big([R_j,P_{kl}(X)]\big)_{k,l=1}^N = \big((\partial_j P_{kl})(X) \sharp \Pi_1\big)_{k,l=1}^N = (\partial^{(N)}_j P)(X) \sharp \Pi_1.$$
A comment on the notation is in order: for any given $T\in B(L^2(\M,\tau))$, we denote by $T \1_N$ the associated ``diagonal operator'' in $M_N(B(L^2(\M,\tau)))$; note that $M_N(B(L^2(\M,\tau)))$ sits like $M_N(\M)$ inside $B(L^2(M_N(\M),\tr_N\circ\tau^{(N)}))$.

\begin{proposition}\label{prop:absolutely_continuous_reduction}
Let $P=P^\ast \in M_N(\C\langle x_1,\dots,x_n\rangle)$ be given and assume that the analytic distribution of the selfadjoint operator $P(X)\in M_N(\M)$ is not Lebesgue absolutely continuous. Then there exists a non-zero projection $p\in \vN(P(X)) \subseteq M_N(\M)$, such that
$$\E_N\big[(\partial_j^{(N)} P)(X)^\sim \sharp (P(X)p)\big] = 0 \qquad\text{for $j=1,\dots,n$}.$$
\end{proposition}

\begin{proof}
The proof proceeds along the same lines as that of Theorem 13 in \cite{CS16}. If we assume that the analytic distribution $\mu_{P(X)}$ is not absolutely continuous with respect to the Lebesgue measure, then Lemma \ref{lem:absolutely_continuous} guarantees, as $M_N(\M)$ is represented in standard form, the existence of a sequence $(T_n)_{n=1}^\infty$ of finite rank operators on $L^2(M_N(\M),\tr_N\circ\tau^{(N)})$, such that $0 \leq T_n \leq 1$ for all $n\in\N$, $T_n \to p$ weakly for some non-zero spectral projection $p$ of $P(X)$, which thus belongs to $\vN(P(X)) \subseteq M_N(\M)$, and $\| [T_n, P(X)] \|_1 \to 0$ as $n\to \infty$.

Let us fix $Z\in M_N(\M)$. Then, for each $j=1,\dots,n$, we may compute that
\begin{align*}
\lefteqn{N^2 \tau\big( \E_N[Z] \E_N\big[ (\partial_j^{(N)} P)(X)^\sim \sharp (P(X)p) \big] \big)}\\
&\qquad = \Tr_N\big( J_N Z^\ast J_N \big( (\partial_j^{(N)} P)(X) \sharp \Pi_1 \big) P(X) p \big) \qquad\text{by Lemma \ref{lem:trace_formula-matricial}}\\
&\qquad = \lim_{n\to \infty} \Tr_N\big( J_N Z^\ast J_N \big( (\partial_j^{(N)} P)(X) \sharp \Pi_1 \big) P(X) T_n \big)\\
&\qquad = \lim_{n\to \infty} \Tr_N\big( J_N Z^\ast J_N [R_j\1_N, P(X)] P(X) T_n \big) \qquad\text{by \eqref{eq:dual_system-matricial}}.
\end{align*}
By the trace property of $\Tr_N$, the simple observation that the commutator $[\cdot,P(X)]$ forms a derivation on $B(L^2(M_N(\M),\tr_N\circ\tau^{(N)}))$, and the fact that both $[J_N Z^\ast J_N, P(X)] = 0$ and $[P(X),P(X)]=0$, we get that
$$\Tr_N\big( J_N Z^\ast J_N [R_j\1_N, P(X)] P(X) T_n \big) = - \Tr_N\big( J_N Z^\ast J_N (R_j\1_N) P(X) [T_n,P(X)] \big)$$
and finally
\begin{align*}
\big|\Tr_N\big( J_N Z^\ast J_N [R_j\1_N, P(X)] P(X) T_n \big)\big|
&\leq \big\| J_N Z^\ast J_N (R_j\1_N) P(X) [T_n,P(X)] \big\|_1\\
&\leq \| J_N Z^\ast J_N (R_j\1_N) P(X) \| \| [T_n,P(X)] \|_1,
\end{align*}
from which it follows that
$$\lim_{n\to \infty} \Tr_N\big( J_N Z^\ast J_N [R_j\1_N, P(X)] P(X) T_n \big) = 0.$$
Thus, in summary, we obtain
$$N^2 \tau\big( \E_N[Z] \E_N\big[ (\partial_j^{(N)} P)(X)^\sim \sharp (P(X)p) \big] \big) = 0.$$
Since $\tau$ is faithful and $Z$ was arbitrary, we conclude that
$$\E_N\big[(\partial_j^{(N)} P)(X)^\sim \sharp (P(X)p)\big] = 0$$
holds for each $j=1,\dots,n$, as desired.
\end{proof}

Now, we are able to provide the announced regularity result.

\begin{theorem}\label{thm:main-4}
Let us suppose the following situation:
\begin{enumerate}[(i)]
 \item $b_1,\dots,b_n$ are selfadjoint matrices in $M_N(\C)$ for which
 $$b_1 x_1 + \dots + b_n x_n \in M_N(\C\langle x_1,\dots,x_n\rangle)$$
 is full over $\C\langle x_1,\dots,x_n\rangle$.
 \item $X_1,\dots,X_n$ are selfadjoint elements in a tracial $W^\ast$-probability space $(\M,\tau)$ to which a dual system exists.
\end{enumerate}
Then the analytic distribution $\mu_\X$ of
$$\X := b_1 X_1 + \dots + b_n X_n,$$
seen as an element in the tracial $W^\ast$-probability space $(M_N(\M),\tr_N \circ \tau^{(N)})$, is absolutely continuous with respect to the Lebesgue measure.
\end{theorem}

\begin{proof}
Consider $P := b_1 x_1 + \dots + b_n x_n \in M_N(\C\langle x_1,\dots,x_n\rangle)$, which by assumption satisfies $P^\ast = P$. Assume to the contrary that the analytic distribution of the selfadjoint operator $\X = P(X) \in M_N(\M)$ would not be absolutely continuous with respect to the Lebesgue measure. Then the previous Proposition \ref{prop:absolutely_continuous_reduction} guarantees that a non-zero projection $p\in \vN(\X) \subseteq M_N(\M)$ exists with the property that
$$\E_N\big[(\partial_j^{(N)} P)(X)^\sim \sharp (P(X)p)\big] = 0 \qquad\text{for $j=1,\dots,n$}.$$
Clearly, $\partial^{(N)}_j P = b_j \odot 1_N$ and thus $(\partial^{(N)}_j P)(X)^\sim = b_j \odot 1_N$ for $j=1,\dots,n$, so we may conclude from the latter that
$$\E_N\big[b_j \X p\big] = 0 \qquad\text{for $j=1,\dots,n$}.$$
Multiplying first with $X_j$ from the left, applying $\tau$ to both side, and finally summing over $j=1,\dots,n$ yields that
$$0 = \sum^n_{j=1} \tau\big(X_j \E_N\big[b_j \X p\big]\big) = \sum^n_{j=1} \tau\big(\E_N[X_j b_j \X p]\big) = \sum^n_{j=1} (\tr_N\circ\tau^{(N)})\big(b_j X_j \X p\big) = (\tr_N\circ\tau^{(N)})\big(\X^2 p\big)$$
and thus, by the faithfulness of $\tr_N\circ\tau^{(N)}$ since $p$ is a projection, finally that $\X p=0$. Because $P$ is assumed to be full over $\C\langle x_1,\dots,x_n\rangle$, the latter contradicts Theorem \ref{thm:main-3} as $p\neq 0$. Thus, the analytic distribution $\mu_\X$ must be absolutely continuous with respect to the Lebesgue measure.
\end{proof}

\section{Hoelder continuity of cumulative distribution functions}
\label{sec:Hoelder_continuity}

In this section, we want to address some further regularity properties of noncommutative random variables of the form
$$\X = b_0 + b_1 X_1 + \dots + b_n X_n,$$
namely the Hoelder continuity of its cumulative distribution function $\F_\X$. Recall that the \emph{cumulative distribution function $\F_\mu$} of a probability measure $\mu$ on $\R$ is the function $\F_\mu: \R \to [0,1]$ that is defined by $\F_\mu(t) := \mu((-\infty,t])$; note that in the case of the analytic distribution $\mu_\X$ of $\X$, we will abbreviate $\F_{\mu_\X}$ by $\F_\X$.

Our considerations are inspired by \cite{CS16}. While their approach works for noncommutative random variables $X_1,\dots,X_n$ that admit a dual system, we can weaken that condition to finite Fisher information. Furthermore, since we deal with operators $\X$ that are matrix-valued but linear in $X_1,\dots,X_n$, we are able to provide some explicit value for the exponent of Hoelder continuity.

Of course, Hoelder continuity implies continuity of $\F_\X$ and thus excludes atoms in the analytic distribution $\mu_\X$. Therefore, in view of our previous result Theorem \ref{thm:main-3}, we can only hope for Hoelder continuity in situations where we impose the additional condition that already the purely linear part $b_1 x_1 + \dots + b_n x_n$ of $b_0 + b_1 x_1 + \dots + b_n x_n$, is full over $\C\langle x_1,\dots,x_n\rangle$. However, it turns out that a stronger version of fullness must be required for that purpose.

The final statement thus reads as follows.

\begin{theorem}\label{thm:main-5}
Let us suppose the following situation:
\begin{enumerate}[(i)]
 \item $b_1,\dots,b_n$ are selfadjoint matrices in $M_N(\C)$ for which the associated quantum operator
 \begin{equation}\label{eq:quantum-operator}
 \cL:\ M_N(\C) \to M_N(\C),\qquad b \mapsto \sum^n_{j=1} b_j b b_j
 \end{equation}
 is \emph{semi-flat} in the sense that there exists some $c>0$ such that the condition $$\cL(b) \geq c \tr_N(b) \1_N$$ is satisfied for all positive semidefinite matrices $b\in M_N(\C)$. Let $b_0$ be any other selfadjoint matrix in $M_N(\C)$.
 \item $X_1,\dots,X_n$ are selfadjoint elements in a tracial $W^\ast$-probability space $(\M,\tau)$ that satisfy $$\Phi^\ast(X_1,\dots,X_n) < \infty.$$
\end{enumerate}
Then the cumulative distribution function $\F_\X$ of
$$\X := b_0 + b_1 X_1 + \dots + b_n X_n,$$
seen as an element in the tracial $W^\ast$-probability space $(M_N(\M),\tr_N \circ \tau^{(N)})$, is Hoelder continuous with exponent $\frac{2}{3}$, i.e., there is some $C>0$ such that
$$|\F_\X(t) - \F_\X(s)| \leq C |t-s|^{\frac{2}{3}} \qquad \text{for all $s,t\in\R$}.$$
More precisely, the above constant $C$ is given by
$$C := 4 c^{-\frac{2}{3}} \bigg(\sum^n_{j=1} \|b_j\|^2\bigg)^{\frac{1}{3}} \Phi^\ast(X_1,\dots,X_n)^{\frac{1}{3}}.$$
\end{theorem}

The proof requires some preparations. First of all, we apply Theorem \ref{thm:Fisher-bound} in order to derive the following proposition.

\begin{proposition}\label{prop:Hoelder-projection}
Let $(\M,\tau)$ be a tracial $W^\ast$-probability space and take any selfadjoint noncommutative random variables $X_1,\dots,X_n$ in $\M$ for which $\Phi^\ast(X_1,\dots,X_n) < \infty$. Denote be $\M_0$ the von Neumann subalgebra of $\M$ that is generated by $X_1,\dots,X_n$ and let $(\xi_1,\dots,\xi_n)$ be the conjugate system for $(X_1,\dots,X_n)$.
Furthermore, let $b_0,b_1,\dots,b_n\in M_N(\C)$ be selfadjoint matrices and consider the associated operator
$$\X = b_0 + b_1 X_1 + \dots + b_n X_n.$$
If $p$ is any projection in $M_N(\M_0)$, then the inequality
\begin{equation}\label{eq:Hoelder-projection-1}
\big|(\tr_N\circ\tau^{(N)})\big(b_j \tau^{(N)}(p) b_j p\big)\big| \leq 8 \|\xi_j\|_2 \|b_j\| \|\X p\|_2
\end{equation}
holds for each $j=1,\dots,n$ and thus, if we denote by $\cL$ the quantum operator \eqref{eq:quantum-operator} associated to $b_1,\dots,b_n$ (note that $b_0$ is not involved intentionally), in particular
\begin{equation}\label{eq:Hoelder-projection-2}
\big|(\tr_N\circ\tau^{(N)})\big(\cL(\tau^{(N)}(p)) p\big)\big| \leq 8 \bigg(\sum^n_{j=1} \|b_j\|^2\bigg)^{\frac{1}{2}} \Phi^\ast(X_1,\dots,X_n)^{\frac{1}{2}} \|\X p\|_2.
\end{equation}
\end{proposition}

\begin{proof}
Fix any $1\leq j \leq n$ and let $p$ be any projection in $M_N(\M_0)$. A direct application of the inequality \eqref{eq:Fisher-bound} that was stated in Theorem \ref{thm:Fisher-bound} yields that
$$\big|\langle p \cdot (\partial_j^{(N)} \X) \cdot p, b_j \odot \1_N \rangle\big| \leq 4 \|\xi_j\|_2 \big(\|\X p\|_2 \|p\| + \|p\| \|\X p\|_2\big) \|b_j\| \|\1_N\|,$$
which simplifies to
$$\big|\langle (pb_j) \odot p, b_j \odot \1_N \rangle\big| \leq 8 \|\xi_j\|_2 \|b_j\| \|\X p\|_2 \|p\|,$$
as $\partial^{(N)} \X = b_j \odot \1_N$, and finally implies
$$\big|\langle (pb_j) \odot p, b_j \odot \1_N \rangle\big| \leq 8 \|\xi_j\|_2 \|b_j\| \|\X p\|_2,$$
as $\|p\| \leq 1$. Now, we observe that
$$\langle (pb_j) \odot p, b_j \odot \1_N \rangle = \langle p b_j \tau^{(N)}(p), b_j \rangle = (\tr_N\circ\tau^{(N)})\big( p b_j \tau^{(N)}(p) b_j \big) = (\tr_N\circ\tau^{(N)})\big( b_j \tau^{(N)}(p) b_j p\big),$$
so that \eqref{eq:Hoelder-projection-1} is derived from the latter. Finally, we may derive \eqref{eq:Hoelder-projection-2} from by summing \eqref{eq:Hoelder-projection-1} over $j=1,\dots,n$ and applying the classical Cauchy-Schwarz inequality. Indeed, we see that
\begin{align*}
\big|(\tr_N\circ\tau^{(N)})\big(\cL(\tau^{(N)}(p)) p\big)\big|
&= \bigg| \sum^n_{j=1}(\tr_N\circ\tau^{(N)})\big( b_j \tau^{(N)}(p) b_j p\big)  \bigg|\\
&\leq \sum^n_{j=1} \big|(\tr_N\circ\tau^{(N)})\big(\cL(\tau^{(N)}(p)) p\big)\big|\\
&\leq 8 \bigg(\sum^n_{j=1} \|\xi_j\|_2 \|b_j\|\bigg) \|\X p\|_2\\
&\leq 8 \bigg(\sum^n_{j=1} \|b_j\|^2 \bigg)^{\frac{1}{2}} \bigg(\sum^n_{j=1} \|\xi_j\|_2^2 \bigg)^{\frac{1}{2}} \|\X p\|_2\\
&= 8 \bigg(\sum^n_{j=1} \|b_j\|^2\bigg)^{\frac{1}{2}} \Phi^\ast(X_1,\dots,X_n)^{\frac{1}{2}} \|\X p\|_2,
\end{align*}
which is \eqref{eq:Hoelder-projection-2}.
\end{proof}

Next, we need the following lemma, which summarizes and slightly extends some of the techniques that were used in \cite{CS16} to control cumulative distribution functions.

\begin{lemma}\label{lem:Hoelder-criterion}
Let $(\M,\tau)$ be a tracial $W^\ast$-probability space and consider any selfadjoint noncommutative random variable $X$ in $\M$. If there exist $c>0$ and $\alpha>1$ such that
\begin{equation}\label{eq:Hoelder-criterion}
c \|(X - s) p\|_2 \geq \|p\|_2^\alpha
\end{equation}
holds for all $s\in\R$ and all spectral projections $p$ of $X$, then, with $\beta := \frac{2}{\alpha-1}$, it holds true that
$$\mu_X\big((s,t]\big) \leq c^\beta (t-s)^\beta \qquad\text{for all $s,t\in\R$ with $s<t$},$$
i.e., the cumulative distribution function of $\mu_X$ is Hoelder continuous with exponent $\beta$.
\end{lemma}

\begin{proof}
Let $E_X$ be the spectral measure of $X$. Fix any $s,t\in\R$ with $s<t$. We consider then the spectral projection $p$ that is given by $p := E_X((s,t])$ and observe that $\|p\|_2 = \mu_X((s,t])^{1/2}$ and
$$\|(X - s) p\|_2^2 = \int_{(s,t]} |x-s|^2\, d\mu_X(x) \leq (t-s)^2 \mu_X((s,t]),$$
so that our assumption \eqref{eq:Hoelder-criterion} enforces that
$$c (t-s) \mu_X\big((s,t]\big)^{\frac{1}{2}} \geq c \|(X - s) p\|_2 \geq \|p\|_2^\alpha = \mu_X\big((s,t]\big)^{\frac{1}{2} \alpha}.$$
Rearranging the latter inequality gives us that
$$\mu_X\big((s,t]\big)^{\frac{1}{2} (\alpha-1)} \leq c (t-s)$$
and finally, after raising both sides to the power $\beta$, which preserves the inequality as $\beta>0$,
$$\mu_X\big((s,t]\big) \leq c^\beta (t-s)^\beta,$$
which is the desired estimate.
\end{proof}

Now, we are prepared to focus on our actual goal.

\begin{proof}[Proof of Theorem \ref{thm:main-5}]
Fix any $s,t\in\R$ with the property that $s<t$. We may apply Proposition \ref{prop:Hoelder-projection} to $\X-s\1_N$, the inequality \eqref{eq:Hoelder-projection-2} gives us that
$$\big|(\tr_N\circ\tau^{(N)})\big(\cL(\tau^{(N)}(p)) p\big)\big| \leq 8 \bigg(\sum^n_{j=1} \|b_j\|^2\bigg)^{\frac{1}{2}} \Phi^\ast(X_1,\dots,X_n)^{\frac{1}{2}} \|(\X -s\1_N) p\|_2$$
for each projection $p\in M_N(\M_0)$, and thus in particular for all spectral projections of $\X$. Now, we use the assumption that $\cL$ is semi-flat; by positivity of $\tr_N\circ\tau^{(N)}$, we may infer that
$$(\tr_N\circ\tau^{(N)})\big(\cL(\tau^{(N)}(p)) p\big) = (\tr_N\circ\tau^{(N)})\big(p \cL(\tau^{(N)}(p)) p\big) \geq c (\tr_N\circ\tau^{(N)})(p)^2 = c \|p\|_2^4.$$
Thus, in summary, we see that
$$\|p\|_2^4 \leq \tilde{c} \|(\X -s\1_N) p\|_2 \qquad\text{with}\qquad \tilde{c} := \frac{8}{c} \bigg(\sum^n_{j=1} \|b_j\|^2\bigg)^{\frac{1}{2}} \Phi^\ast(X_1,\dots,X_n)^{\frac{1}{2}}.$$
With the help of Lemma \ref{lem:Hoelder-criterion}, applied in the case $\alpha=4$ which corresponds to $\beta=\frac{2}{3}$, we conclude now that
$$\mu_\X\big((s,t]\big) \leq C (t-s)^{\frac{2}{3}} \qquad\text{with}\qquad C := \tilde{c}^{\frac{2}{3}} = 4 c^{-\frac{2}{3}} \bigg(\sum^n_{j=1} \|b_j\|^2\bigg)^{\frac{1}{3}} \Phi^\ast(X_1,\dots,X_n)^{\frac{1}{3}}.$$
This proves the assertion.
\end{proof}

An interesting consequence of Theorem \ref{thm:main-5} is given in the following corollary. It relies on bounds derived in \cite{Jam15} for the logarithmic energy for measures having Hoelder continuous cumulative distribution functions; in fact, it was shown in \cite{Jam15} that if a Borel probability measure $\mu$ on $\R$ has a cumulative distribution function $\F_\mu$ that is Hoelder continuous with exponent $\gamma>0$, i.e., it satisfies
$$|\F_\mu(t) - \F_\mu(s)| \leq K |t-s|^\gamma \qquad \text{for all $s,t\in\R$}$$
for some constant $K>0$, then $H^+(\mu) \leq 2 \frac{K}{\gamma}$, where
$$H^+(\mu) := \int_\R \int_\R \log^+\Big(\frac{1}{|s-t|}\Big)\, d\mu(s)\, d\mu(t)$$
with $\log^+(x) := \max\{\log(x),0\}$; consequently, the \emph{logarithmic energy}
$$I(\mu) := \int_\R \int_\R \log |s-t| \, d\mu(s)\, d\mu(t)$$
is then bounded from below by $I(\mu) \geq -2 \frac{K}{\gamma}$.

\begin{corollary}
In the situation of Theorem \ref{thm:main-5}, the noncommutative random variable
$$\X := b_0 + b_1 X_1 + \dots + b_n X_n$$
in the tracial $W^\ast$-probability space $(M_N(\M),\tr_N \circ \tau^{(N)})$ has finite logarithmic energy that can be bounded from below by
$$I(\mu_\X) \geq -12 c^{-\frac{2}{3}} \bigg(\sum^n_{j=1} \|b_j\|^2\bigg)^{\frac{1}{3}} \Phi^\ast(X_1,\dots,X_n)^{\frac{1}{3}}.$$
\end{corollary}

\appendix
\section{Inner rank}
\subsection{\label{sec:inner rank}Inner rank of matrices over noncommutative algebras}

In this subsection, we collect some properties of the inner rank $\rho$ of matrices over a general unital algebra $\A$.

\begin{lemma}
\label{lem:inner rank}Let $A\in M_{m.n}\left(\A\right)$ be given.
\begin{enumerate}
\item $\rho\left(A\right)\leqslant\min\{m,n\}$.
\item $\rho\left(A\right)$ is invariant under multiplication with invertible matrices over $\A$. In particular, the inner rank $\rho\left(A\right)$ doesn't change when multiplied by invertible scalar-valued matrices.
\item Writing $A=BC$ with $B\in M_{m,\rho\left(A\right)}(\A)$ and $C\in M_{\rho\left(A\right),n}\left(\A\right)$, then $B$ and $C$ both are full matrices.
\item If we write $A=(B\ C)$, where $B\in M_{m,r}\left(\A\right)$ and $C\in M_{m,n-r}\left(\A\right)$ for some integer $r<n$, then
\[\rho\left(A\right)\geqslant\max\left(\rho\left(B\right),\rho\left(C\right)\right).\]
\end{enumerate}
\end{lemma}

These facts can be verified directly from the definition.

\begin{lemma}
\label{lem:minor}(See \cite[Lemma 5.4.8(i)]{Coh06}) Let $A\in M_{m,n}\left(\A\right)$ be full with $2\leqslant m\leqslant n$.
\begin{enumerate}
\item $A$ remains full when the first row is omitted.
\item If $A$ does not remain full when the first column is omitted, then there is a factorization
\[A=B\begin{pmatrix}\1 & \0\\\0 & C\end{pmatrix},\]
where $B\in M_{m}\left(\A\right)$ and $C\in M_{m-1,n-1}\left(\A\right)$ both are full.
\end{enumerate}
\end{lemma}

\begin{proof}
To prove the statement $(i)$, we write
\[A=\begin{pmatrix}a\\A'\end{pmatrix},\]
where $a\in M_{1,n}\left(\A\right)$ and $A'\in M_{m-1,n}\left(\A\right)$. Suppose that $A'$ is not full, then there exist matrices $B\in M_{m-1,r}\left(\A\right)$ and $C\in M_{r,n}\left(\A\right)$ with an integer $r<m-1$ such that $A'=BC$. Then we can see that the factorization
\[A=\begin{pmatrix}a\\A'\end{pmatrix}=\begin{pmatrix}a\\BC\end{pmatrix}=\begin{pmatrix}\1 & \0\\\0 & B\end{pmatrix}\begin{pmatrix}a\\C\end{pmatrix}\]
holds for $A$. This implies that $\rho\left(A\right)\leqslant r+1<m$, which is a contradiction to the fullness of $A$.

For the statement $(ii)$, we write $A=(a \  A')$, where $a\in M_{m,1}\left(\A\right)$ and $A'\in M_{m,n-1}\left(\A\right)$. By assumption, $A'$ is not full, so we have a rank factorization $A'=DC$ where $D\in M_{m,r}\left(\A\right)$ and $C\in M_{r,n-1}\left(\A\right)$ are full matrices (by the part (iii) of Lemma \ref{lem:inner rank}). Then we have the factorization
\[A=\begin{pmatrix}a & A'\end{pmatrix}=\begin{pmatrix}a & DC\end{pmatrix}=\begin{pmatrix}a & D\end{pmatrix}\begin{pmatrix}\1 & \0\\\0 & C\end{pmatrix}.\]
Since $A$ is full, we obtain $r+1=\rho\left(A\right)=m$ and thus $B:=(a \  D)$ is a full square matrix of dimension $m$. Hence the above factorization is the desired factorization.
\end{proof}

\begin{definition}
\label{def:diagonal sum}Let $A\in M_m(\A)$ and $B\in M_n(\A)$ be two matrices, we define their \emph{diagonal sum}, denoted by $A\oplus B$, as
\[A\oplus B=\begin{pmatrix}A & \0\\\0 & B\end{pmatrix}\in M_{m+n}(\A).\]
\end{definition}

\begin{theorem}
\label{thm:full minor-appendix}(See \cite[Theorem 5.4.9]{Coh06}) Suppose that the set of all square full matrices over $\A$ is closed under products and diagonal sums. Then for any $A\in M_{m,n}\left(\A\right)$, there exists a square block of $A$ which is a full matrix over $\A$ of dimension $\rho\left(A\right)$. Moreover, $\rho\left(A\right)$ is the maximal dimension for such blocks.
\end{theorem}

\begin{proof}
We prove the existence of such blocks by induction on $m+n$. For $m+n\leqslant3$, it's easy to check directly from the definition. Now suppose that the statement holds for $m+n-1$, we want to prove it for $m+n$. First, we can write $A=BC$, where $B\in M_{m,\rho\left(A\right)}\left(\A\right)$ and $C\in M_{\rho\left(A\right),n}\left(\A\right)$ are two full matrices over $\A$. Since the product of two square full matrices is again full by assumption, it is enough to show that $B$ and $C$ both have a full $\rho\left(A\right)\times\rho\left(A\right)$ block. In the following we prove that $C$ has a full $\rho\left(A\right)\times\rho\left(A\right)$ block; such a block for $B$ can be obtained by a similar argument.

Put $C=(c\ C')$, where $c\in M_{\rho\left(A\right),1}\left(\A\right)$ and $C'\in M_{\rho\left(A\right),n-1}\left(\A\right)$. If $C'$ is full then, by the induction hypothesis, we can conclude that $C'$ has a full $\rho\left(A\right)\times\rho\left(A\right)$ block and so does $C$. Now if $C'$ is not full, then, by part (ii) of Lemma \ref{lem:minor}, there is a factorization
\[C=D\begin{pmatrix}1 & \0\\\0 & E\end{pmatrix},\]
where $D$ is a full $\rho\left(A\right)\times\rho\left(A\right)$ matrix and $E$ is full matrix with one less row and column than $C$. Thus by the induction we can find a full $(\rho\left(A\right)-1)\times(\rho\left(A\right)-1)$ block of $E$. Then by assumption, the diagonal sum of $1$ and this block is also full. And its product with $D$ is full by assumption again, which becomes a full $\rho\left(A\right)\times\rho\left(A\right)$ block of $C$, as desired.

Now we want to show that such a block is maximal. For that purpose, we suppose that $A\in M_{m,n}\left(\A\right)$ has a full block with dimension $r>\rho\left(A\right)$. By applying row and column permutations, we can assume $A$ is of form
\[A=\begin{pmatrix}P & Q\\R & S\end{pmatrix},\]
where $P\in M_{r}\left(\A\right)$ is full and $Q\in M_{r,n-r}\left(\A\right)$, $R\in M_{m-r,r}\left(\A\right)$, $S\in M_{m-r,n-r}\left(\A\right)$. Let
\[A=\begin{pmatrix}B'\\B''\end{pmatrix}\begin{pmatrix}C' & C''\end{pmatrix}\]
be a rank factorization where $B'\in M_{r,\rho\left(A\right)}\left(\A\right)$, $B''\in M_{m-r,\rho\left(A\right)}\left(\A\right)$, $C'\in M_{\rho\left(A\right),r}\left(\A\right)$ and $C''\in M_{\rho\left(A\right),n-r}\left(\A\right)$, then
\[A=\begin{pmatrix}P & Q\\R & S\end{pmatrix}=\begin{pmatrix}B'C' & B'C''\\B''C' & B''C''\end{pmatrix},\]
from which we can see $P=B'C'$. This gives a factorization of $P$ that yields $\rho\left(P\right)\leqslant\rho\left(A\right)<r$, which is a contradiction to the fullness of $P$. Hence we can conclude that $r\leqslant\rho\left(A\right)$ for any full $r\times r$ block of $A$.
\end{proof}

Now we want to present a proof for Proposition \ref{prop:invertible minor}. For that purpose, we first prove the following useful lemma.

\begin{lemma}
\label{lem:full identity}If $\A$ is stably finite, then the identity matrix of each dimension is full.
\end{lemma}

\begin{proof}
Assume that $\rho(\1_{n})=r$ for some $r<n$, then there is a rank factorization
\[\1_{n}=\begin{pmatrix}A\\A'\end{pmatrix}\begin{pmatrix}B & B'\end{pmatrix},\]
where $A,B\in M_{r}\left(\A\right)$, $A'\in M_{n-r,r}\left(\A\right)$, and $B'\in M_{r,n-r}\left(\A\right)$. This yields the equations
\[AB=\1_{r},\ AB'=\0,\ A'B=\0,\ A'B'=\1_{n-r}.\]
Since $\A$ is stably finite, we have $AB=\1_{r}$ implies $BA=\1_{r}$, which means $A$, $B$ are invertible and $A^{-1}=B$. Then the above equations reduce to $A'=\0$, $B'=\0$, and thus
\[\1_{n}=\begin{pmatrix}A\\\0\end{pmatrix}\begin{pmatrix}B & \0\end{pmatrix}=\begin{pmatrix}\1_{r} & \0\\\0 & \0\end{pmatrix}.\]
This is impossible and so we can conclude that $\rho(\1_{n})=n$, i.e., $\1_{n}$ is full.
\end{proof}

\begin{proposition}
\label{prop:invertible minor-appendix}(See \cite[Proposition 5.4.6]{Coh06}) Suppose that $\A$ is stably finite. Let $A\in M_{m+n}\left(\A\right)$ be of the form
\[A=\begin{pmatrix}B & C\\D & E\end{pmatrix},\]
where $B\in M_{m}\left(\A\right)$, $C\in M_{m,n}\left(\A\right)$, $D\in M_{n,m}\left(\A\right)$ and $E\in M_{n}\left(\A\right)$. If $B$ is invertible, then $\rho\left(A\right)\geqslant m$, with equality if and only if $E=DB^{-1}C$.
\end{proposition}

\begin{proof}
It's easy to check that the factorization
\begin{equation}
A=\begin{pmatrix}B & C\\D & E\end{pmatrix}=\begin{pmatrix}B & \0\\D & \1_{n}\end{pmatrix}\begin{pmatrix}1_{m} & \0\\\0 & E-DB^{-1}C\end{pmatrix}\begin{pmatrix}\1_{m} &B^{-1}C\\\0 & \1_{n}\end{pmatrix}\label{factorization-1}
\end{equation}
holds as $B$ is invertible over $\A$. Moreover, from
\[\begin{pmatrix}B & \0\\D & \1_{n}\end{pmatrix}\begin{pmatrix}B^{-1} & \0\\-DB^{-1} & \1_{n}\end{pmatrix}=\1_{m+n}=\begin{pmatrix}\1_{m} & B^{-1}C\\\0 & \1_{n}\end{pmatrix}\begin{pmatrix}\1_{m} & -B^{-1}C\\\0 & \1_{n}\end{pmatrix},\]
we can see that
\[\begin{pmatrix}B & \0\\D & \1_{n}\end{pmatrix}\text{ and }\begin{pmatrix}\1_{m} & B^{-1}C\\\0 & \1_{n}\end{pmatrix}\]
are invertible over $\A$. Since the inner rank doesn't change when multiplied by invertible matrices, we obtain
\[\rho\left(A\right)=\rho\begin{pmatrix}\1_{m} & \0\\\0 & E-DB^{-1}C\end{pmatrix}.\]
From the last item of Lemma \ref{lem:inner rank}, we have
\[\rho\left(A\right)\geqslant\max\left\{ \rho\left(\1_{m}\right),\rho\left(E-DB^{-1}C\right)\right\} \geqslant\rho\left(\1_{m}\right).\]
Since $\A$ is stably finite, by Lemma \ref{lem:full identity}, we conclude that $\rho\left(A\right)\geqslant m$.

Now if we assume that $E=DB^{-1}C$, then from the factorization (\ref{factorization-1}), it is clear that $\rho\left(A\right)=\rho\left(\1_{m}\right)=m$. For the converse, we assume that $\rho\left(A\right)=m$ and we want to prove $E-DB^{-1}C=0$. From the factorization (\ref{factorization-1}), we have
\[\rho\begin{pmatrix}\1_{m} & \0\\\0 & E-DB^{-1}C\end{pmatrix}=\rho\left(A\right)=m\]
and then there exists a rank factorization
\[\begin{pmatrix}\1_{m} & \0\\\0 & E-DB^{-1}C\end{pmatrix}=\begin{pmatrix}R\\R'\end{pmatrix}\begin{pmatrix}S & S'\end{pmatrix},\]
where $R\in M_{m}\left(\A\right)$, $R'\in M_{n,m}\left(\A\right)$, $S\in M_{m}\left(\A\right)$ and $S'\in M_{m,n}\left(\A\right)$. So we obtain $RS=\1_{m}$ and thus $R=S^{-1}$ as $\A$ is stably finite. Then from $RS'=0$ and $R'S=0$ we get $S'=0$ and $R'=0$. Hence the above factorization reduces to
\[\begin{pmatrix}\1_{m} & \0\\\0 & E-DB^{-1}C\end{pmatrix}=\begin{pmatrix}R\\\0\end{pmatrix}\begin{pmatrix}S & \0\end{pmatrix}=\begin{pmatrix}\1_{m} & \0\\\0 & \0\end{pmatrix},\]
and so $E-DB^{-1}C=0$.
\end{proof}

\subsection{\label{sec:full minor}Inner rank of matrices over noncommutative polynomials}

In this subsection, we want to verify that the set of all square full matrices over $\C\left\langle x_{1},\dots,x_{d}\right\rangle $ is closed under products and diagonal sums, so Proposition \ref{prop:full minor} follows from Theorem \ref{thm:full minor}. For that purpose, we need some lemmas first. Given a polynomial $a\in\C\left\langle x_{1},\dots,x_{d}\right\rangle $, we denote its degree by $d(a)$. Similar to the commutative case, we have the following properties:
\begin{itemize}
\item $d(ab)=d(a)+d(b)$,
\item $d(a+b)=\max\{d(a),d(b)\}$,
\end{itemize}
with, in particular, $d(0):=-\infty$. Moreover, we define the degree of a matrix $A=(a_{ij})\in M_{m,n}(\C\left\langle x_{1},\dots,x_{d}\right\rangle )$ as \[d(A)=\max_{1\leqslant i\leqslant m,1\leqslant j\leqslant n}d(a_{ij}).\]

\begin{definition}
Let $a_{1},\dots,a_{n}$ be $n$ polynomials. They are called \emph{right $d$-dependent} if $a_{i}=0$ for some index $i$, or there exists a non-zero vector $(b_{1},\dots,b_{n})$ of polynomials such that
\[d(\sum_{i=1}^{n}a_{i}b_{i})<\max_{1\leqslant i\leqslant n}(d(a_{i})+d(b_{i})).\]
Otherwise, if $a_{1},\dots,a_{n}$ are all non-zero and for any polynomials $b_{1},\cdots,b_{n}$, we have
\[d(\sum_{i=1}^{n}a_{i}b_{i})\geqslant\max_{1\leqslant i\leqslant n}(d(a_{i})+d(b_{i})),\]
then $a_{1},\dots,a_{n}$ are called \emph{right $d$-independent}. Similarly, we can define \emph{left $d$-dependent} and \emph{left $d$-independent} for $a_{1},\dots,a_{n}$.
\end{definition}

Let us remark that if $a_{1},\dots,a_{n}$ are right linear dependent over $\C\left\langle x_{1},\dots,x_{d}\right\rangle $, i.e., if there exists a non-zero vector $(b_{1},\dots,b_{n})$ of polynomials such that $\sum_{i=1}^{n}a_{i}b_{i}=0$, then $a_{1},\dots,a_{n}$ are also right $d$-dependent. This follows from
\[-\infty=d(0)=d(\sum_{i=1}^{n}a_{i}b_{i})<0\leqslant\max_{1\leqslant i\leqslant n}d(b_{i})\leqslant\max_{1\leqslant i\leqslant n}(d(a_{i})+d(b_{i})).\]
This shows that right $d$-independence implies right linear independence over $\C\left\langle x_{1},\dots,x_{d}\right\rangle $. Moreover, similar to the usual notion of linear dependence, the following lemma shows that any tuple of polynomials always can be reduced to a $d$-independent one.

\begin{lemma}
\label{lem:d-independence}(See \cite[Theorem 2.5.1]{Coh06}) Let $a_{1},\dots,a_{n}$ be $n$ polynomials, then there exists an invertible matrix over $\C\left\langle x_{1},\dots,x_{d}\right\rangle $, which reduces $(a_{1},\dots,a_{n})$,
by acting on the right, to a tuple whose non-zero entries are right $d$-independent. Similarly, we can reduce $(a_{1},\dots,a_{n})$ to a tuple whose non-zero entries are left $d$-independent by an invertible matrix over $\C\left\langle x_{1},\dots,x_{d}\right\rangle $ acting on the left.
\end{lemma}

\begin{proof}
In order to prove this lemma, we introduce the notion of left transduction: the \emph{left transduction} for $x_{i}$ is defined as a linear map $L_{i}$ of $\C\left\langle x_{1},\dots,x_{d}\right\rangle $ into itself, which sends any monomial of the form $x_{j_{1}}\cdots x_{j_{s}}x_{i}$ to $x_{j_{1}}\cdots x_{j_{s}}$ and all other monomials to $0$. Furthermore, for a given monomial $a=x_{j_{1}}\cdots x_{j_{s}}$, we define the \emph{left transduction} for this monomial as $L_{a}=L_{j_{1}}\cdots L_{j_{s}}$. For convenience, we also define the \emph{left transduction} for any non-zero number $a$ as the identity map on $\C\left\langle x_{1},\dots,x_{d}\right\rangle $. Then it's not difficult to check that $d(L_{a}(b))\leqslant d(b)-s$ for any $b\in\C\left\langle x_{1},\dots,x_{d}\right\rangle $.

Now, if $a_{1},\dots,a_{n}$ are right $d$-independent, then nothing needs to be proved. So we suppose that $a_{1},\dots,a_{n}$ are right $d$-dependent; then there exists $b_{1},\dots,b_{n}$ such that
\begin{equation}
d(\sum_{i=1}^{n}a_{i}b_{i})<\max_{1\leqslant i\leqslant n}(d(a_{i})+d(b_{i})).\label{eq:d-dependence}
\end{equation}
Denote $k=\max_{1\leqslant i\leqslant n}(d(a_{i})+d(b_{i}))$, then from (\ref{eq:d-dependence}), we see that the monomials of degree $k$ degree in each product $a_{i}b_{i}$ have to cancel with each other in the sum $\sum_{i=1}^{n}a_{i}b_{i}$; therefore, omitting pairs $(a_{i},b_{i})$ with $d(a_{i})+d(b_{i})<k$ does not change the relation (\ref{eq:d-dependence}), so we may assume $d(a_{i})+d(b_{i})=k$ for all $i=1,\dots,n$. Without loss of generality, we may additionally arrange $d(a_{1})\geqslant\cdots\geqslant d(a_{n})$, then we have $0\leqslant d(b_{1})\leqslant\cdots\leqslant d(b_{n})$. In order to find the desired invertible matrix, we set $a$ as the monomial of highest degree $s:=d(b_1)$ in $b_{1}$ and define
\[p_{i}=L_{a}(b_{i}),i=1,\dots,n;\]
then we can write $b_{i}=p_{i}a+b_{i}'$, where $b_{i}'$ consists of the monomials in $b_{i}$ which are cancelled by $L_{a}$. Hence, we have
\[L_{a}(\sum_{i=1}^{n}a_{i}b_{i})=\sum_{i=1}^{n}a_{i}p_{i}+\sum_{i=1}^{n}L_{a}(a_{i}b_{i}'),\]
where each product $a_{i}p_{i}=a_{i}L_{a}(b_{i})$ has degree $k-s$, while each polynomial $L_{a}(a_{i}b_{i}')$ has degree $d(L_{a}(a_{i}b_{i}'))<d(a_{i})$ (because every monomial in $b_{i}'$ does not have $a$ as a factor from the right side). As $d(a_{i})\leqslant d(a_{1})=k-s$ and
\[d(L_{a}(\sum_{i=1}^{n}a_{i}b_{i}))\leqslant d(\sum_{i=1}^{n}a_{i}b_{i})-s<k-s,\]
where the second inequality comes from (\ref{eq:d-dependence}), it follows that
\[d(\sum_{i=1}^{n}a_{i}p_{i})<k-s=d(a_{1}),\]
since $\sum_{i=1}^{n}a_{i}p_{i}$ is the difference of two polynomials whose degree $<k-s.$ Setting
\[P=\begin{pmatrix}p_{1} & 0 & \cdots & 0\\p_{2} & 1 & \cdots & 0\\\vdots & \vdots & \ddots & \vdots\\p_{n} & 0 & \cdots & 1\end{pmatrix}\in M_{n}(\C\left\langle x_{1},\dots,x_{d}\right\rangle ),\]
then it is invertible as $p_{1}=L_{a}(b_{1})\in\C$ is non-zero because $a$ is the monomial of highest degree in $b_{1}$; moreover, the acting of $P$ on $(a_{1},\dots,a_{n})$ reduces the degree of $a_{1}$ but doesn't change the other entries.

We continue this procedure as long as $a_{1},\dots,a_{n}$ are right $d$-dependent; we arrive either at $a_{1},\dots,a_{n}$ which are  are $d$-independent or at $a_{1}=0$. In the case where $a_{1}$ is reduced to $0$, we just repeat the procedure for the remaining terms. Therefore, the assertion can be achieved after finitely many steps.
\end{proof}

\begin{definition}
A matrix $A$ over $\C\left\langle x_{1},\dots,x_{d}\right\rangle $ is called \emph{right regular} if it has no right zero divisor, i.e., if there is a matrix $B$ over $\C\left\langle x_{1},\dots,x_{d}\right\rangle $ such that $AB=\0$, then $B=\0$. Similarly, if a matrix has no left zero divisor, then we call it \emph{left regular}.
\end{definition}

From the definition, if $A$ is not right regular, then we have a matrix $B$ such that $AB=\0$; in particular, we can say that each column of $B$ is also a right zero divisor of $A$. Therefore, we see that $A$ is right regular if and only if there is no non-zero column $b$ over $\C\left\langle x_{1},\dots,x_{d}\right\rangle $ satisfying $Ab=\0$.

\begin{lemma}
\label{lem:regular}(See \cite[Lemma 3.1.1]{Coh06}) If $A\in M_{m,n}(\C\left\langle x_{1},\dots,x_{d}\right\rangle )$ is full, then $A$ is right regular whenever $m\geqslant n$ and left regular whenever $m\leqslant n$.
\end{lemma}

\begin{proof}
Suppose that $m\geqslant n$; we want to prove that any full $m\times n$ matrix $A$ is right regular by induction on $n$. When $n=1$, $A$ is a full column means that there is at least one entry of $A$ which is non-zero; then, as $\C\left\langle x_{1},\dots,x_{d}\right\rangle $ is an integral domain, there is no zero divisor for this non-zero entry and so for $A$.

Now suppose the assertion is proven for $n=k-1$ ($k\leqslant m$), and we consider a full $m\times k$ matrix $A$ such that $Ab=\0$ for some column $b$. If there is one entry of $b$ being zero, then we can obtain $b=\0$ by the assertion for $k-1$: Actually, we may assume that $b=(b' \ 0)^{T}$, namely, has its last entry as zero, then by writing $A=(A'\  a)$, where $A'$ is the block of the first $n-1$ columns of $A$, we have $Ab=A'b'=\0$ and thus $b'=\0$ follows from the assertion for $k-1$. In order to use the assertion for $k-1$, we only need to show that $A'$ is full, which is guaranteed by part (i) (with  exchanged rows and columns) of Lemma \ref{lem:minor}.

Finally, we deal with the case where every entry of $b$ is non-zero. We apply Lemma \ref{lem:d-independence} to $b$, that is, $b$ is either left $d$-independent or can be reduced to a new column who has zero entries. In the case that $b$ is reduced by an invertible matrix $P$ to a column $Pb$ has zeros, then from $(AP^{-1})(Pb)=\0$ we see that $Pb=\0$ by the previous paragraph. So we have $b=\0$ as desired, and then it remains to conside the case that $b$ is left $d$-independent and has no zero entry. In this case, $b$ is also left linear independent over polynomials, which enforces that $A=\0$, which is a contradiction with fullness of $A$.

Therefore, we can conclude that for any matrix $A\in M_{m,n}(\C\left\langle x_{1},\dots,x_{d}\right\rangle)$ with $n\leqslant m$ is right regular. By symmetry the second half of our assertion follows.

\end{proof}

\begin{lemma}
(See \cite[Proposition 3.1.3]{Coh06}) If two full matrices $A\in M_{m,r}(\C\left\langle x_{1},\dots,x_{d}\right\rangle )$ and $B\in M_{r,n}(\C\left\langle x_{1},\dots,x_{d}\right\rangle )$ satisfy $AB=\0$, then $m+n\leqslant r$.
\end{lemma}

\begin{proof}
By the previous lemma, it is clear that $m<r$ and $n<r$. Moreover, as $A$ is not right regular, there exists some non-zero column $b$ such that $Ab=\0$; we can use Lemma \ref{lem:d-independence} to reduce $b$ to a column $(\0\ b')^T$ whose non-zero part $b'$ is $d$-independent; then by the same invertible matrix acting on $A$, we reduce $A$ to the form 
$(A'\ \0)$ as $b'$ is independent over $\C\left\langle x_{1},\dots,x_{d}\right\rangle$. So we may assume that $A$ is of the form $(A'\ \0)$, where $A'\in M_{m,s}(\A)$ is right regular for some $s<r$. Hence the relation $AB=\0$ turns to
\[\begin{pmatrix}A' & \0\end{pmatrix}\begin{pmatrix}B'\\B''\end{pmatrix}=\0;\]
then it follows $A'B'=\0$ and thus $B'=\0$ as $A'$ is right regular. Note that $(A'\ \0)$ and $(\0\ B'')^T$ are full, this implies that $s\geqslant m$ and $r-s\geqslant n$, otherwise they would be hollow. Therefore, we obtain $m+n\leqslant r-s+s=r$.
\end{proof}

\begin{lemma}
\label{lem:product and sum} The product and the diagonal sum of full square matrices over $\C\left\langle x_{1},\dots,x_{d}\right\rangle$ are full.
\end{lemma}

\begin{proof}
We want to prove the assertion by contradiction. Suppose that $A$ and $B$ are two full $n\times n$ matrices such that $AB$ is not full. Then there exists a rank factorization $AB=CD$, where $C\in M_{n,r}(\A)$, $D\in M_{r,n}(\A)$ with $r<n$. Alternatively, we can write
\[\begin{pmatrix}A & C\end{pmatrix}\begin{pmatrix}B\\-D\end{pmatrix}=\0;\]
therefore if $(A\ C)$ and $(B \  (-D))^{T}$ are full, then by the previous lemma, we have $2n\leqslant n+r$, as desired. To see that these two matrices are full, recall part (iv) of Lemma \ref{lem:inner rank}; this gives
\[\rho\begin{pmatrix}A & C\end{pmatrix}\geqslant\max\{\rho(A),\rho(C)\}=n,\]
so $(A\ C)$ is full; by symmetry we can also see that $(B \ D)^{T}$ is full.

Suppose that $A\in M_{m}(\C\left\langle x_{1},\dots,x_{d}\right\rangle )$ and $B\in M_{n}(\C\left\langle x_{1},\dots,x_{d}\right\rangle )$ are full but $A\oplus B$ is not full. Then there exists a rank factorization
\[\begin{pmatrix}A & \0\\\0 & B\end{pmatrix}=\begin{pmatrix}C'\\C''\end{pmatrix}\begin{pmatrix}D' & D''\end{pmatrix},\]
where $C'$, $C''$, $D'$ and $D''$ are of sizes $m\times r$, $n\times r$, $r\times m$ and $r\times n$, respectively, for some $r<m+n$. By comparing both sides of the equation, we obtain $C'D''=\0$; therefore, if $C'$ and $D''$ are full, then we obtain $m+n\leqslant r$ by the previous lemma, as desired. To see that $C'$ is full, we consider the relation $A=C'D'$: if $C'$ is not full, then there exists a rank factorization $C'=FG$ with $F\in M_{m,s}(\A)$, $G\in M_{s,r}(\A)$ and $s<m$; thus, $A=F(GD')$ yields that $\rho(A)<s<m$, contradicting the fullness of $A$. Similarly, we can also prove that $D''$ is full.
\end{proof}

Finally we have the following proposition as a corollary of Theorem \ref{thm:full minor} (or Theorem \ref{thm:full minor-appendix}).

\begin{proposition}
\label{prop:full minor-appendix}Let $A\in M_{n}(\C\left\langle x_{1},\dots,x_{d}\right\rangle )$ be given in the form $A=(a \ A')$, where $a$ is the first column of $A$ and $A'$ is the remaining block. Assume that $A$ is full, then there is a full $(n-1)\times(n-1)$ block in $A'$.
\end{proposition}

\begin{proof}
First, by part (i) (with exchanged rows and columns) of Lemma \ref{lem:minor}, we can see that $A'$ is full. Then the assertion can be deduced from Theorem \ref{thm:full minor} if we can verify its requirements; this can be done by the previous lemma.
\end{proof}

\subsection{\label{sec:linear matrices}Characterizations of full matrices with linear entries}

Theorem \ref{thm:hollow and full} establishes the equivalence between fullness and hollowness for linear matrices over $\C\left\langle x_{1},\dots,x_{d}\right\rangle $ up to invertible matrices over $\C$; for reader's convenience, we are presenting a proof based on Section 5.8 of \cite{Coh06}.

\begin{lemma}
\label{degree control}(See \cite[Lemma 5.8.7]{Coh06}) Let $A\in M_{m,n}(\C\left\langle x_{1},\dots,x_{d}\right\rangle )$ be a matrix with rank factorization $A=BC$, then there exists an invertible matrix $P$ over $\C\left\langle x_{1},\dots,x_{d}\right\rangle $ such that $d(BP)\leqslant d(A)$ and $d(P^{-1}C)\leqslant d(A)$.
\end{lemma}

\begin{proof}
Consider the free algebra of polynomials $\C\left\langle x_{1},\dots,x_{d},y_{1},\dots,y_{m},z_{1},\dots,z_{n}\right\rangle $ with more variables $y_{1},\dots,y_{m}$ and $z_{1},\dots,z_{n}$, which contains $\C\left\langle x_{1},\dots,x_{d}\right\rangle $ as a subalgebra. Writing $A=(a_{ij})$, $B=(b_{ik})\in M_{m,r}(\C\left\langle x_{1},\dots,x_{d}\right\rangle) $, $C=(c_{kj})\in M_{r,n}(\C\left\langle x_{1},\dots,x_{d}\right\rangle) $, we can define $b_{k}'=\sum_{i=1}^{m}y_{i}b_{ik}$ and $c_{k}'=\sum_{j=1}^{n}c_{kj}z_{j}$ for $k=1,\dots,r$, then
\[\sum_{k=1}^{r}b_{k}'c_{k}'=\sum_{k=1}^{r}\sum_{i=1}^{m}\sum_{j=1}^{n}y_{i}b_{ik}c_{kj}z_{j}=\sum_{i=1}^{m}\sum_{j=1}^{n}y_{i}(\sum_{k=1}^{r}b_{ik}c_{kj})z_{j}=\sum_{i=1}^{m}\sum_{j=1}^{n}y_{i}a_{ij}z_{j}.\]
If $b_{1}',\dots,b_{r}'$ are right $d$-independent, then we have
\[d(A)+2=d(\sum_{i=1}^{m}\sum_{j=1}^{n}y_{i}a_{ij}z_{j})=d(\sum_{k=1}^{r}b_{k}'c_{k}')\geqslant\max_{1\leqslant k\leqslant r}(d(b_{k}')+d(c_{k}'));\]
combining this with
\[d(b_{k}')=\max_{1\leqslant i\leqslant m}b_{ik}+1,\ d(c_{k}')=\max_{1\leqslant j\leqslant n}c_{kj}+1\]
for $k=1,\dots,r$, implies that
\[d(A)\geqslant\max_{1\leqslant k\leqslant r}(\max_{1\leqslant i\leqslant m}b_{ik}+\max_{1\leqslant j\leqslant n}c_{kj}).\]
So we see that
\[d(A)\geqslant\max_{1\leqslant k\leqslant r}\max_{1\leqslant i\leqslant m}b_{ik}=d(B)\]
and $d(A)\geqslant d(C)$ similarly.

Therefore, it remains to consider the case when $b_{1}',\dots,b_{r}'$ are not right $d$-independent. In this situation, we apply Lemma \ref{lem:d-independence} to $b_{1}',\dots,b_{r}'$, then there exists an invertible matrix $P$ reducing $(b_{1}',\dots,b_{r}')$ to a tuple whose non-zero entries are $d$-independent. First, from the proof of Lemma \ref{lem:d-independence}, we can see in the construction of this invertible matrix $P$ that $P$ only contains polynomials in $x_{1},\dots,x_{d}$, since each $c_{k}'$ has at least a monomial containing some variable $z_{j}$ (otherwise, $c_{k}'=0$ and thus $C$ has a zero column, which is impossible by the fullness of $C$). Secondly, we can show that $b_{1}',\dots,b_{r}'$ are right linear independent over $\C\left\langle x_{1},\dots,x_{d}\right\rangle $: if this is not the case, then there exist polynomials $f_{1},\dots,f_{r}\in \C\left\langle x_{1},\dots,x_{d}\right\rangle $ such that
\[0=\sum_{k=1}^{r}b_{k}'f_{k}=\sum_{i=1}^{m}y_{i}(\sum_{k=1}^{r}b_{ik}f_{k}),\]
which enforces that $\sum_{k=1}^{r}b_{ik}f_{k}=0$ for all $i=1,\dots,m$, i.e., $B$ is not right regular over $\C\left\langle x_{1},\dots,x_{d}\right\rangle$; but $B$ is right regular by Lemma \ref{lem:regular} since $B$ is full. As $b_{1}',\dots,b_{r}'$ are right linear independent, the right acting of $P$ on them can not reduce any of them to zero. Hence we are the back to the situation that $b_{1}',\dots,b_{r}'$ are right $d$-independent as previous.
\end{proof}

\begin{definition}
Let $A=A_{0}+A_{1}x_{1}+\cdots+A_{d}x_{d}\in M_{m,n}(\C\left\langle x_{1},\dots,x_{d}\right\rangle )$ with $m\geqslant n$. $A$ is called \emph{left monic} if there are matrices $B_{1},\dots,B_{d}\in M_{n,m}\left(\C\right)$ such that
\[\sum_{i=1}^{d}B_{i}A_{i}=\1_{n}.\]
That is, the coefficients of homogeneous terms in $A$ form a left invertible matrix $(A_{1}\ \dots\  A_{d})^{T}\in M_{md,n}\left(\C\right)$. Similarly, we say a linear matrix $A\in M_{m,n}(\C\left\langle x_{1},\dots,x_{d}\right\rangle )$ with $m\leqslant n$ is \emph{right monic} if $(A_{1}\ \dots\  A_{d})\in M_{m,nd}\left(\C\right)$ has a right inverse.
\end{definition}

\begin{lemma}
\label{lem:monic}(See \cite[Corollary 5.8.4]{Coh06}) If $A\in M_{m,n}(\C\left\langle x_{1},\dots,x_{d}\right\rangle )$ is a linear full matrix with $m\geqslant n$, then there exist an invertible matrix $U\in M_{m}(\C)$ and an invertible linear matrix $P\in M_{n}(\C\left\langle x_{1},\dots,x_{d}\right\rangle )$ such that
\[UAP=\begin{pmatrix}B & \0\\\0 & \1_{s}\end{pmatrix},\]
where $B\in M_{m-s,n-s}(\C\left\langle x_{1},\dots,x_{d}\right\rangle )$ is left monic for some $s=0,1,\dots,n$.
\end{lemma}

\begin{proof}
Let $A=A_{0}+A_{1}x_{1}+\cdots+A_{d}x_{d}$ be a full matrix. Suppose that $A$ is not left monic, then the scalar-valued matrix $(A_{1}\ \dots\  A_{d})^{T}$ has no left inverse. This means that its rank is less than $n$, and thus we can find a column operation to eliminate its last column. Now if it happens that the last column of $A_{0}$ is also zero after the same column operation, then $A$ becomes hollow after this operation. This implies that $A$ is not full, which is a contradiction to our assumption. Hence the last column of $A_{0}$ cannot be zero and then by some row operations over $\C$ we can turn the last column $A_{0}$ into the form $(\0,1)^{T}$, namely, the last column has its last entry as $1$ and all other entries as $0$. Then by further column operations over linear polynomials, we reduce $A$ into the form
\[\begin{pmatrix}B & \0\\\0 & 1\end{pmatrix},\]
where $B$ is a linear matrix in $M_{m-1,n-1}(\C\left\langle x_{1},\dots,x_{d}\right\rangle )$. Therefore, we can continue this procedure if $B$ is still not left monic and we will reach either a left monic $B$ or a vanishing $B$.
\end{proof}

\begin{theorem}
\label{thm:hollow and full-appendix}(See \cite[Corollary 6.3.6]{Coh95} or \cite[Theorem 5.8.8]{Coh06}) Let $A$ be a linear matrix over $\C\left\langle x_{1},\dots,x_{d}\right\rangle$ that is not full. Then there exist invertible matrices $U,V\in M_{n}\left(\C\right)$ such that $UAV$ is hollow.
\end{theorem}

\begin{proof}
Let $A=BC$ be a rank factorization, where $B$ and $C$ are two full matrices. According to Lemma \ref{degree control}, we may assume that $B$ and $C$ are also linear. Denote $r=\rho(A)$. By applying Lemma \ref{lem:monic} to $B$, there exist invertible matrices $U\in M_{n}\left(\C\right)$ and $P\in M_{r}(\C\left\langle x_{1},\dots,x_{d}\right\rangle )$ such that
\[UBP=\begin{pmatrix}B' & \0\\\0 & \1_{s}\end{pmatrix},\]
where $B'\in M_{n-s,r-s}(\C\left\langle x_{1},\dots,x_{d}\right\rangle )$ is left monic for some $s=0,\cdots,r$. Writing
\[P^{-1}C=\begin{pmatrix}C'\\C''\end{pmatrix}\]
and
\begin{align*}
B' & =B_{0}'+B_{1}'x_{1}+\cdots+B_{d}'x_{d},\\
C' & =C_{0}'+C_{1}'x_{1}+\dots+C_{d}'x_{d}+C^{\lambda}\in M_{r-s,n}(\C\left\langle x_{1},\dots,x_{d}\right\rangle ),
\end{align*}
where $C^{\lambda}$ contains the terms in $C'$ of degree higher than $1$, then by comparing the two sides of the equation
\[UA=(UBP)(P^{-1}C)=\begin{pmatrix}B' & \0\\\0 & \1_{s}\end{pmatrix}\begin{pmatrix}C'\\C''\end{pmatrix}=\begin{pmatrix}B'C'\\C''\end{pmatrix},\]
we get $d(B'C')\leqslant1$ and $d(C'')\leqslant1$. Hence the coefficients of those terms in $B'C'$ whose degree is larger than $1$ are all zero, that is,
\[B_{i}'C_{j}'=0,\ i,j=1,\dots,d\text{ and }C^{\lambda}=0.\]
Since $B'$ is left monic, it follows that $C_{j}'=0$ for $j=1,\dots,d$ and thus $C'=C_{0}'\in M_{r-s,n}(\C)$. By Lemma \ref{lem:minor}, $C'$ is full and so by column operations over $\C$ we can reduce $C'$ to the form $(\1_{r-s}\ \0)$. Therefore, finally we can find an invertible matrix $V\in M_{n}(\C)$ such that $P^{-1}CV$ is of the form
\[\begin{pmatrix}\1_{r-s} & \0\\C_{1}'' & C_{2}''\end{pmatrix}\]
and we have
\[UAV=\begin{pmatrix}B' & \0\\\0 & \1_{s}\end{pmatrix}\begin{pmatrix}\1_{r-s} & \0\\C_{1}'' & C_{2}''\end{pmatrix}=\begin{pmatrix}B' & \0\\C_{1}'' & C_{2}''\end{pmatrix},\]
where the zero block has size $(n-s)\times(n-r+s)$, as desired.
\end{proof}

\bibliographystyle{amsalpha}
\bibliography{rational_regularity}

\end{document}